\documentclass[11pt]{amsart}
\usepackage{amscd,amsmath,amssymb,amsfonts}
\usepackage{url,enumerate,color}
\usepackage{dsfont}
\usepackage[cmtip, all]{xy}

\usepackage[top=30mm,right=30mm,bottom=30mm,left=20mm]{geometry}

\theoremstyle{plain}
\newtheorem{thm}{Theorem}
\newtheorem{lem}[thm]{Lemma}
\newtheorem{cor}[thm]{Corollary}
\newtheorem{prop}[thm]{Proposition}
\theoremstyle{definition}

\newtheorem{rmk}[thm]{Remark}

\numberwithin{thm}{section}
\numberwithin{equation}{thm}

\newcommand{\al}{\alpha}

\newcommand{\inj}{\hookrightarrow}

\newcommand{\End}{{\rm End}}

\newcommand{\Char}{{\rm char}}
\newcommand{\Tr}{{\rm Tr}}
\newcommand{\Gal}{{\rm Gal}}
\newcommand{\Min}{{\rm Min}}
\newcommand{\Max}{{\rm Max}}
\newcommand{\Syl}{{\rm Syl}}

\newcommand{\Trace}{{\rm Trace}}
\newcommand{\Stab}{{\rm Stab}}
\newcommand{\Ker}{{\rm Ker}}

\newcommand{\Aut}{\mathrm{Aut}}

\newcommand{\Irr}{\mathrm{Irr}}

\newcommand{\eps}{\epsilon}
\newcommand{\gam}{\gamma}

\newcommand{\tG}{\tilde{G}}

\newcommand{\diag}{\mathrm{diag}}
\newcommand{\ord}{\mathrm{ord}}

\newcommand{\SL}{\mathrm{SL}}

\newcommand{\GL}{\mathrm{GL}}

\newcommand{\GU}{\mathrm{GU}}
\newcommand{\SU}{\mathrm{SU}}
\newcommand{\PSU}{\mathrm{PSU}}
\newcommand{\PGU}{\mathrm{PGU}}
\newcommand{\Sp}{\mathrm{Sp}}
\newcommand{\PSp}{\mathrm{PSp}}

\newcommand{\sF}{{\mathcal F}}
\newcommand{\sG}{{\mathcal G}}
\newcommand{\sH}{{\mathcal H}}

\newcommand{\sK}{{\mathcal K}}
\newcommand{\sL}{{\mathcal L}}

\newcommand{\sW}{{\mathcal W}}
\newcommand{\sX}{{\mathcal X}}

\newcommand{\A}{{\mathbb A}}

\newcommand{\C}{{\mathbb C}}
{

\newcommand{\F}{{\mathbb F}}
\newcommand{\G}{{\mathbb G}}

\newcommand{\K}{{\mathbb K}}

\renewcommand{\P}{{\mathbb P}}
\newcommand{\Q}{{\mathbb Q}}

\newcommand{\Z}{{\mathbb Z}}

\newcommand{\ZB}{{\mathbf Z}}
\newcommand{\CB}{{\mathbf C}}
\newcommand{\NB}{{\mathbf N}}
\newcommand{\OB}{{\mathbf O}}
\newcommand{\ari}{\mathrm{arith}}
\newcommand{\geo}{\mathrm{geom}}

\newcommand{\bj}{\boldsymbol{j}}
\newcommand{\bt}{\boldsymbol{t}}
\newcommand{\Id}{\mathrm{Id}}
\newcommand{\stW}{\widetilde{\mathcal{W}}}
\newcommand{\stH}{\widetilde{\mathcal{H}}}
\newcommand{\tC}{\tilde{C}}
\newcommand{\tH}{\tilde{H}}
\newcommand{\Gauss}{{\mathsf {Gauss}}} 
\newcommand{\triv}{{\mathds{1}}}

\renewcommand{\mod}{\bmod \,}
\renewcommand{\Char}{{\sf Char}}

\begin{document}
\title[Exponential sums and symplectic and unitary groups]
{Exponential sums and total Weil representations of finite symplectic and unitary groups}
\author{Nicholas M. Katz and Pham Huu Tiep}
\address{Department of Mathematics, Princeton University, Princeton, NJ 08544}
\email{nmk@math.princeton.edu}
\address{Department of Mathematics, Rutgers University, Piscataway, NJ 08854}
\email{tiep@math.rutgers.edu}
\thanks{The second author gratefully acknowledges the support of the NSF (grant DMS-1840702), and the Joshua 
Barlaz Chair in Mathematics.}

\maketitle

\begin{abstract}
We construct explicit local systems on the affine line in characteristic $p>2$, whose geometric monodromy groups are the finite symplectic groups $\Sp_{2n}(q)$ for all $n \ge 2$, and others whose geometric monodromy groups are the special unitary groups $\SU_n(q)$ for all  odd $n \ge 3$, and $q$ any power of $p$, in their total Weil representations.  One principal merit of these local systems is that their associated trace functions are one-parameter 
families of exponential sums of a very simple, i.e.,  easy to remember, form. We also exhibit 
hypergeometric sheaves on $\G_m$, whose geometric monodromy groups are the finite symplectic groups
$\Sp_{2n}(q)$ for any $n \ge 2$, and others whose geometric monodromy groups are
the finite general unitary groups $\GU_n(q)$ for any odd $n \geq 3$.
\end{abstract}

\tableofcontents

\section{Introduction}
Throughout this paper, $p$ is an {\bf odd} prime and $q=p^f$ is a (strictly positive) power of $p$.
In our previous paper \cite{KT3}, we exhibited explicit local systems on the affine line $\A^1/\F_p$ whose geometric monodromy groups were  the symplectic groups $\Sp_{2n}(q)$ for all {\bf even} $n \ge 2$, or the special unitary groups $\SU_n(q)$ for all {\bf odd} $n \ge 3$, in their total Weil representations. In this paper, we give new local systems which do this, and which also handle the case of  
$\Sp_{2n}(q)$ for $n$ {\bf odd}. Moreover, our results lead to hypergeometric sheaves whose geometric monodromy groups are 
$\Sp_{2n}(q)$ for any $n \geq 2$, and the general unitary groups $\GU_n(q)$ for any odd $n \geq 3$. This paper may also be viewed as a companion piece to \cite{KT4}, which determines which almost quasisimple groups can possibly occur as monodromy groups of
hypergeometric sheaves.

All of our local systems on $\A^1$ are those attached to one-parameter families of exponential sums of the following simple shape. We fix a nontrivial additive character $\psi$ of $\F_p$, and for each finite extension $k/\F_p$, we obtain the additive character $\psi_k$ of $k$ by composition with $\Trace_{k/\F_p}$. For fixed positive integers $N>M$ with $p \nmid NM$, we look at the one-parameter family of the shape
$$t \in k \mapsto (1/\Gauss_k)\sum_{x \in k}\psi_k(x^N +tx^M),$$
with $\Gauss_k$ a (correctly chosen) quadratic Gauss sum over $k$.

We first give some general results about local systems of this $(N,M)$ type. We then specialize to the cases where
$$N=q^n+1, M=q^m+1, n > m>0,$$
which, under suitable hypotheses, we show realize various total Weil representations. In hindsight, our earlier paper \cite{KT3} was devoted to the special case $m=1$. Despite the apparent simplicity of these
local systems, analysis of them depends heavily on their relation to hypergeometric sheaves, and on a great deal of finite group theory. The finite group theory is used to ``go-up'' from known one-parameter local systems to multi-parameter local systems, and 
then to ``go-down"  to  our target one-parameter local systems. This technology of ``going up" and ``going down" also turns out
to be a crucial ingredient in our paper \cite{KT5}.

Our main results for finite symplectic groups $\Sp_{2n}(q)$ are Theorems \ref{main-sp3}, \ref{main-sp4}, and 
\ref{sp-center}. In Theorem \ref{main-sp3}, we show that certain local systems on $\A^1/\F_p$ have as their geometric monodromy groups the image of $\Sp_{2n}(q)$ in its total Weil representation of degree $q^n$ and whose trace functions are easy to remember  one-parameter
families of exponential  sums. In Theorem  \ref{main-sp4} and Theorem \ref{sp-center} we show that certain hypergeometric sheaves on $\G_m/\F_q$ have geometric monodromy groups which are the images of $\Sp_{2n}(q)$ in 
its irreducible Weil representations of degree $(q^n \pm 1)/2$. The structure of  the arithmetic monodromy groups is also determined completely.
We  obtain similar results for the finite unitary groups, see Theorems \ref{main-su1}--\ref{main-su4}.

\section{A miscellany on moments, irreducibility, and van der Geer--van der Vlugt}
Let us recall the basic mechanism.

Let $p$ be a prime, $k/\F_p$ a finite extension, $U/k$ smooth, and geometrically connected of dimension $\dim(U) > 0$, $\ell$ a prime $\ell \neq p$, and $\sF$ a lisse $\overline{\Q_\ell}$ sheaf on $U$ which is $\iota$-pure of weight zero for a given $\iota:\overline{\Q_\ell} \subset \C$. By purity, one knows \cite[3.4.1(iii)]{De-Weil II} that $\sF$ is geometrically semisimple, say on $U_{\overline{k}}:=U\otimes_k \overline{k}$ we have
$$\sF \cong \bigoplus_i n_i\sG_i,$$
with lisse $\sG_i$ on $U\otimes_k\overline{k}$ which are geometrically irreducible and pairwise non-isomorphic.
\begin{prop}\label{momentcalc}One has
$$\sum_i (n_i)^2 = \limsup_{{\rm finite  \ extensions\ }E/k}(1/\#E)^{\dim(U)}\sum_{x \in U(E)}|\Trace(Frob_{x,E}|\sF)|^2.$$
\end{prop}
\begin{proof}We have 
$$\sum_i (n_i)^2={\rm dim}{\rm End}_{U_{\overline{k}}}(\sF)) ={\rm dim}H^{2\dim(U)}_c(U_{\overline{k}}, \End(\sF)(\dim(U))).$$
This cohomology group is pure of weight zero, say with Frobenius$_k$ eigenvalues $\alpha_j$, $j =1,\ldots,d$, each unitary, for $d$ the dimension of this $H^{2\dim(U)}_c$. By the Lefschetz trace formula, for each finite extension $E/k$, we have
$$\begin{aligned}& \Trace(Frob_E|H^{2\dim(U)}_c(U_{\overline{k}}, \End(\sF)(\dim(U)))\\ 
+ & \sum_{i < 2\dim(U)}(-1)^i\Trace(Frob_E|H^i_c(U_{\overline{k}}, \End(\sF)(\dim(U)))\\
= & \sum_{x \in U(E)}\Trace(Frob_{x,E}|\End(\sF))/(\#E)^{\dim(U)}.\end{aligned}$$
The $H^i_c$ traces for $i < 2\dim(U)$ are $O(1/\sqrt{E})$ (because the coefficients are pure of weight $-2\dim(U)$, so each $H^i_c$ is mixed of weight $\le i - 2\dim(U)\le -1$),
while the  $H^{2\dim(U)}_c$ trace is
$$\sum_j(\alpha_j)^{\deg(E/k)}.$$
As the $\alpha_j$ are each unitary,  the $H^{2\dim(U)}_c$ trace is always $\le d$ in absolute value, but comes arbitrarily close to $d$ for
some infinite sequence of $E/k$ with suitably chosen degrees. The lower $H^i_c$ traces do not affect the $\limsup$, as they tend to $0$ as the degree grows.
\end{proof}

We will refer to the quantity $\sum_i (n_i)^2$ in Proposition \ref{momentcalc} above as the {\it second moment} $M_2(\sF)$.

\begin{lem}\label{ABmoment1}Given strictly positive integers $A \neq B$ which are both prime to $p$, a finite extension $k/\F_p$, a nontrivial additive character $\psi$ of $k$, invertible scalars $a,b \in k^\times$, and a choice $\Gauss_k$ in $\overline{\Q_\ell}$ with absolute value $\sqrt{\#k}$, consider the lisse $\overline{\Q_\ell}$-sheaf $\sF$ on $\G_m/k$ whose trace function is given by
$$t \in E/k \mapsto \frac{-1}{(\Gauss_k)^{\deg(E/k)}}\sum_{x \in E}\psi_E(ax^A +btx^B).$$
Then we have the following results.
\begin{itemize}
\item[(i)] We have $$M_2(\sF)=\gcd(A,B).$$
\item[(ii)] Let us denote by $D:=\gcd(A,B)$. Over the extension field $k(\mu_D)$, for each multiplicative character $\chi$ of order dividing $D$ we have the lisse sheaf $\sF_\chi$ whose trace function is given by
$$t \in E/k(\mu_D) \mapsto \frac{-1}{(\Gauss_k)^{\deg(E/k)}}\sum_{x \in E}\psi_E(ax^{A/D} +btx^{B/D})\chi(x).$$
Then geometrically we have
$$\sF \cong \bigoplus_{\chi \in \Char(D)}\sF_\chi,$$
each $\sF_\chi$ is geometrically irreducible, and the various $\sF_\chi$ are pairwise not geometrically isomorphic.
\end{itemize}
\end{lem}

In fact, this is a special case of the following slightly more general statement. Recall that a one-variable polynomial $f(x)$ over an $\F_p$-algebra is said to be {\it Artin-Schreier reduced} if it is the zero polynomial, or if it has no constant term, and if any monomial appearing with a nonzero coefficient has degree prime to $p$. Given an Artin-Schreier reduced polynomial $f(x)$, we denote by
$${\rm gcd}_{\deg}(f)$$
the greatest common divisor of the degrees of the monomials appearing in $f$.
 
\begin{prop}\label{ABmoment}Given strictly positive integers $A \neq B$ which are both prime to $p$, a finite extension $k/\F_p$, a nontrivial additive character $\psi$ of $k$, an Artin-Schreier reduced polynomial $f(x) \in k[x]$ of degree $A$, and a choice $\Gauss_k$ in $\overline{\Q_\ell}$ with absolute value $\sqrt{\#k}$, consider the lisse $\overline{\Q_\ell}$-sheaf $\sF$ on $\G_m/k$ whose trace function is given by
$$t \in E/k \mapsto \frac{-1}{(\Gauss_k)^{\deg(E/k)}}\sum_{x \in E}\psi_E(f(x) +tx^B).$$
Then we have the following results.
\begin{itemize}
\item[(i)] We have $$M_2(\sF)=\gcd({\rm gcd}_{\deg}(f),B).$$

\item[(ii)]Let us denote by $D:=\gcd({\rm gcd}_{\deg}(f),B)$. Then $f(x)$ is of the form $g(x^D)$ for a unique polynomial $g(x) \in k[x]$. Over the extension field $k(\mu_D)$, for each multiplicative character $\chi$ of order dividing $D$ we have the lisse sheaf $\sF_\chi$ whose trace function is is given by
$$t \in E/k(\mu_D) \mapsto \frac{-1}{(\Gauss_k)^{\deg(E/k)}}\sum_{x \in E}\psi_E(g(x) +tx^{B/D})\chi(x).$$
Then geometrically we have
$$\sF \cong \bigoplus_{\chi \in \Char(D)}\sF_\chi,$$
each $\sF_\chi$ is geometrically irreducible, and the various $\sF_\chi$ are pairwise not geometrically isomorphic. The rank of $\sF_{\triv}$ is $\Max(A/D,B/D) -1$, the rank of each $\sF_\chi$ with $\chi \neq \triv$ is $\Max(A/D,B/D)$.
\end{itemize}
\end{prop}

\begin{rmk}When $A >B$, the sheaves $\sF$ and $\sF_\chi$ are all lisse on $\A^1$, not ``just" on $\G_m$, and this fact would
slightly simplify the proof in this case. 
\end{rmk}

\begin{proof}We first calculate 
$$(1/\#E)\sum_{t \in E}|\Trace(Frob_{t,E}|\sF)|^2$$
for a single $E/k$, large enough to contain the $B^{\mathrm {th}}$ roots of unity. It is
$$(1/\#E)^2 \sum_{t \in E^\times, x,y \in E}\psi_E(f(x)-f(y) +t(x^B - y^B))=$$
$$=(1/\#E)^2 \sum_{x,y \in E}\psi_E(f(x)-f(y))\sum_{t \in E^\times}\psi_E(t(x^B-y^B)).$$

We now rewrite the sum so that the sum over $t$ is over all $t \in E$. It becomes
$$-(1/\#E)^2 \sum_{x,y \in E}\psi_E(f(x)-f(y)) + (1/\#E)^2 \sum_{x,y \in E}\psi_E(f(x)-f(y))\sum_{t \in E}\psi_E(t(x^B-y^B)).$$
We claim that the first term
$$-(1/\#E)^2 \sum_{x,y \in E}\psi_E(f(x)-f(y))$$
is $O(1/\#E)$. Indeed, it is minus the square absolute value of  
$$(1/\#E)\sum_{x \in E}\psi_E(f(x)),$$
which is $O(1/\sqrt{\#E})$ (because $f$ has degree prime to $p$).

So it is only the second term which affects the $\limsup$. That term is.
$$(1/\#E)\sum_{x,y \in E,\,x^B=y^B}\psi_E(f(x)-f(y)).$$
The domain of summation is the union of the lines $y=\zeta x$, one for each $\zeta \in \mu_B$. They all intersect in $x=y=0$, but otherwise are pairwise disjoint. So up to an error of at most $B/\#E$, this sum is
$$\sum_{\zeta \in \mu_B}(1/\#E)\sum_{x \in E}\psi_E(f(x) -f(\zeta x)).$$
Because $f$ is Artin-Schreier reduced, so also is $f(x) -f(\zeta x)$. If $f(x) -f(\zeta x)$ is nonzero, then its degree is prime to $p$, 
and by Weil the sum 
$$(1/\#E)\sum_{x \in E}\psi_E(f(x) -f(\zeta x))$$
has absolute value $O(1/\sqrt{\#E})$. If $f(x) -f(\zeta x)=0$, then this sum is $1$. Thus up to an $O(1/\sqrt{\#E})$ error, the sum is the number of $\zeta \in \mu_B$ for which $f(x)=f(\zeta x)$, an equality which holds precisely for $\zeta$ a root of unity of order
dividing ${\rm gcd}_{\deg}(f)$. This proves the first assertion.

Once we have (i), we write the  trace function of $\sF$ as
$$t \in E/k \mapsto (-1/(\Gauss_k)^{\deg(E/k)})\sum_{x \in E}\psi_E(g(x^D) +tx^{D(B/D)}.$$
If $E$ contains $\mu_D$, this is the sum over $\chi \in \Char(D)$ of the functions
$$t \in E/k \mapsto (-1/(\Gauss_k)^{\deg(E/k)})\sum_{x \in E}\psi_E(g(x) +tx^{B/D}).\chi(x),$$
each of which is the trace function of a lisse $\sF_\chi$ which is pure of weight zero and lisse of rank $A/D=\deg(g)$ for $\chi =\triv$, and
of rank $\deg(g)-1$. [Notice that $\deg(g) \ge 2$, because $\Max(A,B)/D > \Min(A,B)/D \ge 1$, so each $\sF_\chi$ is nonzero.]
Once we have $\sF$ having second moment $D$ expressed geometrically as the sum of $D$ nonzero summands, each summand must irreducible (otherwise we get even more summands) and the $D$ summands must be pairwise nonisomorphic (for if $\sum_i n_i=D$ and 
$\sum_i (n_i)^2=D$, then each $n_i=1$).
\end{proof}

Here is a slight generalization of this last result, where we allow a multiplicative character to ``decorate" the sum in question. The proof,
a straightforward rewriting of the proof of Proposition \ref{ABmoment}, is left to the reader.

\begin{prop}\label{ABmomentrho}Given strictly positive integers $A \neq B$ which are both prime to $p$, a finite extension $k/\F_p$, a nontrivial additive character $\psi$ of $k$, an Artin-Schreier reduced polynomial $f(x) \in k[x]$ of degree $A$, a nontrivial multiplicative character $\rho$ of $k^\times$, and a choice $\Gauss_k$ in $\overline{\Q_\ell}$ with absolute value $\sqrt{\#k}$, consider the lisse $\overline{\Q_\ell}$-sheaf $\sF$ on $\G_m/k$ whose trace function is given by
$$t \in E/k \mapsto \frac{-1}{(\Gauss_k)^{\deg(E/k)}}\sum_{x \in E}\psi_E(f(x) +tx^B)\rho(x).$$
Then we have the following results.
\begin{itemize}
\item[(i)] We have $$M_2(\sF)=\gcd({\rm gcd}_{\deg}(f),B).$$

\item[(ii)]Let us denote by $D:=\gcd({\rm gcd}_{\deg}(f),B)$. Then $f(x)$ is of the form $g(x^D)$ for a unique polynomial $g(x) \in k[x]$. 
Extend scalars so that $k$ contains $\mu_D$ and so that $\rho$ is a $D$'th power, say $\rho =\sigma^D$.
For each multiplicative character $\chi$ of order dividing $D$ we have the lisse sheaf $\sF_\chi$ whose trace function is is given by
$$t \in E/k(\mu_D) \mapsto \frac{-1}{(\Gauss_k)^{\deg(E/k)}}\sum_{x \in E}\psi_E(g(x) +tx^{B/D})\chi(x)\sigma(x).$$
Then geometrically we have
$$\sF \cong \bigoplus_{\chi \in \Char(D)}\sF_\chi,$$
each $\sF_\chi$ is geometrically irreducible, and the various $\sF_\chi$ are pairwise not geometrically isomorphic.
\end{itemize}
\end{prop}

We will also frequently use the following higher-dimensional analogue of the previous result:

\begin{cor}\label{multiABmomentrho}Given $r+1 \ge 3$ pairwise distinct integers $A, B_1,\ldots,B_r$ with
$$A \ge 0, B_1>B_2 >\ldots >B_r \ge 1,$$
 which are each prime to $p$, a finite extension $k/\F_p$, a nontrivial additive character $\psi$ of $k$, an Artin-Schreier reduced polynomial $f(x) \in k[x]$ of degree $A$, a nontrivial multiplicative character $\rho$ of $k^\times$, and a choice $\Gauss_k$ in $\overline{\Q_\ell}$ with absolute value $\sqrt{\#k}$, consider the lisse $\overline{\Q_\ell}$-sheaf $\sF$ on $(\G_m)^r/k$ whose trace function is given by
$$(t_1,\ldots,t_r) \in (E^\times)^r/k \mapsto \frac{-1}{(\Gauss_k)^{\deg(E/k)}}\sum_{x \in E}\psi_E(f(x) +\sum_i t_ix^{B_i})\rho(x).$$
Then we have the following results.
\begin{itemize}
\item[(i)] We have $$M_2(\sF)=\gcd({\rm gcd}_{\deg}(f), B_1,\ldots,B_r).$$

\item[(ii)]Let us denote by $D:=\gcd({\rm gcd}_{\deg}(f),B_1,\ldots,B_r)$. Then $f(x)$ is of the form $g(x^D)$ for a unique polynomial $g(x) \in k[x]$. 
Extend scalars so that $k$ contains $\mu_D$ and so that $\rho$ is a $D$'th power, say $\rho =\sigma^D$.
For each multiplicative character $\chi$ of order dividing $D$ we have the lisse sheaf $\sF_\chi$ whose trace function is is given by
$$(t_1,\ldots,t_r) \in (E^\times)^r/k(\mu_D) \mapsto \frac{-1}{(\Gauss_k)^{\deg(E/k)}}\sum_{x \in E}\psi_E(g(x) +\sum_it_ix^{B/D})\chi(x)\sigma(x).$$
Then geometrically we have
$$\sF \cong \bigoplus_{\chi \in \Char(D)}\sF_\chi,$$
each $\sF_\chi$ is geometrically irreducible, and the various $\sF_\chi$ are pairwise not geometrically isomorphic.\end{itemize}
\end{cor}

\begin{proof}It is clear that the trace function of $\sF$ is the sum of the trace functions of the $\sF_\chi$. So it suffices to show
that each $\sF_\chi$ is geometrically irreducible, and that they are pairwise not geometrically isomorphic.

Freeze $t_i$ for $i \ge 2$ by setting $t_i=a_i$ for any chosen $a_i \in \F_p$ for which 
 $$a_i +{\rm the\ coefficient\ of\ }x^{B_i}\ {\rm in\ }f(x) \neq 0.$$
 [By such a choice of the $a_i$, the monomials that appear in $f(x) +\sum_i t_ix^{B_i}$ will be exactly those that appear in $(f(x) +t_1x^{B_1} +\sum_{i \ge 2}a_ix^{B_i}$.]
By the previous result applied to this one variable ($t_1$) family, the 
pullbacks of the $\sF_\chi$ to $\G_m$ by $t_1 \mapsto (t_1,a_2,\ldots,a_r)$ are each geometrically irreducible and pairwise not geometrically isomorphic. So a fortiori the same is true of the $\sF_\chi$ themselves. [We use $r \ge 2$ in this argument to be sure we may apply this pullback argument in the case $A=0$, in which case $f(x)$, being Artin-Schreier reduced, is the zero polynomial, and our family
has trace function
$$(t_1,\ldots,t_r) \in (E^\times)^r/k \mapsto (-1/(\Gauss_k)^{\deg(E/k)})\sum_{x \in E}\psi_E(\sum_i t_ix^{B_i})\rho(x).]$$
\end{proof}

To end this section, let us recall the wonderful insight of van der Geer and van der Vlugt.
\begin{thm}\label{thm:vdG-vdV}
Let $p$ be a prime, $q$ a power of $p$, $E/\F_p$ a finite extension, and $f(x)\in E[X]$ a polynomial of the form
$$f(x) =\sum_{i=0}^n a_i x^{q^i + 1}$$
with $n \ge 0$ and $a_n \neq 0$.
\begin{itemize}
\item[(i)]Consider the sum 
$$S_f:= \frac{-1}{(\Gauss_{\F_p})^{\deg(E/\F_p)}}\sum_{x \in E}\psi_E(f(x)).$$
If $p$ is odd, and $E$ contains $\F_q$, then $|S_f|^2$ is a power $q^r$ of $q$, with $0 \le r \le n$. If $p=2$, the value $0$ may also occur (as it does, for
example, in  the n=0 case). If $E \subset \F_q$, then $|S_f|^2$ is a power of $\#E$. If $p=2$, the value $0$ may also occur.
\item[(ii)]Suppose $p$ is odd, and denote by $\K$ the unique subfield of $\Q(\zeta_p)$ which is quadratic over $\Q$.
Consider the polynomial 
$$\tilde f(x):= \sum_{i=0}^n a_i x^{(q^i + 1)/2}$$
and the two sums
$$S_{\tilde f,+}:= \frac{-1}{(\Gauss_{\F_p})^{\deg(E/\F_p)}}\sum_{x \in E}\psi_E(\tilde f(x)),$$
$$ S_{\tilde f,-}:= \frac{-1}{(\Gauss_{\F_p})^{\deg(E/\F_p)}}\sum_{x \in E}\psi_E(\tilde f(x))\chi_2(x).$$
Both these sums lie in $\K$. Moreover, if $q$ is a square, and $E$ contains $\F_q$, then both these sums lie in $\Q$.
\end{itemize}
\end{thm}
\begin{proof}The first statement is  van der Geer and van der Vlugt \cite{vdG-vdV}. 
For the second statement, we argue as follows.
The Gauss sum itself lies in $\K$, so it suffices to look at the sums without the Gauss sum factor.
For $\lambda^2 $ a square in $\F_p^\times$, and any power $Q$ of $p$, we have ${(\lambda^2)}^{(Q+1)/2}=\lambda^{Q+1}=\lambda^2$. So the substitution $x \mapsto \lambda^2x$ leaves the sum invariant. When $q$ is a square, and $E$ is an extension of $\F_q$, the Gauss sum factor lies in $\Q$, and again it suffices to look at the sums without the Gauss sum factor. Then every $\lambda \in \F_p^\times$ becomes a square $\tau^2$ with $\tau \in \F_q$. Then we have $({\tau^2})^{(q^i+1)/2} =\tau^{q^i+1}=\tau^2$, so the substitution $x \mapsto \tau^2 x =\lambda x$ leaves the sum invariant.
\end{proof}

\section{$(A,B)$ generalities}
In this section, we consider the following situtation. We are given a prime $p$, a  (strictly positive) power $q=p^f$ of $p$, and two
relatively prime, strictly positive integers $A,B$, both of which are prime to $p$. We also fix a prime $\ell \neq p$ so as to be able to use $\overline{\Q_\ell}$-cohomology, and an embedding of $\Q^{\mathrm {ab}}=\Q({\rm all\ roots\ of\  unity})$ into $\overline{\Q_\ell}$. We also fix a nontrivial additive character $\psi$ of $\F_p$ which, unless explicitly specified otherwise, is the additive character we will
use in forming hypergeometric sheaves. For a multiplicative character $\chi$, we denote
$\Char(A,\chi):= \{ \xi \mid \xi^A = \chi\}$ and $\Char(A):=\Char(A,\triv)$.

We first define
$$\sH_{small,A,B}:=\sH yp_\psi(\Char(A)\setminus \{\triv \};\Char(B)\setminus \{\triv \}),$$
of type $(A-1, B-1)$ and rank $\max(A,B)-1$. It is pure of weight $A+B-3$.
For each character $\chi$ with $\chi^A \neq \triv$, we define 
$$\sH_{big,A,B,\chi}:=\sH yp_\psi(\Char(A);\Char(B,\overline{\chi})),$$
of rank $\max(A,B)$.
The hypothesis that $\chi^A \neq \triv$ insures that the two sets $\Char(A)$ and $\Char(B,\chi)$ are disjoint, 
for if $\rho$ were in both, then $\rho^B=\chi$, hence $\rho^{AB}=\chi^A$. But also $\rho^A=\triv$, so  $\rho^{AB}=\triv$, a contradiction. Thus $\sH_{big,A,B,\chi}$ is of type $(A,B)$ and rank $\max(A,B)$. It is pure of weight $A+B-1$.

Similarly, for each character $\chi$ with $\chi^B \neq \triv$, we define
$$\sH^\sharp_{big,A,\chi,B}:=\sH yp_\psi(\Char(A,\chi), \Char(B)).$$
The hypothesis that $\chi^B \neq \triv$ insures that the two sets $\Char(A,\chi)$ and $\Char(B)$ are disjoint (same argument as above). Thus $\sH^\sharp_{big,A,\chi,B}$ is of type $(A,B)$ and rank $\max(A,B)$. It is pure of weight $A+B-1$.

Because $\gcd(A,B)=1$, at least one of $A,B$ is odd.

\begin{lem}\label{detdown1}If $A-B \ge 2$, then the sheaves $\sH_{small,A,B}$ and $\sH_{big,A,B,\chi}$ for any $\chi$ with $\chi^A \neq \triv$  each have geometric determinant $\sL_{\chi_2^{A-1}}$, with the understanding that $\chi_2$ is the quadratic character if $p$ is odd, and is $\triv$ if $p=2$.  If $B-A \ge 2$, then $\sH_{small,A,B}$ has geometric determinant $\sL_{\chi_2^{B-1}}$, and $\sH_{big,A,B,\chi}$ has geometric determinant $\sL_{\chi\chi_2^{B-1}}$. 
\end{lem}
\begin{proof}Immediate from \cite[8.11.2]{Ka-ESDE}.
\end{proof}

\begin{lem}\label{detdown2}
If $A-B \ge 2$, the sheaf $\sH^\sharp_{big,A,\chi,B}$ for any $\chi$ with $\chi^B \neq \triv$ has geometric determinant $\chi_2^{A-1}\chi$. If $B-A \ge 2$, $\sH^\sharp_{big,A,\chi,B}$ for any $\chi$ with $\chi^B \neq \triv$ has geometric determinant $\chi_2^{B-1}$.
\end{lem}
\begin{proof}Immediate from \cite[8.11.2]{Ka-ESDE}.
\end{proof}

\begin{rmk}\label{rmk-descent}
In the following sections, we will deal systematically with descents of hypergeometric sheaves. Here is one way to think of them. View a given hypergeometric sheaf $\sH$ as living on $\G_m/\overline{\F_p}$, and giving an irreducible $\overline{\Q_\ell}$-representation $V$ of the geometric fundamental group $\pi_1^{\geo} :=\pi_1(\G_m/\overline{\F_p})$. Given a finite field $k/\F_p$, a {\it descent} $\sH_0$ of $\sH$ may be seen as an an irreducible $\overline{\Q_\ell}$-representation $V_0$ of the arithmetic fundamental
group $\pi_1^{\ari}:=\pi_1(\G_m/k)$ whose restriction to the normal subgroup $\pi_1^{\geo} < \pi_1^{\ari}$ is $V$. If such a descent $V_0$ to $\G_m/k$ exists, any other such descent is of the form $V_0\otimes \rho$ for some one-dimensional representation $\rho$ 
of the quotient group $\pi_1^{\ari}/\pi_1^{\geo} \cong \Gal(\overline{\F_p}/k)$. This indeterminacy will appear later as a {\it clearing factor}
when we force our descents to be pure of weight zero.
\end{rmk}

The arguments of \cite[\S\S7,8]{KT3} give the following lemmas.
\begin{lem}\label{smalltracebis1}Suppose  $A$ and $B$ are strictly positive integers, both prime to $p$, with $\gcd(A,B)=1$. Then $\sH_{small,A,B}$ is geometrically isomorphic to the lisse sheaf 
$$\sH_{small,A,B, descent}$$
on $\G_m/\F_p$ whose trace function is as follows: for $E/\F_p$ a finite extension, and $u \in E^\times$, the trace at time $u$ is given by
$$u \in E^\times \mapsto -\sum_{x,y \in E, x^A=uy^B}\psi_E(Ax-By).$$
\end{lem}

\begin{lem}\label{bigtracebis1}Suppose  $A$ and $B$ are strictly positive integers, both prime to $p$, with $\gcd(A,B)=1$. Suppose $\chi^A \neq \triv$. Then $\sH_{big,A,B,\chi}$ is geometrically isomorphic to the lisse sheaf
$$\sH_{big,A,B,\chi, descent}$$
on $\G_m/\F_p(\chi)$ 
whose trace function is as follows: for $E/\F_p(\chi)$
a finite extension, and $u \in E^\times$, the trace at time $u$ is given by
$$u \in E^\times \mapsto -\sum_{x,y \in E, x^A=uy^B}\psi_E(Ax-By)\chi(y).$$
\end{lem}

\begin{lem}\label{bigtracebis2}Suppose  $A$ and $B$ are strictly positive integers, both prime to $p$, with $\gcd(A,B)=1$. Suppose $\chi^B \neq \triv$. Then $\sH^\sharp_{big,A,\chi,B}$ is geometrically isomorphic to the lisse sheaf
$$\sH^\sharp_{big,A,\chi,B, descent}$$
on $\G_m/\F_p(\chi)$ 
whose trace function is as follows: for $E/\F_p(\chi)$
a finite extension, and $u \in E^\times$, the trace at time $u$ is given by
$$u \in E^\times \mapsto -\sum_{x,y \in E, x^A=uy^B}\psi_E(Ax-By)\chi(x).$$
\end{lem}

We have the following rationality results.
\begin{lem}\label{tracesinchi1}
Suppose $A \equiv B \mod (p-1)$. Then we have the following results.
\begin{itemize}
\item[(i)] $\sH_{small,A,B, descent}$ has all its traces in $\Q$. 
\item[(ii)] Given a character $\chi$ with $\chi^A \neq \triv$, for $D$ the order of $\chi$ define 
$E_\chi:=\F_p(\mu_{D(p-1)})$.
Then pulled back to $\G_m/E_\chi$, $\sH_{big,A,B,\chi, descent}$ has all its traces in $\Q(\chi):=\Q({\rm the\ values\ of\ }\chi)$.
\item[(iii)]Given a character $\chi$ with $\chi^B \neq \triv$, for $D$ the order of $\chi$ define 
$E_\chi:=\F_p(\mu_{D(p-1)})$.
Then pulled back to $\G_m/E_\chi$, $\sH^\sharp_{big,A,\chi,B, descent}$ has all its traces in $\Q(\chi):=\Q({\rm the\ values\ of\ }\chi)$.
\end{itemize}
\end{lem}

\begin{proof}From the explicit formulas, we see that the traces lie in $\Q(\zeta_p)$ and in $\Q(\zeta_p,\chi)$ respectively. So
it suffces to show the traces are fixed by $Gal(\Q(\zeta_p)/\Q)$ and by $Gal(\Q(\zeta_p,\chi)/\Q(\chi))$ respectively.  
For any $\lambda \in \F_p^\times$, the domain of summation, $x^A=uy^B$ is mapped to itself by the automorphism $(x,y) \mapsto (\lambda x,\lambda y)$, precisely because $A \equiv B \mod (p-1)$. Making this substitution shows that each trace of $\sH_{small,A,B, descent}$ is fixed by $\F_p^\times \cong Gal(\Q(\zeta_p)/\Q)$. To prove (ii) and (iii), the key point is that over extensions of $E_\chi$, the restriction of $\chi$ to the subgroup $\F_p^\times $ of $E_\chi^\times$ is trivial. So the same substitution shows that each trace of $\sH_{big,A,B,\chi, descent}$ on $\G_m/E_\chi$ is fixed by  $\F_p^\times \cong Gal(\Q(\zeta_p,\chi)/\Q(\chi))$.
\end{proof}

\begin{lem}\label{tracesinK1}Suppose $p$ is odd. Denote by $\K$ the unique subfield of $\Q(\zeta_p)$ which is quadratic over $\Q$. {\rm [Thus $\K$ is $\Q(\Gauss_p)$, with $\Gauss_p$ either choice of quadratic Gauss sum over $\F_p$.]} Suppose  $A$ and $B$ are strictly positive integers, both prime to $p$, with $\gcd(A,B)=1$. Suppose further that we have the congruence
$$2A \equiv 2B \mod (p-1),$$
or that we have the congruence
$$(q+1)A \equiv (q+1)B \mod (p-1),$$
Then all three of the sheaves $\sH_{small,A,B, descent}$, $\sH_{big,A,B, \chi_2,descent}$ and $\sH^\sharp_{big,A,\chi_2,B,descent}$
have all their traces in $\K$.
\end{lem}

\begin{proof}For any $\lambda \in \F_p^\times$, we have $\lambda^2=\lambda^{q+1}$. So the domain of summation, $x^A=uy^B$ is mapped to itself by the automorphism $(x,y) \mapsto (\lambda^2x,\lambda^2y)=(\lambda^{q+1}x,\lambda^{q+1}y)$, either because $2A \equiv 2B \mod (p-1)$ or because $(q+1)A \equiv (q+1)B \mod (p-1)$. Making this substitution shows that each trace is fixed by by the subgroup of squares in $\F_p^\times \cong Gal(\Q(\zeta_p)/\Q)$.
\end{proof}

We now use the fact that $\gcd(A,B)=1$ to find integers $\alpha, \beta$ with
$$\alpha A -\beta B=1.$$
In the indexing set equation, $x^A=uy^B$, write $u =u^{ \alpha A -\beta B}.$ Then this equation becomes
$$ (x/u^\alpha)^A =(y/u^\beta)^B.$$
Again because  $\gcd(A,B)=1$, there exists a unique $z \in E$ such that
$$x/u^\alpha=z^B, y/u^\beta=z^A.$$
Making use of these substitutions, we obtain the following.

\begin{lem}\label{ztracesbis1}Suppose  $A$ and $B$ are strictly positive integers, both prime to $p$, with $\gcd(A,B)=1$. 
Then we have the following results.
\begin{itemize}
\item[(i)] $\sH_{small,A,B, descent}$ is isomorphic to the lisse sheaf on $\G_m/\F_p$ whose trace function is as follows: for $E/\F_p$ a finite extension, and $u \in E^\times$, the trace at time $u$ is given by
$$u \in E^\times \mapsto -\sum_{z \in E}\psi_E(Au^\alpha z^B -Bu^\beta z^A).$$

\item[(ii)]$\sH_{big,A,B,\chi,descent}$ is isomorphic to the lisse sheaf on $\G_m/\F_p(\chi)$ whose trace function is as follows: for $E/\F_p(\chi)$ a finite extension, and $u \in E^\times$, the trace at time $u$ is given by
$$u \in E^\times \mapsto -\sum_{z \in E}\psi_E(Au^\alpha z^B -Bu^\beta z^A)\chi(u^\beta z^A).$$
\item[(iii)]$\sH^\sharp_{big,A,\chi,B,descent}$ is isomorphic to the lisse sheaf on $\G_m/\F_p(\chi)$ whose trace function is as follows: for $E/\F_p(\chi)$ a finite extension, and $u \in E^\times$, the trace at time $u$ is given by
$$u \in E^\times \mapsto -\sum_{z \in E}\psi_E(Au^\alpha z^B -Bu^\beta z^A)\chi(u^\alpha z^B).$$
\end{itemize}
\end{lem}

We now consider the Kummer pullbacks by $[A]$, $u \mapsto u^A$, and by $[B]$, $u \mapsto u^B$. 
\begin{cor}\label{Apullbackbis1}Suppose  $A$ and $B$ are strictly positive integers, both prime to $p$, with $\gcd(A,B)=1$, and with $\chi$ a character with $\chi^A \neq \triv$
and with $\rho$ a character with $\rho^B \neq \triv$.
Then we have the following results.
\begin{itemize}
\item[(i)]The Kummer pullback $[A]^\star \sH_{small,A,B,descent}$ is isomorphic to the lisse sheaf on $\G_m/\F_p$ whose trace function is as follows: for $E/\F_p$ a finite extension, and $u \in E^\times$, the trace at time $u$ is given by
$$u \in E^\times \mapsto  -\sum_{z \in E}\psi_E(Au z^B -B z^A),$$
 by the substitution $z \mapsto z/u^\beta$.
 
The Kummer pullback $[B]^\star \sH_{small,A,B,descent}$ is isomorphic to the lisse sheaf on $\G_m/\F_p$ whose trace function is as follows: for $E/\F_p$ a finite extension, and $u \in E^\times$, the trace at time $u$ is given by
$$u \in E^\times \mapsto  -\sum_{z \in E}\psi_E(Az^B -Bu^{-1} z^A),$$
 by the substitution $z \mapsto z/u^\alpha$.

 
\item[(ii)]The Kummer pullback $[A]^\star \sH_{big,A,B,\chi,descent}$ is isomorphic to the lisse sheaf on $\G_m/\F_p(\chi^A)$ whose trace function is as follows: for $E/\F_p(\chi^A)$ a finite extension, and $u \in E^\times$, the trace at time $u$ is given by
$$u \in E^\times \mapsto  -\sum_{z \in E}\psi_E(Au z^B -B z^A)\chi^A(z),$$
the last equality by the substitution $z \mapsto z/u^\beta$.

The Kummer pullback $[B]^\star \sH_{big,A,B,\chi,descent}$ is isomorphic to the lisse sheaf on $\G_m/\F_p(\chi)$ whose trace function is as follows: for $E/\F_p(\chi)$ a finite extension, and $u \in E^\times$, the trace at time $u$ is given by
$$u \in E^\times \mapsto  -\sum_{z \in E}\psi_E(A z^B -Bu^{-1} z^A)\chi(u^{-1} z^A),$$
the last equality by the substitution $z \mapsto z/u^\alpha$.

\item[(iii)]The Kummer pullback  $[A]^\star \sH^\sharp_{big,A,\rho,B,descent}$ is isomorphic to the lisse sheaf on $\G_m/\F_p(\rho)$ whose trace function is as follows: for $E/\F_p(\chi^A)$ a finite extension, and $u \in E^\times$, the trace at time $u$ is given by
$$u \in E^\times \mapsto  -\sum_{z \in E}\psi_E(A uz^B -Bz^A)\rho(uz^B),$$

The Kummer pullback $[B]^\star \sH^\sharp_{big,A,\rho,B,descent}$ is isomorphic to the lisse sheaf on $\G_m/\F_p(\rho^B)$ whose trace function is as follows: for $E/\F_p(\chi)$ a finite extension, and $u \in E^\times$, the trace at time $u$ is given by
$$u \in E^\times \mapsto  -\sum_{z \in E}\psi_E(A z^B -Bu^{-1} z^A)\rho^B(z),$$
the last equality by the substitution $z \mapsto z/u^\alpha$.

\end{itemize}
\end{cor}

%

\begin{lem}\label{tracesinK2}
Suppose  $A$ and $B$ are strictly positive integers, both prime to $p$, with $\gcd(A,B)=1$. 
Define $\K :=\Q(\Gauss_{\F_p})$. Suppose further that $A>B$ and that
$$2A \equiv 2B \equiv 2 \mod (p-1),$$
or that
$$(q+1)A \equiv (q+1)B\equiv 2 \mod (p-1).$$
Then we have the following results.
\begin{itemize}
\item[(i)]The lisse sheaf $[A]^\star \sH_{small,A,B,descent}$ on $\G_m/\F_p$ has all traces in $\K$.
\item[(ii)]Given $\chi$ with $\chi^A \neq \triv$, and $D$ the order of $\chi^A$, define $\F_{\chi^A} :=\F_p(\mu_{D(p-1)/2})$. Then the lisse sheaf $[A]^\star \sH_{big,A,B,\chi,descent}$ on $\G_m/\F_{\chi^A}$ has all traces in $\K$.
\item[(iii)]Given $\rho$ with $\rho^B \neq \triv$, and $D$ the order of $\rho^B$, define $\F_{\rho^B} :=\F_p(\mu_{D(p-1)/2})$. Then the lisse sheaf $[B]^\star \sH^\sharp_{big,A,\chi,B,descent}$ on $\G_m/\F_{\rho^B}$ has all traces in $\K$.
\end{itemize}
\end{lem}

\begin{proof}The key point is that under the first hypothesis, for $\lambda \in \F_p^\times$, we have $(\lambda^2)^A = \lambda^{2A }=\lambda^2$, similarly $(\lambda^2)^B= \lambda^2$. Under the second hypothesis,
 $(\lambda^2)^A=(\lambda^{q+1})^A=\lambda^2$, similarly $(\lambda^2)^B= \lambda^2$. To prove (i),  
for each $\lambda \in \F_p^2$, simply make the substitution $z \mapsto \lambda^2z$; this does not change the sum, but both $z^A$ 
and $z^B$ are multiplied by $\lambda^2$. For (ii) and (iii), use the same substitution, remembering that $\chi^A(\lambda^2)=1$ over 
extensions of $F_{\chi^A}$ and $\rho^B(\lambda^2)=1$ over 
extensions of $F_{\rho^B}$.
\end{proof}

\begin{prop}\label{arithdet}Suppose  $A$ and $B$ are strictly positive integers, both prime to $p$, with $\gcd(A,B)=1$. Suppose further that $A-B \ge 2$ and that $A$ is odd. 
Then we have the following results.
\begin{itemize}
\item[(i)]Suppose $p$ is odd. The for any nontrivial additive character $\psi'$ of $\F_p$, the lisse sheaf $\sH_{small,A,B, descent}\otimes (-\Gauss(\psi',\chi_2))^{-\deg}$ on $\G_m/\F_p$ is pure of weight zero and has arithmetic determinant $\bigl((\chi_2(-1))^{(A-1)/2}\bigr)^{\deg}$.
\item[(ii)]Suppose $p=2$. Then the lisse sheaf $\sH_{small,A,B, descent}\otimes (-\sqrt{2})^{-\deg}$ on $\G_m/\F_2$ is pure of weight zero and has arithmetically trivial determinant.
\item[(iii)]Suppose $p$ is odd. Define $\epsilon(A):=(-1)^{(A-1)/2}$, and denote by $\psi_{-\epsilon(A)AB}$ the  nontrivial additive character $t \mapsto \psi(-\epsilon(A)ABt)$ of $\F_p$. Then the lisse sheaf 
$$\sH_{big,A,B, \chi_2,descent} \otimes (-\Gauss(\psi_{-\epsilon(A)AB},\chi_2))^{-\deg}$$ 
on $\G_m/\F_p$ has arithmetically trivial determinant.
\end{itemize}
\end{prop}
\begin{proof}The explicit formulas of Lemma \ref{ztracesbis1} makes clear that $\sH_{small,A,B}$ and $\sH_{big,A,B, \chi_2,descent}$ are pure of weight one, and thus their twists in the two cases puts us in weight zero. 
We first prove (i) and (ii). To save having to write 
``by $(-\Gauss(\psi',\chi_2))^{-\deg}$ or by $(-\sqrt{2})^{-\deg}$" in the rest of the proof, let us adopt the convention that
$$\Gauss(\psi',\chi_2) := \sqrt{p} \ \ {\rm for\ }p=2.$$

The hypergeometric description of $\sH_{small,A,B}$ shows that the determinant is geometrically trivial, cf. \cite[8.12.2, (3)]{Ka-ESDE}. Therefore the determinant is of the form $D^{\deg}$ for some $\ell$-adic unit $D$. To show that $D=1$, it suffices to do so at the single point $v=1$ in $\G_m(\F_p)$. This determinant at $v=1$ is equal to the determinant at $v=1$ on the Kummer pullback by $v \mapsto v^A$. But this Kummer pullback, whose trace function is, by Corollary \ref{Apullbackbis1},
$$v \in E^\times \mapsto  (1/\Gauss(\psi,\chi_2))^{\deg(E/\F_p)}\sum_{z \in E}\psi_E(Av z^B -B z^A),$$
is lisse on $\A^1/\F_p$, and its determinant remains $D^{\deg}$. 

So its determinant at $v=1$ is equal to its determinant at $v=0$.
But at $v=0$, we are looking at the cohomology group $$H^1(\A^1/\overline{\F_p},\sL_{\psi(-Bz^A)}) \otimes(-\Gauss(\psi,\chi_2))^{-\deg}.$$
Because $A$ is odd, {\bf if} we had twisted instead by either choice of $\sqrt{p}^{-\deg}$, then we would be arithmetically symplectically self dual, and would have determinant $1$. So if $p$ is $1$ mod $4$, this is our situation:
we are arithmetically symplectically self-dual, and therefore the determinant $D=1$. However, if $p$ is $3$ mod $4$, then our twisting Gauss sum is $i\sqrt{p}$, so our determinant is $i^{A-1}$, which is $(-1)^{(A-1)/2}$, which, because $p$ is $3$ mod $4$, is the asserted $(\chi_2(-1))^{(A-1)/2}$.

We now turn to proving (iii). Again the hypergeometric description of $\sH_{big,A,B, \chi_2,descent}$ shows, by  \cite[8.12.2, (3)]{Ka-ESDE}, that its determinant is geometrically trivial, of the form $D^{\deg}$ for some $\ell$-adic unit $D$. We now repeat the argument above. It suffices to show that $D=1$ at the point $v=1$ in $\G_m(\F_p)$. This is equal to the determinant at $v=1$ on the Kummer pullback by $v \mapsto v^A$. But this Kummer pullback, whose trace function is, by Corollary \ref{Apullbackbis1},
$$v \in E^\times \mapsto  (1/\Gauss(\psi_{\epsilon(A)A},\chi_2),\chi_2))^{\deg(E/\F_p)}\sum_{z \in E}\psi_E(Av z^B -B z^A)\chi_2(z),$$
is lisse on $\A^1/\F_p$, and its determinant remains $D^{\deg}$. So its determinant at $v=1$ is equal to its determinant at $v=0$.
But at $v=0$, we are looking at the cohomology group $$H^1(\A^1/\overline{\F_p},\sL_{\psi(-Bz^A)}\otimes \sL_{\chi_2(z)}) \otimes(-\Gauss(\psi_{-\epsilon(A)AB},\chi_2),\chi_2),\chi_2))^{-\deg},$$
which, because $A$ is odd, is orthogonally self-dual. It is proven in \cite[1.4]{Ka-NG2} that with the imposed choice of quadratic Gauss sum, this orthogonal autoduality has determinant $D=1$ (remembering that we are applying the cited result to the additive character $\psi_{-B}$). 
\end{proof}

\begin{prop}\label{arithdetbis}Suppose  $A$ and $B$ are strictly positive integers, both prime to $p$, with $\gcd(A,B)=1$. Suppose further that $p$ is odd, that $A-B \ge 2$ and that $A$ is even. 
Then we have the following results.
\begin{itemize}
\item[(i)]For any choice of $C \in \F_p^\times$, denote by $\psi_C$ the additive character $t \mapsto  \psi(Ct)$. 
The lisse sheaf $[A]^\star \sH_{small,A,B, descent}\otimes (-\Gauss(\psi_C,\chi_2))^{-\deg}$ on $\A^1/\F_p$, whose trace function at $u \in E/\F_p$ is
$$u \in E \mapsto  (-1/ (-\Gauss(\psi_{C},\chi_2))^{\deg(E/\F_p)})\sum_{z \in E}\psi_E(Au z^B -B z^A),$$
 is pure of weight zero and has arithmetic determinant $\bigl(\chi_2(2ABC(-1)^{A/2})\bigr)^{\deg}$.
\item[(ii)]Choose a character $\rho$ with $\rho^A=\chi_2$. For any choice of  $C \in \F_p^\times$, the lisse sheaf
$$[A]^\star \sH_{big,A,B, \rho,descent}\otimes (-\Gauss(\psi_{C},\chi_2))^{-\deg}$$ on $\A^1/\F_p$, with trace function
$$u \in E \mapsto  (-1/ (-\Gauss(\psi_{C},\chi_2))^{-\deg(E/\F_p)}\sum_{z \in E}\psi_E(Au z^B -B z^A)\chi_2(z),$$
is pure of weight zero and has arithmetic determinant $\bigl(\chi_2(2(-1)^{A/2})\bigr)^{\deg}$.
\end{itemize}
\end{prop}
\begin{proof}The explicit formulas of Corollary \ref{Apullbackbis1} make that the $[A]^\star$ pullbacks are lisse on $\A^1/\F_p$, and after the twist by any quadratic Gauss sum, are pure of weight zero. Because $A$ is even, it results from Lemma \ref{detdown1} that each pullback has geometrically trivial determinant. So to compute the arithmetic determinant, it suffices to do so at time $u=0$ in
$\A^1(\F_p)$. At this point, we use the idea already used in the proof of Proposition \ref{arithdet}, namely we first compute the determinants of the cohomology groups
$$H^1(\A^1/\overline{\F_p},\sL_{\psi(-Bz^A)}) {\rm \ \  and\ }H^1(\A^1/\overline{\F_p},\sL_{\psi(-Bz^A)}\otimes \sL_{\chi_2(z)}) .$$
These are computed in parts (1) and (2) of \cite[Theorem 2.3]{KT1}, where the $D$ there is our $A$, the $q$ there is $p$, and
the $\psi$ there is our $\psi_{-B}$. The first determinant is
$$(-\Gauss(\psi_{-BA/2},\chi_2)p^{(A/2)-1}=\chi_2(-CAB/2)(-\Gauss(\psi_C,\chi_2)(\chi_2(-1)(-\Gauss(\psi_C,\chi_2))^2)^{(A/2)-1}=$$
$$=\chi_2(-2CAB)\chi_2(-1)^{(A/2)-1}(-\Gauss(\psi_C,\chi_2)^{A-1}=\chi_2(2CAB(-1)^{A/2})(-\Gauss(\psi_C,\chi_2)^{A-1}.$$
The second determinant is $(-\Gauss(\psi_{BA},\chi_2)$ times the first, so is
$$\chi_2(CAB)\chi_2(2CAB(-1)^{A/2})(-\Gauss(\psi_C,\chi_2)^{A}.$$
In both cases, after the Gauss sum twisting, we are left with the asserted arithmetic determinant.
\end{proof}

\begin{prop}\label{arithdetter}Suppose  $A$ and $B$ are strictly positive integers, both prime to $p$, with $\gcd(A,B)=1$. Suppose further that $p$ is odd, that $A-B \ge 2$, that $A$ is even and that $B$ is odd. Choose integers $\alpha, \beta$ with
$$\alpha A -\beta B =1, \ \ \alpha \ {\rm even},$$
(which is always possible, for if $(\alpha, \beta)$ works, then so does $(\alpha +B, \beta + A)$, and $B$ is odd).
For any choice of  $C \in \F_p^\times$, the lisse sheaf
$$\sH^\sharp_{big,A,\chi_2,B,descent}\otimes (-\Gauss(\psi_{C},\chi_2))^{-\deg}$$ on $\G_m/\F_p$, with trace function
$$u \in E \mapsto  (-1/ (-\Gauss(\psi_{C},\chi_2))^{-\deg(E/\F_p)}\sum_{z \in E}\psi_E(Au^\alpha z^B -Bu^\beta z^A)\chi_2( z),$$
is pure of weight zero and has arithmetic determinant $\bigl(\chi_2(2(-1)^{A/2})\bigr)^{\deg}$.
\end{prop}
\begin{proof}From the explicit formula for its trace function, it is obvious that $\sH^\sharp_{big,A,\chi_2,B,descent}$ is pure of weight one,
and hence that twisting by any quadratic Gauss sum renders it pure of weight zero. Because $B$ is odd and $\alpha$ is even, the asserted trace formula is just (iii) of Proposition \ref{ztracesbis1}.
Because $A$ is even, it results from Lemma \ref{detdown2} that $\sH^\sharp_{big,A,\chi_2,B,descent}$ has geometrically trivial determinant. Therefore we may compute its arithmetic determinant at the point $u=1$ in $\G_m(\F_p)$. This is the determinant attached to the trace function over varying extensions $E/\F_p$ given by
$$-\sum_{z \in E}\psi_E(A z^B -Bz^A)\chi_2( z),$$
which is in term the trace at $u=1$ on the lisse sheaf on $\A^1/\F_p$ whose trace function is
$$u \in E \mapsto -\sum_{z \in E}\psi_E(Au z^B -Bz^A)\chi_2( z).$$
This lisse sheaf is none other than $[A]^\star \sH_{big,A,B, \rho,descent}$, for any choice of $\rho$ with $\rho^A=\chi_2$, which, we have seen in  Proposition \ref{arithdetbis}, has geometrically trivial determinant, and whose arithmetic determinant is the asserted 
$\bigl(\chi_2(2(-1)^{A/2})\bigr)^{\deg}$.
\end{proof}

From \cite[Prop. 1.2]{Ka-RL-T-2Co1}, one sees the following.
\begin{prop}\label{ABBelyi1}If a geometrically irreducible hypergeometric sheaf of type $(n,m)$ with $n > m > 0$, or with $m > n > 0$, is Belyi induced, then $n-m$ is prime to $p$, and is divisible by $p-1$.
\end{prop}
\begin{proof}In the notations of \cite[Prop. 1.2]{Ka-RL-T-2Co1} (whose $A$ and $B$ have nothing to do with ours), when $n >m >0$, we have $n=A+B$ and either $A+B$ or $A$ or $B$ is $d_0p^r$ with $sr \ge 1$ and $d_0$ prime to $p$. In these cases, $m$ is either
$d_0$, or $d_0 +B$, or $A +d_0$. So in each case, $n-m =d_0p^r -d_0 =d_0(p^r -1)$, which is prime to $p$ and is divisible by $p-1$. To deal with the case $m > n > 0$, first apply multiplicative inversion.
\end{proof}

\begin{cor}\label{ABprimitive1}Suppose  $A$ and $B$ are strictly positive integers, both prime to $p$, with $\gcd(A,B)=1$. Suppose further that $A-B$ is divisible by $p$. Then neither $\sH_{small,A,B}$ nor $\sH_{big,A,B,\chi}$ for any $\chi$ with $\chi^A \neq \triv$ nor $\sH^\sharp_{big,A,\rho,B}$ for any $\rho$ with $\rho^B \neq \triv$
is geometrically induced.
\end{cor}
\begin{proof} The relative primality of $A$ and $B$ shows that neither sheaf is Kummer induced. That neither is Belyi induced results from Proposition \ref{ABBelyi1}.
\end{proof}

Combining the above primitivity result Corollary \ref{ABprimitive1} with \cite[Theorem 1.5]{KT4}, we get the following two results.
\begin{prop}\label{ABcondSbig1}Suppose  $A$ and $B$ are strictly positive integers, both prime to $p$, with $\gcd(A,B)=1$. Suppose that $A-B$ is divisible by $p$, and that $\max(A,B)$ is prime to $p$. Suppose further that 
$$|A-B| > \max(A,B)/2 \ge 2.$$ 
In the special case $\max(A,B)=8$, suppose $A-B=7$ (possible only when $p=7$). In the special case $\max(A,B)=9$, suppose $A-B=7 \ {\rm\ or}\ 8$ (possible only when $p=7$, respectively when $p=2$). Then for any $\chi$ with $\chi^A \neq \triv$,  $\sH_{big,A,B,\chi}$ satisfies condition  {\rm ({\bf S+})  (as defined in \cite[\S1]{KT4})}.
\end{prop}

\begin{prop}\label{ABcondSsmall1}Suppose  $A$ and $B$ are strictly positive integers, both prime to $p$, with $\gcd(A,B)=1$. Suppose that $A-B$ is divisible by $p$, and that $\max(A-1,B-1)$ is  divisible by $p$. Suppose further that $$|A-B| >(2/3)( \max(A-1,B-1)-1)> 2$$ In the special case $\max(A,B)=8$, suppose $A-B=7$ (possible only when $p=7$). In the special case $p=2$, suppose $ \max(A-1,B-1) \neq 8$. Then 
$\sH_{small,A,B}$ satisfies condition {\rm ({\bf S+})}.
\end{prop}

\section{Local system candidates for $\Sp_{2n}(q)$}\label{sec:local-sp}
In this section, expanding \cite{KT3}, we consider the following situation: $p$ is an odd prime, $q$ is a (strictly positive) power $q=p^f$ of $p$, and we are given  two positive integers $n \neq m$ with $\gcd(n,m)=1$ about which we assume
$$\gcd(q^n+1,q^m+1)=2.$$
Notice that $n,m$ cannot both be odd, otherwise $q+1$ divides $\gcd(q^n+1,q^m+1)$, nor can they both be even, as  $\gcd(n,m)=1$.
So precisely one of $n,m$ is even, and the other is odd. In what follows, we suppose that
$$n {\rm \ even},\ m {\rm \ odd}.$$
We define
$$A:=(q^n+1)/2,\ B:=(q^m +1 )/2.$$

We will apply the results of the previous section to this $(A,B)$, and to the quadratic character $\chi_2$.
Thus $A$ is odd. The parity of $B$ depends on the value of $q$ mod $4$ ($B$ will be odd if $q$ is $1$ (mod $4$), and will be even
if $q$ is $3$ (mod $4$)).

%
\begin{lem}\label{parity} There exist integers $\alpha, \beta$ with $\alpha A-\beta B=1$ with $\beta$ even.
\end{lem}
\begin{proof}If $\alpha, \beta$ has $\alpha A-\beta B=1$, so does $(\alpha +B, \beta +A)$. Since $A$ is odd, we may
change the parity of $\beta$ at will.
\end{proof}

For the rest of this section, we fix a choice of $\alpha, \beta$ with 
$$\alpha A-\beta B=1, \  \beta\  {\rm even}.$$


We consider the hypergeometric sheaf
$$\sH_{small,A,B}:=\sH yp_\psi(\Char(A)\setminus \{\triv \};\Char(B)\setminus \{\triv \}),$$
of rank $\max(A,B)-1$, 
and the hypergeometric sheaf
$$\sH_{big,A,B}:=\sH_{big,A,B,\chi_2}:=\sH yp_\psi(\Char(A);\Char(B,\chi_2)),$$
of rank $\max(A,B)$.


\begin{lem}\label{ztraces} We have the following results.
\begin{itemize}
\item[(i)] {\rm (mise\ pour\ m\'{e}moire)} $\sH_{small,A,B, descent}$ is isomorphic to the lisse sheaf on $\G_m/\F_p$ whose trace function is as follows: for $E/\F_p$ a finite extension, and $u \in E^\times$, the trace at time $u$ is given by
$$u \in E^\times \mapsto -\sum_{z \in E}\psi_E(Au^\alpha z^B -Bu^\beta z^A).$$

\item[(ii)]$\sH_{big,A,B,descent}$ is isomorphic to the lisse sheaf on $\G_m/\F_p$ whose trace function is as follows: for $E/\F_p$ a finite extension, and $u \in E^\times$, the trace at time $u$ is given by
$$u \in E^\times \mapsto -\sum_{z \in E}\psi_E(Au^\alpha z^B -Bu^\beta z^A)\chi_2(z).$$
\end{itemize}
\end{lem}
\begin{proof}Immediate from Corollary \ref{Apullbackbis1}. The first assertion is Lemma \ref{ztracesbis1}(i). The second is statement (ii) of that same lemma, remembering that
$\beta$ is even and $A$ is odd, so that $\chi_2(u^\beta z^A)=\chi_2(z)$.
\end{proof}
\smallskip

%

\begin{cor}\label{totalztrace}The direct sum 
$$\sW_{\sH}:=\sH_{small,A,B,descent} \oplus \sH_{big,A,B,descent}$$
is isomorphic to the  arithmetically semisimple lisse sheaf on $\G_m/\F_p$ whose trace function is as follows: for $E/\F_p$ a finite extension, and $u \in E^\times$, the trace at time $u$ is given by
$$u \in E^\times \mapsto -\sum_{z \in E}\psi_E(Au^\alpha z^{q^m+1} -Bu^\beta z^{q^n+1}).$$
\end{cor}
\begin{proof}Indeed, the trace function at time $u \in E^\times$ of this direct sum has value
$$-\sum_{z \in E}\psi_E(Au^\alpha z^B -Bu^\beta z^A)(1+\chi_2(z))=-\sum_{z \in E}\psi_E(Au^\alpha z^{2B} -Bu^\beta z^{2A}),$$
and $2A=q^n+1, \ 2B=q^m+1$.
\end{proof}

%

\begin{rmk}If $A > B$, then after $[A]$ pullback, both $[A]^\star \sH_{small,A,B}$ and $[A]^\star \sH_{big,A,B}$ become lisse on $\A^1$.
But when $A < B$,  neither $[A]^\star \sH_{small,A,B}$ nor $[A]^\star \sH_{big,A,B}$ becomes lisse on $\A^1$.
\end{rmk}

\begin{thm}\label{det}  Suppose $A > B$.  Then we have the following results.
\begin{itemize}
\item[(i)]For ${\sf G}$ either choice of minus the quadratic Gauss sum over $\F_p$, the lisse sheaf $$\sH_{small,A,B,descent}\otimes {\sf G}^{-\deg}$$ on $\G_m/\F_p$ has arithmetically trivial determinant.
\item[(ii)]Denote by $\overline{\psi}$ the nontrivial additive character $t \mapsto \psi(-t)$. Then with the clearing factor 
$${\sf G} := -\Gauss(\overline{\psi},\chi_2)$$
the lisse sheaf  $\sH_{big,A,B,descent}\otimes {\sf G}^{-\deg}$
 on $\G_m/\F_p$ has arithmetically trivial determinant.
\item[(iii)]If $-1$ is not a square in $\F_p$, then with the clearing factor $-\Gauss(\psi,\chi_2)$, the arithmetic determinant of the
 lisse sheaf  $\sH_{big,A,B,descent}\otimes(-\Gauss(\psi,\chi_2))^{-\deg}$ is $(-1)^{\deg}$.
\end{itemize}
\end{thm}

\begin{proof}The first assertion is a special case of Proposition \ref{arithdet}(i), because the $(A-1)/2$ exponent there is
$(q^n-1)/4$, which is even because $n$ is even and $q$ is odd.  The second assertion is a special case of Proposition \ref{arithdet}(iii), remembering that in this case $\epsilon(A)=1$ (because, as $n$ is even,  $(A-1)/2 = (q^n-1)/4$ is even), and
$-AB=-(q^n+1)(q^m+1)/4$ is $-1/4(\bmod\ p)$, and mod squares is $-1$. For (iii), observe that when $-$ is not a square in $\F_p$, the ``usual" Gauss sum is minus the one making the arithmetic determinant trivial in (ii), and as the rank $A$ is odd, the arithmetic determinant in (iii) will be $(-1)^{\deg}$.
\end{proof}

\begin{thm}\label{detbis}Suppose $B >A$. Then on $\G_m/\F_{p^2}$, with the clearing factor 
$$\tilde{\sf G} := (-1)^{(p-1/2}p,$$
we have the following results.
\begin{itemize}
\item[(i)]
If $B$ is odd, the lisse sheaf $\sH_{small,A,B,descent}\otimes \tilde{\sf G}^{-\deg}$ has arithmetically trivial determinant, and
the lisse sheaf $\sH_{big,A,B,descent}\otimes \tilde{\sf G}^{-\deg}$ has geometric determinant  $\sL_{\chi_2}$. The  Kummer pullback $$[2]^\star\sH_{small,A,B,descent}\otimes \tilde{\sf G}^{-\deg}$$ has arithmetically trivial determinant.
\item[(ii)]
If $B$ is even, the lisse sheaf $\sH_{small,A,B,descent}\otimes \tilde{\sf G}^{-\deg}$ has geometric determinant  $\sL_{\chi_2}$, and the lisse sheaf $\sH_{big,A,B,descent}\otimes \tilde{\sf G}^{-\deg}$  has arithmetically trivial determinant. 
 The  Kummer pullback $$[2]^\star\sH_{small,A,B,descent}\otimes \tilde{\sf G}^{-\deg}$$ has arithmetically trivial determinant.
\end{itemize}
\end{thm}
\begin{proof}Again we have $B-A \ge 2$, so the geometric determinant of $\sH_{small,A,B,descent}$ is the product of all the nontrivial characters of order dividing $B$, so is trivial if $B$ is odd and is $\sL_{\chi_2}$ if $B$ is even. The geometric determinant of $\sH_{big,A,B,descent}$
is the product of all the characters in $\Char(B,\chi_2)$, which is $\sL_{\chi_2}$ if $B$ is odd and is trivial if $B$ is even. So after taking
the appropriate Kummer $[2]^\star$ pullbacks, both sheaves in question have geometrically trivial determinants, and we proceed
as in the proof of Theorem \ref{det}, first evaluating at $u=1$, then at $u=1$ on the $[B]^\star$ pullbacks, which are lisse on $\P^1 \setminus 0$, then at $u=\infty$ on these pullbacks, to reduce to applying \cite[2.3, (1) and (2)]{KT1}, with the $D$ there taken to be our $B$, and the $\psi$ taken to be $\psi_{-A}$. The results of  \cite[2.3, (1) and (2)]{KT1} involve various quadratic Gauss sums, but by working on $\G_m/\F_{p^2}$, only their squares occur, and these squares are each the $\tilde{\sf G}$ in the statement of the theorem.
\end{proof}

From Lemma \ref{tracesinK2}, we have
\begin{lem}\label{ourtracesinK1} Each of the sheaves $[A]^\star\sH_{small,A,B,descent}$ and $[A]^\star \sH_{big,A,B,descent}$ has all its traces in $\K$.
\end{lem}

Recall once again that for our $(A,B)$, the image of each of $A,B$ in $\F_p$ is $1/2$. So if we denote by $\psi_{-1/2}$ the additive character of $\F_p$ given by $x \mapsto \psi(-x/2)$, then we can restate Lemma \ref{ztracesbis1} in our case as follows.

\begin{lem}\label{Wtrace}The direct sum 
$$\sW_{ \Sp}:=[A]^\star \sH_{small,A,B,descent} \oplus [A]^\star  \sH_{big,A,B,descent}$$
is isomorphic to the lisse sheaf on $\A_1/\F_p$ whose trace function is as follows: for $E/\F_p$ a finite extension, and $u \in E^\times$, the trace at time $u$ is given by
$$u \in E^\times \mapsto -\sum_{z \in E}\psi_{-1/2,E}(  z^{q^n+1}-uz^{q^m+1}).$$
\end{lem}


We now consider the Kummer pullback by $[A]$. We apply Corollary \ref{Apullbackbis1}.
\begin{cor}\label{pullback}
We have the following results.
\begin{itemize}
\item[(i)]The Kummer pullback $[A]^\star \sH_{small,A,B,descent}$ is isomorphic to the lisse sheaf on $\G_m/\F_p$ whose trace function is as follows: for $E/\F_p$ a finite extension, and $u \in E^\times$, the trace at time $u$ is given by
$$u \in E^\times \mapsto -\sum_{z \in E}\psi_{-1/2,E}(z^A- u z^B).$$
\item[(ii)]The Kummer pullback $[A]^\star \sH_{big,A,B,descent}$ is isomorphic to the lisse sheaf on $\G_m/\F_p$ whose trace function is as follows: for $E/\F_p$ a finite extension, and $u \in E^\times$, the trace at time $u$ is given by
$$u \in E^\times \mapsto  -\sum_{z \in E}\psi_{-1/2,E}(z^A-u z^B )\chi_2(z).$$
\item[(iii)]The Kummer pullback $[A]^\star \sW_{\sH,\Sp}$ is isomorphic to the arithmetically semisimple lisse sheaf on $\G_m/\F_p$ whose trace function is as follows: for $E/\F_p$ a finite extension, and $u \in E^\times$, the trace at time $u$ is given by
$$u \in E^\times \mapsto  -\sum_{z \in E}\psi_{-1/2,E}(z^{2A}-u z^{2B} )=-\sum_{z \in E}\psi_{-1/2,E}(z^{q^n+1}-u z^{q^m+1} ).$$
In particular, if $E$ is a subfield of $\F_q$, then the trace at the time $u=1$ is $-\# E$.
\end{itemize}
\end{cor}


Let us denote by 
$$\sW_{\sH,\Sp}(1/2)$$
the constant field twist of $\sW_{\sH,\Sp}$ obtained by dividing the trace over $E$ by $-{\mathsf {Gauss}}(\overline\psi_E,
\chi_2)$, so that we are pure of weight zero.
\begin{thm}\label{vdG-vdV2}We have the following results.
\begin{itemize}
\item[(ii]Over any extension of $\F_q$, the square absolute values of traces of $\sW_{\sH,\Sp}(1/2)$ are powers of $q$. Over any subfield $k$ of $\F_q$, the square absolute values of the traces of $\sW_{\sH,\Sp}(1/2)$ are powers of $\#k$; moreover, the trace at $u=1$ has squared absolute value equal to $\# k$. 
\item[(ii)]The arithmetic and geometric monodromy groups of $\sW_{\sH,\Sp}(1/2)$ (and hence also of both $\sH_{small,A,B,descent}(1/2)$ and $\sH_{big,A,B,descent}(1/2)$)  are finite, and all three of these sheaves have all traces in $\K$. 
\end{itemize}
\end{thm}
\begin{proof}Assertion (i) is van der Geer-van der Vlugt \cite[Lemma 5.2]{KT2}, and Lemma \ref{ourtracesinK1}. It implies the second assertion, cf. \cite[\S5]{KT2}. [Note that the definition of the relevant local systems in \cite[\S5]{KT2} uses a possibly different clearing factor, 
which can, however, change only the sign of the trace function.]
\end{proof}

Let us define
 $$\sW_{\Sp}:=[A]^\star \sW_{\sH,\Sp}, $$
 and denote by
  $$\sW_{\Sp}(1/2)$$
 the constant field twist of $\sW_{\Sp}$ obtained by dividing the trace over $E$ by $-{\mathsf {Gauss}}(\overline\psi_E,
\chi_2)$. We record the following corollary of Theorem \ref{vdG-vdV2}.
 \begin{cor}\label{AvdG-vdV2}We have the following results.
\begin{itemize}
\item[(ii]Over any extension of $\F_q$, the square absolute values of traces of $\sW_{\Sp}(1/2)$ are powers of $q$. Over any subfield $k$ of $\F_q$, the square absolute values of the traces of $\sW_{\Sp}(1/2)$ are powers of $\#k$; moreover, the trace at $u=1$ has squared absolute value equal to $\# k$. 
\item[(ii)]The arithmetic and geometric monodromy groups of $\sW_{\Sp}(1/2)$ (and hence also of both $[A]^\star\sH_{small,A,B,descent}(1/2)$ and $[A]^\star\sH_{big,A,B,descent}(1/2)$)  are finite, and all three of these sheaves have all traces in $\K$. 
\end{itemize}
\end{cor} 
 
 

\begin{thm}\label{sp-det}
 If $A > B$, then on $\A^1/\F_{p^2}$, both  $\sH_{small,A,B,descent}(1/2)$ and $\sH_{big,A,B,descent}(1/2)$ have arithmetically trivial determinants. If $B > A$, then on  $\G_m/\F_{p^2}$, both  $[2]^\star \sH_{small,A,B,descent}(1/2)$ and $[2]^\star \sH_{big,A,B,descent}(1/2)$ have arithmetically trivial determinants.
 \end{thm}
 \begin{proof}
 This is a special case of Theorem \ref{det} when $A > B$, and of Theorem \ref{detbis} when $B > A$.
 \end{proof}


\begin{prop}\label{cond-S}
Both sheaves $\sH_{small,A,B}$ and $ \sH_{big,A,B}$ satisfy condition {\rm ({\bf S+})}. Their wild parts (at $\infty$ if $A>B$, at $0$ if $B > A$) have dimension $|A-B|$, with all slopes $1/|A-B|$.
\end{prop}
\begin{proof}To avoid the confusion caused by not knowing if $A>B$ or if $B>A$, let us define
$$c:=\Max(n,m), \ d:=\Min(n,m), \ C:=\Max(A,B)=(q^c+1)/2, \ D:=\Min(A,B) =(q^d+1)/2.$$
At the expense of a possible multiplicative inversion, our sheaves are of type $(C-1,D-1)$ and of type $(C,D)$. Both have wild part of dimension
$$W:=C-D =(q^c -q^d)/2.$$
This difference being divisible by $p$, their primitivity results from Corollary \ref{ABprimitive1}. 

The ranks are both prime to $p$, being  $(q^c \pm 1)/2$. So we may apply
\cite[Theorem 1.5]{KT4}. We must show that $W >(D-1)/2$ for $\sH_{small,A,B}$  and $W >D/2$, that $D -1 \ge 4$, and check that neither $D-1$ nor $D$ is $8$. We have $c \ge 2$, so $D-1 = (q^c -1)/2 \ge (3^2-1)/2 =4$. Neither $D-1$ nor $D$ is $8$, otherwise we have either
$$(q^c-1)/2 =8 {\rm \ or \ }(q^c+1)/2 =8.$$
In the first case $q^c=17$, impossible as $c \ge 2$. In the second case, $q^c=15$, nonsense.

It remains to show that $W > D/2$, i.e., that
$$(q^c - q^d)/2 >(q^c+1)/4,$$
or, equivalently,
$$2(q^c - q^d) >q^c _1, {\rm \  \ \ i.e.\ \ \ }q^c -2q^d>1, {\rm \ \ \  i.e.\ \ \ } q^d(q^{c-d}-2) >1.$$
This last holds because $q \ge 3$ and $c-d \ge 1$, $d \ge 1$. 
\end{proof}
 
%
%

\section{Local system candidates for $\SU_n(q)$}
We now turn to the particular situation relevant to unitary groups. Thus $q$ is a power of $p$, $n >m >0$ are odd integers with
$\gcd(n,m)=1$ (which forces $\gcd(q^n+1,q^m+1) =q+1$), and
$$A:=(q^n+1)/(q+1), \ B:=(q^m+1)/(q+1).$$

\begin{rmk}In this section, we impose $n > m$, so that $A > B$. 
In the previous section, we imposed $n$ even, $m$ odd, but this did not determine which of $A, B$ was the larger.
\end{rmk}

\begin{lem}\label{ABparity}Both A and B are odd. Indeed $A \equiv n{\rm\ mod\ }(q+1),\ \ B \equiv m{\rm\ mod\ }(q+1).$
\end{lem}
\begin{proof}If $p=2$, then each of  $q^n+1,q^m+1, q+1$ is odd. If $p$ is odd, then
$$A:=(q^n +1)/(q+1)=1-q+q^2 -q^3+ \ldots +q^{n-1} \equiv n{\rm\ mod\ }(q+1),$$
hence $A$ has the same parity as $n$, which is odd. Simillarly, $B \equiv m{\rm\ mod\ }(q+1)$.
\end{proof}

Notice here that both $A,B$ are $\equiv 1 \mod (q-1)$, hence also  $\equiv 1 \mod (p-1)$.

\begin{prop}\label{totaltraces}For any nontrivial character $\rho$ of order dividing $q+1$, there is a lisse sheaf $\sG_\rho$ on $\A^1/\F_{q^2}$ whose trace function is as follows: for $E/\F_{q^2}$ a finite extension, and $u \in E$, the trace at time $u$ is given by
$$u \in E \mapsto  -\sum_{z \in E}\psi_E(u z^B - z^A)\rho(z).$$
This sheaf $\sG_\rho$ is geometrically isomorphic to the Kummer pullback $[A]^\star \sH_{big,A,B,\chi,descent}$ for any choice of
$\chi$ with $\chi^A=\rho$. Moreover, the sheaf $\sG_\rho$ has all its traces in the field $\Q(\rho)$.
\end{prop}
\begin{proof}To get the existence, choose any $\chi$ with $\chi^A =\rho$, and apply Corollary \ref{Apullbackbis1}. The traces a priori 
lie in $\Q(\zeta_p,\rho)$. To see that they lie in $\Q(\rho)$, we must show that for any $\lambda \in \F_p^\times$, and any $u \in E/\F_{q^2}$ we have the identity
$$\sum_{z \in E}\psi_E(\lambda u z^B - \lambda z^A)\rho(z)=\sum_{z \in E}\psi_E(u z^B - z^A)\rho(z).$$
In fact, this will hold for any $\lambda \in \F_q^\times$, by the substitution $z \mapsto \lambda z$, because  both $A,B \equiv 1 \mod (q-1)$ and because $\lambda$ is the $(q+1)^{\mathrm {th}}$ power of some element of $\F_{q^2}$ by surjectivity of the norm.
\end{proof}

\begin{prop}\label{totaltraces-1}There is a lisse sheaf $\sG_\triv$ on $\A^1/\F_p$ whose trace function is as follows: for $E/\F_p$ a finite extension, and $u \in E$, the trace at time $u$ is given by
$$u \in E \mapsto  -\sum_{z \in E}\psi_E(u z^B - z^A).$$
This sheaf $\sG_\rho$ is geometrically isomorphic to the Kummer pullback $[A]^\star \sH_{small,A,B,descent}$. It has all its traces in $\Q$ {\rm (}again because both $A,B \equiv 1 \mod (q-1)${\rm )}.
\end{prop}

\begin{prop}\label{total-su}
The direct sum $\bigoplus_{\rho \in \Char(q+1)} \sG_\rho$ is geometrically isomorphic to the arithmetically semisimple lisse sheaf
$$\sW_\SU$$
on $\A^1/\F_p$ whose trace function at time $u \in E/\F_p$ is given by
$$u \in E \mapsto  -\sum_{z \in E}\psi_E(u z^{q^m+1} - z^{q^n+1}).$$
\end{prop}

\begin{cor}Each of the lisse sheaves $\sG_\rho$ is geometrically irreducible.
\end{cor}
\begin{proof}Here are two proofs. Because $\gcd(q^n+1,q^m+1)=q+1$, $\sW_\sG$ has second moment $q+1$ (by Lemma \ref{ABmoment}). Being the sum of the $q+1$ summands $\sG_\rho$, each summand must be geometrically irreducible.
Alternatively, one could use the fact that $\gcd(A,B)=1$, and apply Proposition \ref{ABmomentrho} to each $\sG_\rho$.
\end{proof}

\begin{thm}\label{su-det}On $\A^1/\F_{q^4}$, each of the lisse sheaves $\sG_\rho \otimes (1/q^2)^{\deg}$ has arithmetically trivial determinant. In characteristic $2$, this is already true for the lisse sheaves $\sG_\rho \otimes (1/q)^{\deg}$ on $\A^1/\F_{q^2}$.
\end{thm}
\begin{proof}From Lemmas \ref{detdown1} and \ref{ABparity}, we know that each has geometrically trivial determinant.
So it suffices to show that at a single point $u \in \F_{q^4}$, the determinant is $1$. We take the point $u=0$. According to \cite[2.3, parts (3) and (4)]{KT1}, the determinant for $\sG_\triv$ at the point $u=0$ viewed in $\F_{q^2}$ is $(q^2)^{(A-1)/2}=q^{A-1}$, and the determinant of $\sG_\rho$ for $\rho$ nontrivial is $(-\Gauss_{\F_{q^2}}(\overline{\psi}_A,\rho))(q^2)^{(A-1)/2}$. By Stickelberger's theorem \cite[11.6.1]{Be-Ev-Wi}, for $\rho$ nontrivial of order $r $ dividing $q+1$, one has
\begin{equation}\label{stick}
  \Gauss_{\F_{q^2}}(\overline{\psi}_A,\rho) = \left\{ \begin{array}{rl}(-1)^{(q+1)/r}q, & q{\rm\ odd},\\ q, & q {\rm\ even}.\end{array}\right.
\end{equation}  
So over $\F_{q^4}$, these determinants are respectively $(q^2)^{A-1}$ and $(q^2)^A$. Hence after the $\otimes (1/q^2)^{\deg}$ twist,
all these determinants become $1$. And in characteristic $2$, we are already okay on $\A^1/\F_{q^2}$.
\end{proof}

\begin{cor}\label{det-stick}Suppose $p$ is odd, $q$ a power of $p$, and $\rho$ a nontrivial character of $\F_{q^2}^\times$ of order $r$ dividing $q+1$. Then the arithmetic determinant of 
$$\sG_\rho \otimes (-\Gauss_{\F_{q^2}},\chi_2)^{-\deg/\F_{q^2}}$$
is trivial if $(-1)^{(q+1)/r} =(-1)^{(q+1)/2}$, and is $(-1)^{-\deg/\F_{q^2}}$ otherwise.
\end{cor}
\begin{proof}For any $\psi$ which begins life over $\F_p$, the Gauss sum $\Gauss(\psi_{\F_{q^2}},\rho)$ is independent of the choice of $\psi$ (because every element of $\F_p^\times$ becomes a square in $ \F_{q^2}$), and its square is $q^2$. Thus the arithmetic determinant of $\sG_\rho\otimes (-\Gauss_{\F_{q^2}}(\psi,\rho))^{-\deg/\F_{q^2}}$ is trivial. If instead we look at
 $$\sG_\rho\otimes (-\Gauss_{\F_{q^2}}(\psi,\chi_2))^{-\deg/\F_{q^2}},$$
 the factor by which we have twisted differs, thanks to Stickelberger  \cite[11.6.1]{Be-Ev-Wi}, by the ratio
 $$(-\Gauss_{\F_{q^2}}(\psi,\rho))/(-\Gauss_{\F_{q^2}}(\psi,\chi_2)) =(-1)^{(q+1)/r}/(-1)^{(q+1)/2}.$$
 Let us denote $\epsilon :=(-1)^{(q+1)/r}/(-1)^{(q+1)/2}$.
 As $\sG_\rho$ has rank $A$, which is odd (being $1 \mod q-1$), this second twist will have arithmetic determinant $\epsilon^{-\deg/\F_{q^2}}$.
\end{proof}

Let us denote by 
$$\sW_\SU(1/2)$$
the sheaf on $\A^1/\F_{p^2}$ obtained by twisting $\sW_\SU$ by $-\chi_2(-1)/p$ for $p$ odd, and by $1/p$ for $p=2$.
\begin{thm}\label{tracevalues1}We have the following results for the sheaf $\sW_\SU(1/2)$ on $\A^1/\F_{p^2}$.
\begin{itemize}
\item[\rm(i)] Over any extension of $\F_{q^2}$, all traces lie in $\Q$.
\item[\rm(ii]If $p$ is odd, then over any extension of $\F_{q^2}$, all traces are $\pm$(powers of $q$). If $p=2$, the trace $0$ may also occur.
\item[\rm(iii)] The arithmetic and geometric monodromy groups of $\sW_{\SU}(1/2)$ {\rm (}and hence also of each $\sG_\rho(1/2)${\rm )} are finite.
\end{itemize}
\end{thm}
\begin{proof}To see that (i) holds, notice that for $\lambda \in \F_q^\times$, there exists $\tau \in \F_{q^2}$ with $\tau^{q+1}=\lambda$, by surjectivity of the norm. Because $n,m$ are both odd, we have $\tau^q=\tau^{q^m}=\tau^{q^n}$. Apply this to $\lambda$ which lies in $\F_p^\times$. Making
the substitution $z \mapsto \tau z$ has the effect of replacing $\psi(x)$ by $\psi(\lambda x)$, without changing the sum, which therefore lies in $\Q$. to prove (ii), it suffices, given (i), to show that the square absolute values of traces are powers of $q^2$.
This holds because in the van der Geer-van der Vlugt approach, when $p$ is odd, each square absolute value is the cardinality of a vector space over $\F_{q^2}$, cf. \cite[beginning of \S5]{KT2} with $t=0$ to see this in the case $m=1$. When $p=2$, the trace is the sum of the values of an additive character on such a vector space, so the sum is either $0$, if the additve character is nontrivial, or it is the dimension of the vector space.  The integrality of the traces then implies the finiteness of the arithmetic and geometric monodromy groups, cf. \cite[\S5]{KT2}.
\end{proof}


\begin{prop}\label{prop-Sbis}For $A,B$ as above, i.e. 
$$A=(q^n+1)/(q+1),~B=(q^m+1)/(q+1),~n>m>0,~2 \nmid nm,$$ 
and $\chi$ a character with $\chi^A \neq \triv$, both the sheaves $\sH_{small,A,B,descent}$ and $\sH_{big,A,B,\chi,descent}$ satisfy property {\rm ({\bf S+})}.
\end{prop}
\begin{proof}The fact that $\gcd(A,B)=1$ shows that neither is Kummer induced. For each, the wild part has dimension
$A-B =(q^n -q^m)/(q+1)$, which is divisible by $q^m$, hence also by $p$, so neither is Belyi induced (by Proposition \ref{ABBelyi1}). Thus both are primitive. 

We first treat $\sH_{big,A,B,\chi,descent}$. Its rank $A$ is prime to $p$, and the dimension $W$ of the wild part is $A-B$. By \cite[1.2.3]{KT4}, provided that the rank $A \ge 4$ and $A\neq 8,9$, it suffices to show that 
$$A-B > A/2,$$
i.e., that (after multiplying both sides by $q+1$)
$$q^n -q^m > (q^n +1)/2, {\rm\  i.e., that \  \ }2q^n-2q^m >q^n+1,{\rm\  i.e., that \  \ }q^n-2q^m >1,$$ 
i.e., that
$$q^m(q^{n-m}-2) >1.$$
Because $m \ge 1$ and $n-m \ge 2$ (both being odd), we have
$$q^m(q^{n-m}-2) \ge q(q^2-2) \ge 2(2^2-2)=4 >1.$$
The least possible values of $A$ are attained by $(q=2,n=3), (q=3,n=3)$, all other possible $A$ are $\ge 11$ (attained by $(q=2,n=5)$).
For  $(q=2,n=3)$, we have $A=3$, which is prime, and hence property {\rm ({\bf S+})}. holds. Similarly, for $(q=3,n=3)$, we have $A=7$, again prime (or, an acceptable value).

We now treat $\sH_{small,A,B,descent}$. Here the rank is $A-1=(q^n-q)/(q+1)$ which is divisible by $p$, and the wild part has dimension $A-B$. By \cite[1.2.6]{KT4}, provided that the rank $A-1> 4$ and $A-1\neq 8,9$, it suffices to show that
$$A-B >(2/3)(A-2),$$
i.e., that (after multiplying both sides by $q+1$)
$$q^n-q^m >(2/3)(q^n+1 -2(q+1)), {\rm\  i.e., that \  \ }3q^n-3q^m > 2q^n +2 -4q-4,  {\rm\  i.e., that \  \ }q^n-3q^m >-2q-2.$$
In fact, we have $q^n-3q^m \ge q$. Indeed,
$$q^m-3q^m =q^b(q^{n-m}-3) \ge q(q^2-3) \ge q >0.$$
The least possible values of $A-1$ are attained by $(q=2,n=3), (q=3,n=3)$, all other possible $A-1$ are $\ge 10$ (attained by $(q=2,n=5)$). For $(q=2,n=3)$, we have $A-1=2$, which is prime. For $(q=3,n=3)$, we have $A-1=6$, which is an acceptable value.
\end{proof}

\begin{rmk}\label{coprime}There is one situation in which the sheaves $\sG_\rho$ are the $[A]^\star$ pullbacks of canonically chosen hypergeometric sheaves, 
namely the case when $\gcd(n,q+1)=1$. The key observation is that $A \equiv n \mod (q+1)$, cf. Lemma \ref{ABparity}. In defining the sheaves $\sH_{big,A,B,\chi,descent}$, we can choose for $\chi$ a nontrivial character of order dividing $q+1$. For these $\chi$, $\chi \mapsto \chi^A =\chi^n$ is simply a permutation of the nontrivial elements of $\Char (q+1)$. Moreover, in choosing $\alpha, \beta$ with $\alpha A-\beta B=1$, we may change $(\alpha,\beta)$ to $(\alpha + nB,\beta +nA)$ for any integer $n$. Since $A$ is invertible mod $q+1$, we may impose on $\beta$ any congruence mod $q+1$ that we like. We will impose that $\beta$ is divisible by $q+1$, and write
$$\beta = (q+1)\gamma$$
for some $\gamma \in \Z$.
With this choice, the trace function of  $\sH_{big,A,B,\chi,descent}$ is simply
$$\begin{aligned}u \in E^\times & \mapsto -\sum_{z \in E}\psi_E(Au^\alpha z^B -Bu^\beta z^A)\chi(u^\beta z^A)\\
    & =-\sum_{z \in E}\psi_E(Au^\alpha z^B -Bu^\beta z^A)\chi^n(z),\end{aligned}$$
because $\chi$ has order dividing $q+1$, $\beta \equiv 0 \mod (q+1)$, and $\chi^A=\chi^n$.

Recall that $A,B \equiv 1 (\bmod\ q)$. So in this $\gcd(n,q+1)=1$ case the direct sum
\begin{equation}\label{eq:510a}
 \bigoplus_{\chi \in \Char(q+1),\chi \neq \triv}\sH_{big,A,B,\chi,descent} \oplus \sH_{small,A,B,descent}  
\end{equation} 
has trace function
\begin{equation}\label{eq:510aa}
  u \in E^\times \mapsto -\sum_{z \in E}\psi_E(u^\alpha z^{q^m+1}-u^\beta z^{q^n+1}).
\end{equation}  
This formula is  analogous to that of Corollary \ref{totalztrace}, where $A$ was odd and $\beta$ could always be taken even. Moreover,
the $[A]^\star$ Kummer pullback of $\sH^{n,m}$ has trace function
\begin{equation}\label{eq:510b}
\begin{aligned}u \in E^\times & \mapsto - \sum_{z \in E}\psi_E(u^{\alpha A}z^{q^m+1}-u^{\beta A}z^{q^n+1})\\ 
 & = -\sum_{z \in E}\psi_E(u^{\alpha A -\beta B} u^{\beta B}z^{(q+1)B}-u^{\beta A} z^{(q+1)A})\\
 & = -\sum_{z \in E}\psi_E(uu^{\gamma (q+1) B}z^{(q+1)B}-u^{\beta A} z^{(q+1)A})\\
 & = -\sum_{z \in E}\psi_E(u(u^\gamma z)^{(q+1)B}-u^{\beta A-\gamma (q+1)A} (u^\gamma z)^{(q+1)A})\\
  & =- \sum_{z \in E}\psi_E(uz^{q^m+1}-z^{q^n+1} ),
    \end{aligned}
\end{equation}    
    the last equality by the substitution $z \mapsto z/u^\gamma$.
\end{rmk}

\begin{rmk}\label{coprimebis}There is another situation in which there are canonical choices of hypergeometric sheaves whose  $[B]^\star$ pullbacks have interesting trace functions, namely the situation in which $\gcd(m,q+1)=1$, or equivalently $\gcd(B,q+1)=1$. In this case, we can choose integers $(\alpha,\beta)$ with $\alpha A -\beta B=1$ and impose any congruence condition we like on $\alpha$ (because $B$ is invertible mod $q+1$, cf. the remark above). We impose that $\alpha$ is divisible by $q+1$, and write
$$\alpha =(q+1)\delta$$
for some integer $\delta$.
Then from Lemma \ref{ztracesbis1} we have that for any nontrivial character $\chi$ of order dividing $q+1$, the lisse sheaf
$\sH^\sharp_{big,A,\chi,B,descent}$ on $\G_m/\F_{q^2}$ has trace function as follows: for $E/\F_{q^2}$ a finite extension, and $u \in E^\times$, the trace at time $u$ is given by
$$u \in E^\times \mapsto -\sum_{z \in E}\psi_E(u^\alpha z^B -u^\beta z^A)\chi^B(z).$$
Since $\chi \mapsto \chi^B$ is a bijection on the set of nontrivial characters of order dividing $q+1$, with ${\sf G} :=-\Gauss$ 
for either choice $\Gauss$ of quadratic Gauss sum over $\F_p$, the direct sum
\begin{equation}\label{eq:511a}
 \sH_{bis}^{n,m}:= \biggl( \bigoplus_{\chi \in \Char(q+1),\chi \neq \triv}\sH^\sharp_{big,A,\chi,B,descent} \oplus \sH_{small,A,B,descent} \biggr) \otimes
 {\sf {G}}^{-\deg} 
\end{equation} 
has trace function
\begin{equation}\label{eq:511b}
u \in E^\times \mapsto  \frac{1}{\Gauss(\psi_E,\chi_2)}\sum_{z \in E}\psi_E(u^\alpha z^{q^m+1}-u^\beta z^{q^n+1}).
\end{equation} 
[This formula looks very much like that of \eqref{eq:510aa}, but the reader must remember that the exponents $(\alpha,\beta)$ of $u$ are different in the two cases; also, $\sH^{n,m}_{bis}$ has a clearing factor ${\sf G}^{\deg}$.] 
The $[B]^\star$ pullback $\sW^{\, n,m}_{bis}$ of $ \sH_{bis}^{n,m}$ has trace function
\begin{equation}\label{eq:511bb}
u \in E^\times \mapsto  \frac{1}{\Gauss(\psi_E,\chi_2)}\sum_{z \in E}\psi_E( z^{q^m+1}-u^{-1}z^{q^n+1}),
\end{equation} 
by an argument similar to that in Remark \ref{coprime}.
\end{rmk}

In the general case of any odd $n$, we can build the sheaves $\sH_{big,A,B,\chi}$ using the following choice of characters:

\begin{lem}\label{choice}
Let $n \in \Z_{\geq 1}$ be odd and let $q$ be any prime power. If $\gcd(n,q+1) = 1$, define $n_0:=1$. If $\gcd(n,q+1) > 1$,
let $p_1, \ldots, p_s$ be the distinct prime divisors of $\gcd(n,q+1)$, $p_i^{c_i}$ be the $p_i$-part of $n$, and 
define $n_0:=\prod^s_{i=1}p_i^{c_i} = \gcd(n,(q+1)^n)$. If $\nu$ is a multiplicative character of order $n_0(q+1)$, then
$\nu^A$ has order $q+1$; and conversely any character of order $q+1$ can be obtained this way.
\end{lem} 

\begin{proof}
The case $\gcd(n,q+1)=1$ has already been explained in Remark \ref{coprime}, so we will assume $\gcd(n,q+1)>1$. 
First we prove that
\begin{equation}\label{eq:cp11}
  \gcd(n_0(q+1),A) = n_0.
\end{equation}  
Let $p_1, \ldots,p_s,p_{s+1}, \ldots ,p_t$ be the distinct prime divisors of $q+1$ (for some $t \geq s$). If $1 \leq i \leq s$,
then $p_i|n$ and so it is odd but divides $q+1$, whence the $p_i$-part of $A=(q^n+1)/(q+1)$ is exactly $p_i^{c_i}$. To see this, use the fact, applied successively to $x:=-q$ and to an odd prime $p_i:=\ell$, that if  $x \in \Z$ has $\ord_\ell(x-1)=d>0$, then $\ord_\ell(x^{\ell}-1)=d+1$, while for an integer $m$ prime to $\ell$, $\ord_\ell(x^m -1)=d$, cf. also
\cite[Lemma 4.4]{KlT}. In particular, this shows that $n_0|A$ and $p_i \nmid (A/n_0)$. If $s+1 \leq i \leq t$, then $p_i \nmid n$ and $p_i|(q+1)$, but $A \equiv n (\bmod\ q+1)$, and so $A$
is coprime to $p_i$, whence $p_i \nmid (A/n_0)$ in this case as well. Thus \eqref{eq:cp11} is proved.

Now let $\nu$ have order $n_0(q+1)$. Then $\nu^{n_0}$ has order $q+1$. Since $\gcd(A/n_0,q+1)=1$ by \eqref{eq:cp11}, 
$\nu^A = (\nu^{n_0})^{A/n_0}$ has order $q+1$. The converse also follows by the same argument.
\end{proof}

\section{Extensions of Weil representations of finite symplectic groups}
In this section, we consider a non-degenerate symplectic space $W = \F_p^{2N}$ for a fixed prime $p>2$, a 
(reducible) total Weil representation of degree $p^N$ of $\Gamma = \Sp(W) \cong \Sp_{2N}(p)$ with character $\omega_{N,p}$ as in \cite{KT2}; in particular,
\begin{equation}\label{eq:value1}
  |\omega_{N,p}(g)| = |\CB_W(g)|^{1/2}
\end{equation}  
for any $g \in \Gamma$. We will also take $N = nf$ and $q = p^f$ for some positive integers $N,f$. We may then assume that $W$ is 
obtained from the symplectic space $W_1:=\F_q^{2n}$ (with a Witt basis $(e_1, \ldots, e_n,f_1, \ldots ,f_n)$) by base change
from $\F_q$ to $\F_p$. Using this basis, for any divisor $e|f$ we can consider the transformation 
$$\sigma_r:\sum^n_{i=1}(x_ie_i+y_if_i) \mapsto \sum^n_{i=1}(x_i^re_i+y_i^rf_i)$$ 
induced by the Galois automorphism $x \mapsto x^r$ for $r := q^{1/e}$. Then $\sigma_e$ belongs to $\Gamma$ and 
induces a field automorphism of order $e$ of $L := \Sp(W_1) \cong \Sp_{2n}(q)$. 
In what follows, we will refer to $L \rtimes \langle \sigma_r$ as a {\it standard subgroup $\Sp_{2n}(q) \rtimes C_e$} of $\Gamma$.
Also, let $\bj$ be the central involution of $\Gamma$.

We will denote by $\omega_{N,q}$ the restriction of
the total Weil character $\omega_{N,p}$ to $L$, and let the character $\xi_{n,q}$, respectively $\eta_{n,q}$, denote the irreducible summand of $\omega_{N,q}$ of degree $(q^n+1)/2$, respectively $(q^n-1)/2$. 
The  following statement clarifies the behavior of total Weil characters under 
embeddings $\Sp_{2n}(2n,q) \hookrightarrow \Sp_{2nm}(q)$ (by a similar base change as above); cf. \cite[Lemma 4.3]{KT2} for the 
case of irreducible Weil characters.

\begin{lem}\label{emb}
Let $q$ be an odd prime power and let $n, m \geq 1$. For fixed total Weil 
representations 
$$\Phi: \Sp_{2n}(q^m) \to \GL_{q^{nm}}(\C),~~\Lambda: \Sp_{2nm}(q) \to \GL_{q^{nm}}(\C),$$
affording character $\omega_{n,q^m}$, respectively $\omega_{nm,q}$,
there exists an embedding $\Theta: \Sp_{2n}(q^m) \to \Sp_{2nm}(q)$  
such that the representations $\Phi$ and $\Lambda \circ \Theta$ of $\Sp_{2n}(q^m)$ are equivalent. 
\end{lem}

\begin{proof}
Fix an embedding $\iota$ of $X:=\Sp_{2n}(q^m)$ into $\Sp_{2nm}(q)$. 
For $\eps = \pm$, let $\Phi^\eps$, respectively $\Lambda^\eps$, denote unique the irreducible constituent 
of $\Phi$, respectively of $\Lambda$, of degree $(q^{nm} + \eps)/2$. As is well-known, 
$\Lambda^\eps \circ \iota$ is an irreducible Weil representation of $X$ of degree $\deg(\Phi^\eps)$, and so
there is an outer diagonal automorphism $\alpha^\eps$ of $X$ such that 
$\Lambda^\eps \circ \iota \circ \alpha^\eps \cong \Phi^\eps$. If $\alpha^+$ and $\alpha^-$ belong to the same coset 
of ${\mathrm {Inn}}(X)$ in $\Aut(X)$, then we may in fact assume that $\alpha^+ = \alpha^-$ and 
take $\Theta = \iota \circ \alpha^+$. Otherwise, since any non-inner diagonal automorphism of $X$ fuses the two irreducible
Weil characters of any given degree, we see by \cite[Lemma 2.6(iii)]{TZ2} that 
$$\Tr(\Phi^+(t)) = - \Tr(\Phi^-(t))$$
for any transvection $t \in X$, and so $\omega_{n,q^m}(t) = \Tr(\Phi(t)) = 0$, contradicting \eqref{eq:value1}.
\end{proof} 

We will need the following slight extension of \cite[Lemma 3.1]{NT}:
 
\begin{lem}\label{value1}
Let $e$ be an odd integer, and consider the subgroup $L \rtimes \langle \sigma_r \rangle$ of $\Gamma$. Then 
$$\xi_{N,p}(\sigma_r) = (r^n+1)/2,~\eta_{N,p}(\sigma_r) = (r^n-1)/2.$$ 
\end{lem}

\begin{proof}
Note that $B : = \langle e_1, \ldots,e_n\rangle_{\F_q}$ 
and $B^* := \langle f_1, \ldots,f_n \rangle_{\F_q}$ are complementary
maximal totally isotropic $\F_p$-subspaces of $W$, both fixed by $\sigma_r \in \Gamma$. Then $\omega_{N,p}$ is the
character afforded by the reducible Weil representation of $\Gamma = \Sp(W)$ as constructed in 
\cite[\S13]{Gross1} using the decomposition $W = B \oplus B^*$. Let 
$\left(\frac{\cdot}{p}\right)$ denote the Legendre symbol on $\F_p$. Then the character value of 
$\omega_{N,p}$ at any element $g \in \Stab_{\Gamma}(B,B^*)$ is given in 
\cite[(13.3)]{Gross1}:
$$\omega_{N,p}(g) =  \left(\dfrac{\det {(g|_B)}}{p}\right) \cdot |\CB_B(g)|.$$ 
Recall that $\sigma_r$ has {\bf odd} order $e$, and 
$$\left(\dfrac{\det {((\sigma_r)|_B)}}{p}\right)^e =  \left(\dfrac{\det {(((\sigma_r)^e)|_B)}}{p}\right) = 1.$$
Hence $\left(\frac{\det {((\sigma_r)|_B)}}{p}\right) = 1$ and so, since, $|\CB_B(\sigma_r)| = r$, we have 
$$\xi_{N,p}(\sigma_r)+\eta_{N,p}(\sigma_r) = r^n.$$
For the central involution $\bj$ of $\Sp(W)$ we have that $\det(\bj|_B) = (-1)^N$, and so 
$\left(\frac{\det {((\bj\sigma_r)|_B)}}{p}\right) = \left(\frac{-1}{p}\right)^N =: \kappa$. 
Furthermore, since $e$ is odd, the equation $x^r = -x$ has only one solution $x=0$ in the field $\F_q = \F_{r^e}$, which implies
that $|\CB_B(\bf\sigma_r)| = 1$. It follows that 
$$\xi_{N,p}(\bj\sigma_r)+\eta_{N,p}(\bj\sigma_r) = \kappa.$$
Note that, by \cite[Lemma 2.6(i)]{TZ2}, $\bj$ acts via multiplication by $\kappa$ in any irreducible Weil representation 
of degree $(p^N+1)/2$ and via multiplication by $-\kappa$ in any irreducible Weil representation of degree $(p^N-1)/2$. 
Therefore, 
$$\xi_{N,p}(\bj\sigma_r) = \xi_{N,p}(\sigma_r)\cdot\kappa,~~\eta_{N,p}(\bj\sigma_r) = -\eta_{N,p}(\sigma_r)\cdot\kappa.$$
It follows that $\xi_{N,p}(\sigma_r) = (r^n+1)/2$ and $\eta_{N,p}(\sigma_r) = (r^n-1)/2$, as stated. 
\end{proof}

\begin{lem}\label{cent-sp}
Let $q=p^f$ be a power of a prime $p > 2$, $n \in \Z_{\geq 1}$, and let $L := \Sp_{2n}(q)$ with $(n,q) \neq (1,3)$. Suppose
that $\Phi: G \to \GL_{q^n}(\C)$ is a faithful representation of a finite group $G \geq L$ with the following properties:
\begin{enumerate}[\rm(a)]
\item $\Phi$ is a sum of two representations, $\Phi^+$ of degree $(q^n-1)/2$ and $\Phi^-$ of degree $(q^n+1)/2$;
\item For $\eps = \pm$, $\K_\eps:= \Q(\Tr(\Phi^\eps(g)) \mid g \in G)$ is contained in $\K:= \Q(\sqrt{(-1)^{(p-1)/2}p})$; 
\item $\Phi^\eps|_L$ is irreducible for each $\eps = \pm$; and
\item For all $g \in G$, $|\Tr(\Phi(g))|^2$ is always a power of $p$.
\end{enumerate}
Then the following statements hold.
\begin{enumerate}[\rm(i)]
\item $\Phi|_L$ is a total Weil representation. 
\item $\CB_G(L) = \ZB(G)=C \times \ZB(L)$, where we have
$\ZB(L) = \langle \bj \rangle \cong C_2$, and $C$ can be chosen to act via scalars in $\Phi$. Furthermore, either 
\begin{enumerate}
\item[$(\alpha)$] $C= 1$ or $C_2$,  or 
\item[$(\beta)$] $p=3$, $C \in \{C_3,C_6\}$, $\K_{\eps} = \K$ and $p$ divides $|\det(\Phi^\eps(G))|$ for both 
$\eps = \pm$.
\end{enumerate}
\end{enumerate}
\end{lem}

\begin{proof}
(i) Using the well-known character table of $\Sp_2(q)$ when $n=1$ and \cite[Theorem 5.2]{TZ1} when $n \geq 2$, we see that 
$\Phi^\eps|_L$ is an irreducible Weil representation of $L$ for each $\eps = \pm$. Now if these two irreducible Weil representations 
do not come from the same total Weil representation, then for a transvection $t \in L$ we have by \cite[Lemma 2.6(iii)]{TZ2} that 
$\Tr(\Phi(t)) = 0$, contradicting (d). Hence $\Phi|_L$ is a total Weil representation. 

\smallskip
(ii) Consider any element $z \in \CB_G(L)$. By Schur's lemma, $\Phi^\eps(g) =  z_\eps \cdot \mathrm{Id}$ for some root of unity
$z_\eps \in \C$. As $\Phi = \Phi^+ \oplus \Phi^-$, $z \in \ZB(G)$ and so 
$$\CB_G(L) = \ZB(G).$$ 
By (b), $z_\eps \in \K$. Now, $\pm 1$ are the only roots of 
unity in $\K$, unless $p=3$. In the former case, we obtain that $\ZB(G) = C \times \ZB(L)$, where
$1 \leq C \leq C_2$ and 
$$\Phi(C_2 \times \ZB(L)) = \{ \diag(\pm \mathrm{Id}, \pm \mathrm{Id})\}.$$ 
Thus, assuming $\ZB(G) \not\leq C_2 \times \ZB(L)$, we must have that $p=3$, and $z_\eps$ is a root of unity of order $3$ or $6$ 
for some $\eps$ and some $z \in \CB_G(L)$.
In this case, 
$\K$ contains exactly $6$ roots of unity $\pm \zeta_3^k$, $0 \leq k \leq 2$. Next, condition (d) implies that 
$$|\omega(z)|^2=\frac{q^{2n}+1}{2} + \frac{q^{2n}-1}{4}(t + t^{-1})$$
is a power of $3$ for $t = z_+/z_-$, a $6^{\mathrm {th}}$ root of unity. As $q^n > 3$, this can happen only when $t= \pm 1$.
Thus both $z_+$ and $z_-$ now have order divisible by $3$.
It is easy to see that in this case we have $\ZB(G) = C \times \ZB(L)$, where $C_3 \leq C \leq C_6$ and
$$\Phi(C_6 \times \ZB(L)) = \{ \pm \zeta_3^k\cdot \diag(\mathrm{Id}, \pm \mathrm{Id}) \mid 0 \leq k \leq 2 \}.$$
In particular, we can find an element $t \in \CB_G(L)$ with $\Phi(t)=\zeta_3 \cdot \mathrm{Id}$. It follows that for each $\eps=\pm$
we have that $\det(\Phi^\eps(t)) = \zeta_3^\eps$ has order $3$ and so $p$ divides $|\det(\Phi^\eps(G))|$, and furthermore
$\K_\eps = \K$.

We have also shown that for any $z \in \ZB(G)$,
\begin{equation}\label{eq:z1}
  \Phi(z) = z_-\Phi(\bj^k)
\end{equation}
for some $k = 0,1$ and some $z_- \in \K^\times$; i.e. $z$ acts as a scalar in $\Phi$.  
\end{proof}

The main result of this section is the following theorem:

\begin{thm}\label{ext-sp}
Let $q=p^f$ be a power of a prime $p > 2$, $n \in \Z_{\geq 1}$, and let $L := \Sp_{2n}(q)$ with $(n,q) \neq (1,3)$. Suppose
that $\Phi: G \to \GL_{q^n}(\C)$ is a faithful representation of a finite group $G \rhd L$ with the following properties:
\begin{enumerate}[\rm(a)]
\item $\Phi$ is a sum of two representations, $\Phi^+$ of degree $(q^n-1)/2$ and $\Phi^-$ of degree $(q^n+1)/2$;
\item For all $g \in G$ and $\eps = \pm$, $\Tr(\Phi^\eps(g)) \in \K$, where $\K = \Q$ if $p=3$ and $2|f$, and $\K = \Q(\sqrt{(-1)^{(p-1)/2}p})$
otherwise;
\item $\Phi^\eps|_L$ is irreducible for each $\eps = \pm$; and
\item For all $g \in G$, $|\Tr(\Phi(g))|^2$ is always a power of $q$.
\end{enumerate}
Then $\Phi|_L$ is a total Weil representation. Furthermore, $G = CL$ and $\ZB(G) = C \times \ZB(L)$, where
$\ZB(L) = \langle \bj \rangle \cong C_2$, and either 
\begin{enumerate}
\item[$(\alpha)$] $|C| \leq 2$, or 
\item[$(\beta)$] $p=3$ divides $|\det(\Phi^\eps(G))|$ for each $\eps = \pm$, $2 \nmid f$, and $C \in \{C_3,C_6\}$.
\end{enumerate}
In all cases, $C$ can be chosen to act via scalars in $\Phi$.
\end{thm}

\begin{proof}
(i) The fact that $\Phi|_L$ is a total Weil representation and that $\CB_G(L)=\ZB(G)$ together with its structure are already proved 
in Lemma \ref{cent-sp}.

By Lemma \ref{emb}, we may assume that $L$ embeds in $\Gamma = \Sp(W) \cong \Sp_{2nf}(p)$ as the subgroup 
$\Sp(W_1)$ introduced above via base change, and so the character of $\Phi|_L$ is the restriction to $L$ of $\omega:= \omega_{nf,p}$.
We may also assume that $\Phi$, $\Phi^\eps$ extend to $\Gamma$, and denote them by the same symbols.
Also write $\xi$ and $\eta$ instead of $\xi_{nf,p}$ and $\eta_{nf,p}$.

Note that any element in $\NB_{\Gamma}(L)$ preserves the equivalence class of each of the Weil representations $\Phi^\eps|_L$,
hence it can only induce a field automorphism of $L$ (modulo $\mathrm{Inn}(L)$). The subgroup of all the field automorphisms of 
$L$ is cyclic of order $f$. Thus we may assume that there is some $e|f$ such that $G$ induces a subgroup of field automorphisms 
of $L$ of order $e$. In the notation introduced right before Lemma \ref{emb}, this means that the action of $G$ via conjugation on $L$
induces the same automorphism subgroup as of 
$$H := L \rtimes \langle \sigma_r \rangle < \Gamma.$$
In view of the above results, we are done if $e=1$. Assume the contrary: $e > 1$.

\smallskip
(ii) Replacing $G$ by a subgroup of index $2$ if necessary, first we rule out the case $e > 1$ is odd. Let $g \in G$ induce (via conjugation)
the same automorphism of $L$ as of $\sigma_r$. It follows that $\Phi^\eps(g)\Phi^\eps(\sigma_r)^{-1}$ centralizes $\Phi^\eps(L)$, and so 
by Schur's lemma we have
\begin{equation}\label{eq:phi1}
  \Phi^+(g) = \beta\Phi^+(\sigma_r),~\Phi^-(g) = \beta\alpha\Phi^-(\sigma_r)
\end{equation}  
for some $\alpha,\beta \in \C^\times$. As $\sigma_r$ has order $e$, we obtain that 
$$\Phi(g^e) = \beta^e \cdot \diag(\mathrm{Id},\alpha^e\mathrm{Id}).$$
By (b), $\beta^e,\alpha^e\beta^e \in \K$; in particular, $\alpha^e \in \K$. Using the oddness of $e$ and replacing $g$ by $g\bj$ if necessary, 
we may assume that either $\alpha = 1$, or $p=3$ and $\alpha = \zeta_3^{\pm 1}$. 
However, in the latter case by Lemma \ref{value1} we have that
$$|\omega(g)|^2 = |\Tr(\Phi(g))|^2 = |\xi(\sigma_r)+\alpha\eta(\sigma_r)|^2 = \bigl|\frac{r^n+1}{2} + \alpha\frac{r^n-1}{2}\bigr|^2 = 
    \frac{r^{2n}+3}{4},$$
which can be a $p$-power only when $r^n=3$. In particular, it can be a $q$-power only when
$r^n=3=q$, and this contradicts the assumption $(n,q) \neq (1,3)$. 

Thus we have shown that $\alpha=1$ in \eqref{eq:phi1}, i.e. $\Phi(g) = \beta\Phi(\sigma_r)$. On the other hand, by the choice of $g$, 
we have that $G = \langle L,\ZB(G),g \rangle$. As $H = \langle L,\sigma_r \rangle$, together with \eqref{eq:phi1}, this shows that 
$\Phi(G)$ and $\Phi(H)$ are the same up to scalar matrices in $\GL_{q^n}(\C)$. Now we apply \cite[Theorem 3.5]{KT3} to get 
an element $h \in H$ such that $|\Tr(\Phi(h))|^2 =r$. As we have just shown, $\Phi(h) = \gamma\Phi(\tilde g) $ for some $\gamma \in \C^\times$. Since $h$ and $\tilde g$ have finite order, $|\gamma| = 1$. Thus 
$$|\omega(\tilde g)|^2 = |\Tr(\Phi(\tilde g))|^2 = |\Tr(\Phi(h))|^2 = r,$$
contradicting (d).

\smallskip
(iii) It remains to rule out the case $e=2$. We again 
use the element $h \in H$ with $|\Tr(\Phi(h))|^2 = r = q^{1/2}$ constructed in \cite[Theorem 3.5]{KT3}. Let $g' \in G$ induce the same
automorphism of $L$ as of $h$. As in (iii), we can again write 
$$\Phi^+(g') = \beta'\Phi^+(h),~\Phi^-(g') = \beta'\alpha'\Phi^-(h)$$
for some $\alpha',\beta' \in \C^\times$. Since $g',h$ have finite order, $|\beta'| = |\alpha'|=1$.
Now we have 
$$r=|\Tr(\Phi(h))|^2 = |\xi(h)+\eta(h)|^2,$$
but 
$$|\omega(g')|^2 = |\Tr(\Phi(g'))|^2 = |\xi(h)+\alpha'\eta(h)|^2$$
is a power of $q$ by (d). This implies that $\xi(h),\eta(h) \neq 0$ and 
\begin{equation}\label{eq:phi2}
 \alpha' \neq 1.  
\end{equation} 
Since $\Q(\eta) = \Q(\xi) \subseteq \K':=\Q(\sqrt{(-1)^{(p-1)/2}p})$
for $\Gamma$, cf. \cite[Lemma 13.5]{Gross1}, it follows from (b) that 
\begin{equation}\label{eq:phi3}
  \alpha' \in \K'.
\end{equation}
In our case, $h^2 \in L$, and so $h^2 \in G$. It follows that
$$\Phi(h^2)^{-1}\Phi(g'^2) = \beta'^2 \cdot \diag(\mathrm{Id},\alpha'^2\mathrm{Id})$$
belongs to $\Phi(G)$ and centralizes $\Phi(L)$. In this case, 
\eqref{eq:z1} shows that $\alpha'^2 = \pm 1$. But then \eqref{eq:phi3} rules out the case $\alpha'^2 = -1$,
and so $\alpha' = -1$ by \eqref{eq:phi2}. Recalling $\Phi(\bj) = \pm \diag(\mathrm{Id},-\mathrm{Id})$, we now 
have 
$$\Phi(\bj g') = \pm \beta' \cdot \diag(\Phi^+(h),\Phi^-(h)),$$
whence
$$|\omega(\bj g')|^2 = |\Tr(\Phi(\bj g'))|^2 = |\xi(h)+\eta(h)|^2 = r,$$
again contradicting (d).
\end{proof}

Theorem \ref{ext-sp} will usually be used in tandem with the following Goursat-like theorem:

\begin{thm}\label{weil-sp}
Let $p \geq 3$ be a prime, $f,n \in \Z_{\geq 1}$, $q=p^f$, and let $(n,q) \neq (1,3)$. Let $\Phi:L \to \GL_{q^{n}}(\C)$ be a faithful 
representation of a finite group $L$, which is a sum of two irreducible representations $\Phi^\eps$ of degree
$(q^{n}-\eps)/2$, $\eps = \pm$, that satisfies the following two conditions:
\begin{enumerate}[\rm(a)]
\item For each $\eps = \pm$, $\Phi^\eps(L)$ is quasisimple;
\item For (at least) one value $\gam = \pm$, there are positive integers $n_\gam,a_\gam$ such that $nf = n_\gam a_\gam$ and 
$\Phi^\gam(L)$ is isomorphic to a quotient of $\Sp_{2n_\gam}(p^{a_\gam})$; and 
\item For each $g \in L$, $|\Tr(\Phi(g))|^2$ is a power of $q$.
\end{enumerate}
Then there is some divisor $d$ of $n$ such that $L \cong \Sp_{2n/d}(q^d)$, and furthermore $\Phi$ is a total Weil representation of $L$. 
\end{thm}

\begin{proof}
(i) For each $\eps = \pm$, we have that $L_\eps:=\Phi^\eps(L)$ is quasisimple, with cyclic center $\ZB(L_\eps)$ and simple 
quotient $S_\eps = L_\eps/\ZB(L_\eps)$. Let $K_\eps$ denote the kernel of $\Phi^\eps$. 
Interchanging $\eps$ with $-\eps$ if necessary, we may assume that $\gam=-$.
Now we have the exact sequences
$$1 \to K_+ \to L \stackrel{\Phi^+}{\longrightarrow} L_+ \to 1,~~1 \to K_- \to L \stackrel{\Phi^-}{\longrightarrow} L_- \to 1.$$
As $L_-$ is quasisimple and $K_+ \lhd L$, either $\Phi^-(K_+) = L_-$ or $\Phi^-(K_+) \leq \ZB(L_-)$. 

Assume we are in the former case. Then $\Phi$ maps $K_+$ isomorphically onto $L_-$, and so $K_+ \cong L_-$. 
Furthermore, $\Phi^-(L) = \Phi^-(K_+)$, which implies that $L = K_+K_-$. However, $K_+ \cap K_- = \Ker(\Phi)=1$, so 
we must have that $L = K_+ \times K_-$. We also have $K_- \cong L/K^+ \cong \Phi^+(L) \cong L_+$. Thus 
$$L = K_+ \times K_- \cong L_- \times L_+.$$
By Burnside's theorem, $\Phi^+(K_-) = \Phi^+(L)$ contains some element $g_- \in K_-$ with $\Tr(\Phi^+(g_-)) = 0$. Likewise,
$\Phi^-(K_+) = \Phi^-(L)$ contains some element $g_+ \in K_+$ with $\Tr(\Phi^-(g_+)) = 0$. As $g_\eps \in K_\eps$ acts trivially in
$\Phi^\eps$, for $g :=g_+g_- \in L$ we have 
$$\Tr(\Phi(g)) = \Tr(\Phi^+(g)) + \Tr(\Phi^-(g)) = \Tr(\Phi^+(g_-)) + \Tr(\Phi^-(g_+)) = 0,$$
contradicting (c).

We have shown that $\Phi^-(K_+) \leq \ZB(L_-)$; i.e. $\Phi^-(h)$ centralizes $\Phi^-(L)$ if $h \in K_+$.
Clearly, $\Phi^+(h) = \mathrm{Id}$ also centralizes $\Phi^+(L)$. Thus $\Phi(h)$ centralizes $\Phi(L)$, and so
$h \in \ZB(L)$ since $\Phi$ is faithful. It follows that $K_+ \leq \ZB(L)$.

\smallskip
(ii) Recall that $L/K_- = \Phi^\gam(L) = L_\gam$ is isomorphic to a quotient of $\Sp_{2n_\gam}(p^{a_\gam})$; in particular,
$S_\gam=\PSp_{2n_\gam}(p^{a_\gam})$ is a quotient of $L$. As $K_+ \leq \ZB(L)$, we see that $S_\gam$ is a non-abelian
composition factor of the quasisimple group $L_+= L/K_+$. It follows that $S_+ \cong S_\gam$. In the case 
$S_\gam \cong \PSp_2(9)$, note that, among central extensions of $\PSp_2(9)$, 
only quotients of $\Sp_2(9)$ can have irreducible representations of degree $4$ or $5$. 
Thus in all cases $L_\eps$ are quotients of $\Sp_{2n_\gam}(p^{a_\gam})$, and we can view 
$\Phi^\eps$ as an irreducible representation of $\Sp_{2n_\gam}(p^{a_\gam})$ of degree $(p^{a_\gam n_\gam} -\eps)/2$.
Using the well-known character table of $\Sp_2(p^{a_\gam})$ when $n_\gam=1$ and \cite[Theorem 5.2]{TZ1} when $n_\gam \geq 2$, 
we see that each $\Phi^\eps$ is an irreducible {\it Weil} representation of the quasisimple group $2S_\gam=\Sp_{2n_\gam}(p^{a_\gam})$.

Without loss of generality, we may assume that $\deg(\Phi^-)$ is odd, whence $\Phi^-(L)=L^- \cong S_\gam$ is simple, and the arguments in 
(i) show that $\Phi^-(K_+)=1$. Hence $\Phi(K_+) = 1$, and so $K_+ = 1$ since 
$\Phi$ is faithful. As $\deg(\Phi^+)$ is even, it follows that $L \cong \Phi^+(L) \cong 2S_\gam = \Sp_{2n_\gam}(p^{a_\gam})$.
Now if these two irreducible Weil representations $\Phi^\pm$ of $L$ 
do not come from the same total Weil representation of $L$, then for a transvection $t \in L$ we have by \cite[Lemma 2.6(iii)]{TZ2} that 
$\Tr(\Phi(t)) = 0$, again contradicting (c). Hence $\Phi|_L$ is a total Weil representation. 

Next, we choose $u \in L$ to be a regular unipotent element. Then $u$ acting on the natural module $\F_{p^{a_\gam}}^{2n_\gam}$ 
with exactly $p^{a_\gam}$ fixed points. Applying to $u$ the well-known character formula \cite[(3.4.5)]{KT2} for total Weil representations, 
we then obtain that $|\Tr(\Phi(u))|^2 = p^{a_\gam}$. Condition (c) implies that $p^{a_\gam}$ is a power $q^d$ of $q$, and we are done.
\end{proof}

The following statement is well known; we include the proof for the reader's convenience.

\begin{lem}\label{simple-product}
Let $S_1, \ldots,S_n$ be finite non-abelian simple groups. For each $i$, let $\pi_i$ denote the projection of 
$S_1 \times S_2 \times \ldots \times S_n$ onto the $i^{\mathrm {th}}$ component. Let $G \leq S_1 \times \ldots \times S_n$ 
be a subgroup such that $\pi_i(G)=S_i$ for all $i = 1,2, \ldots ,n$. Then there exists a subset $J$ of $\{1,2, \ldots ,n\}$ such that
$G$ is isomorphic to the direct product $\prod_{j \in J}S_j$.
\end{lem}

\begin{proof}
We induct on $n$, with the induction base $n=1$ being trivial. For the induction step $n \geq 2$, let $K_1 := \Ker\bigl((\pi_1)|_G\bigr)$ and
let $K_2$ be the kernel of the homomorphism 
$$\pi' := \pi_2 \times \pi_3 \times \ldots \times \pi_n: G \to S_2 \times S_3 \times \ldots \times S_n.$$
Then $K_1,K_2 \lhd G$ and $K_1 \cap K_2 = 1$. In particular, $\pi_1$ maps $K_2$ injectively onto a normal subgroup of $\pi_1(G)=S_1$.
By simplicity, $K_2 = 1$ or $\pi_1(K_2)=S_1$. In the former case, $\pi'$ is injective, hence we can consider $G$ as a subgroup of
$S_2 \times S_3 \times \ldots \times S_n$, and conclude by applying the induction hypothesis for $n-1$.

In the latter case, $\pi_1(K_2) = \pi_1(G)$, whence $G = K_1K_2 = K_1 \times K_2$ and $K_2 \cong S_1$. It also follows for each
$i \geq 2$ that $\pi_i(K_1) = \pi_i(G) = S_i$. As $K_1 \leq S_2 \times S_3 \times \ldots \times S_n$, we can again apply the induction hypothesis to $K_1$ to obtain $J' \subseteq \{2,3, \ldots,n\}$ such that $K_1 \cong \prod_{j \in J'}S_j$. Thus 
$G \cong \prod_{j \in J}S_j$ with $J := J' \cup \{1\}$.
\end{proof}

The next result generalizes Goursat's lemma:

\begin{prop}\label{goursat}
Let $G$ be a perfect finite group, and let $\Phi:G \to \GL_N(\C)$ be a faithful representation that satisfies the following conditions:
\begin{enumerate}[\rm(a)]
\item $\Phi=\oplus^n_{i=1}\Phi_i$ is a sum of $n$ irreducible constituents;
\item $L_i:=\Phi_i(G)$ is quasisimple, with simple quotient $S_i = L_i/\ZB(L_i)$.
\end{enumerate}
Then there is a subset $\{j_1, j_2, \ldots,j_m\}$ of $\{1,2,\ldots,n\}$ such that $G$ is isomorphic to a central product
$R_{j_1} * R_{j_2} * \ldots *R_{j_m}$, where each $R_{j_i}$ is a quasisimple cover of $S_{j_i}$. 

Suppose that, in addition to {\rm (a)--(b)}, $\Phi$ also satisfies the following two additional conditions:
\begin{itemize}
\item[(c)] $\Tr(\Phi(g)) \neq 0$ for all $g \in G$;
\item[(d)] For any quasisimple subgroup $H \leq G$ and for any proper subset $\sX \subset \{\Phi_1, \Phi_2, \ldots,\Phi_n\}$
with the property that $\Phi_i(H) = \Phi_i(G)$ for all $\Phi_i \in \sX$, there exists $h \in H$ such that $\Tr(\Phi_i(h)) = 0$ for all
$\Phi_i \in \sX$. {\rm [Note that, by Burnside's theorem, this condition is automatic if $n=2$.]}
\end{itemize}
Then $G$ is quasisimple.
\end{prop}

\begin{proof}
(i) Consider the natural projection
$$\Theta: L_1 \times L_2 \times \ldots \times L_n \to L_1/\ZB(L_1) \times L_2/\ZB(L_2) \times \ldots \times L_n/\ZB(L_n) = S_1 \times S_2
    \times \ldots \times S_n.$$
Next, $\Phi$ gives rise to an injective homomorphism
$$\tilde\Phi:g\in G \mapsto \bigl(\Phi_1(g),\Phi_2(g), \ldots,\Phi_n(g)\bigr) \in L_1 \times L_2 \times \ldots \times L_n.$$ 
Then $\Theta\tilde\Phi(G)$ is a subgroup of  
$S_1 \times S_2 \times \ldots \times S_n$ that projects surjectively onto each of the $n$ components. By Lemma \ref{simple-product}, 
there is a subset $J:=\{j_1, j_2, \ldots,j_m\}$ of $\{1,2,\ldots,n\}$ such that 
$$\Theta\tilde\Phi(G) \cong S_{j_1} \times S_{j_2} \times  \ldots  \times S_{j_m}.$$
Let $Z := \Ker(\Theta\tilde\Phi)$. Then $\tilde\Phi(Z) \leq \ZB(L_1) \times \ZB(L_2) \times \ldots \times \ZB(L_n)$ and so centralizes 
$\tilde\Phi(G)$. As $\tilde\Phi$ is faithful, $Z \leq \ZB(G)$. 

We have proved that there is a surjective homomorphism
$$\Sigma:G \twoheadrightarrow S_{j_1} \times S_{j_2} \times  \ldots  \times S_{j_m}$$
with kernel $Z\leq \ZB(G)$ (in fact, the simplicity of $S_j$ for all $j \in J$ then implies that $Z = \ZB(G)$). Let $j \in J$ and set 
$\tilde R_j := \Sigma^{-1}(S_j)$, $R_j := \tilde R_j^{(\infty)}$. Then it is easy to see that $\Sigma(R_j) = S_j = \Sigma(\tilde R_j)$ and so $\tilde R_j = ZR_j$.
This in turn implies that $R_j/(Z \cap R_j) \cong \tilde R_j/Z \cong S_j$, whence $R_j$ is a quasisimple cover of $S_j$. 
Since $\tilde R_j \lhd G$, we also have that $R_j \lhd G$. Consider any $i \in J$ with $i \neq j$. Then $\Sigma(R_i \cap R_j) = 1$ and 
so $R_i \cap R_j \leq Z$. In particular, $[R_i,R_j] \leq R_i \cap R_j$ centralizes $R_j$. It follows from the Three Subgroups Lemma
that $[R_i,R_j] = [[R_i,R_i],R_j]$ is in fact trivial. This implies that 
$R_{j_1}R_{j_2} \ldots R_{j_m}$ is a central product. Next,
$$\Sigma(ZR_{j_1}R_{j_2} \ldots R_{j_m}) = S_{j_1} \times S_{j_2} \times  \ldots  \times S_{j_m} = \Sigma(G),$$
implying $G = ZR_{j_1}R_{j_2} \ldots R_{j_m}$. As $G = G^{(\infty)}$, it follows that 
\begin{equation}\label{eq:g1}
  G=R_{j_1}R_{j_2} \ldots R_{j_m} = R_{j_1}* R_{j_2} * \ldots * R_{j_m}.
\end{equation}   
In particular, the first statement of the proposition follows.

\smallskip
(ii) For the second statement, note that, by \eqref{eq:g1}, for each $1 \leq i \leq n$ we can express
$$\Phi_i = \Psi_{i,1} \boxtimes \Psi_{i,2} \boxtimes \ldots \boxtimes \Psi_{i,m}$$
as an outer tensor product of $\Psi_{i,k} \in \Irr(R_{j_k})$, $1 \leq k \leq m$. It follows that $L_i = \Phi_i(G)$ is a central product 
$\Psi_{i,1}(R_{j_1})*\Psi_{i,2}(R_{j_2})* \ldots * \Psi_{i,m}(R_{j_m})$ of (normal) subgroups. Since $L_i$ is quasisimple and since
each $R_{j_k}$ is also quasisimple, we conclude that all but one $\Psi_{i,k}$ are trivial, say for all $k \neq k_i$. This implies that
$$L_i=\Phi_i(G) = \Psi_{i,k_i}(R_{k_i}) = \Phi_i(R_{k_i}).$$

On the other hand, the faithfulness of $\Phi$ implies that each $R_j$ with $j \in J$ must be acting nontrivially in some $\Phi_i$. So we 
can partition $\{\Phi_1, \Phi_2, \ldots,\Phi_n\}$ into a disjoint union $\sX_1 \sqcup \sX_2 \sqcup \ldots \sqcup \sX_m$ of non-empty subsets 
such that for each $1 \leq t \leq m$ and for all $\Phi_i \in \sX_t$ we have 
\begin{equation}\label{eq:g2}
  L_i = \Phi_i(G) = \Phi_i(R_{j_t})
\end{equation}  
but $\Phi_i(R_{j'})$ is trivial for all $j' \in J \smallsetminus \{j_t\}$.

Now if $m=1$, then $G$ is quasisimple, as stated. Suppose $m \geq 2$. Then $R_{j_t}$ is a quasisimple subgroup of $G$, and 
$|\sX_t| \leq n-1$. By \eqref{eq:g2} and assumption (d), there exists $x_t \in R_{j_t}$ such that $\Tr(\Phi_i(x_t)) = 0$ for all $\Phi_i \in \sX_t$. 
Setting $g:=x_1x_2 \ldots x_t$, we see that 
$$\Tr(\Phi_i(g)) = \Tr(\Phi_i(x_t)) = 0$$
for all $\Phi_i \in \sX_t$. It follows that $\Tr(\Phi(g)) = \sum^n_{i=1}\Tr(\Phi_i(g)) = 0$, contradicting (c).
\end{proof}

\section{Going-up and going-down}
First we prove the following going-up result. (Notice the difference between this and Theorem \ref{ext-sp}: in the latter we assumed 
that $L$ is a {\it normal} subgroup of $G$.) Usually, we apply this result with $(G,e)=(\tilde G,1)$, in which case condition (e) is equivalent 
to (d).

\begin{thm}\label{sp-up}
Let $q=p^f$ be a power of a prime $p > 2$, $n \in \Z_{\geq 1}$, and let $H := \Sp_{2n}(q)$ with $q^n \geq 9$. Suppose
that $\Phi: \tilde G \to \GL_{q^n}(\C)$ is a faithful representation of a finite group $\tilde G \geq H$ with the following properties:
\begin{enumerate}[\rm(a)]
\item $\Phi$ is a sum of two representations, $\Phi^+$ of degree $(q^n-1)/2$ and $\Phi^+$ of degree $(q^n-1)/2$;
\item For all $g \in \tilde G$ and $\eps = \pm$, $\Tr(\Phi^\eps(g)) \in \K$, where $\K = \Q$ if $p=3$ and $2|f$, and 
$\K = \Q(\sqrt{(-1)^{(p-1)/2}p})$ otherwise;
\item $\Phi^\eps|_H$ is irreducible for each $\eps = \pm$; and
\item For some subgroup $G \leq \tilde G$ that contains $\tilde G^{(\infty)}$, $|\Tr(\Phi(x))|^2$ is always a power of $q$ for all $x \in G$.
\item There exists some divisor $e$ of $f$ such that $|\Tr(\Phi(y))|^2$ is always a power of $q^{1/e}$ for all $y \in \tilde G$. If $e>1$,
assume in addition that there exists some $g \in \tilde G$ such that $|\Tr(\Phi(g))|^2=q^{1/e}$ and that, if $p=3$ then 
$p \nmid |\det(\Phi^\eps(\tilde G))|$ for $\eps = \pm$.
\end{enumerate}
Then $H \lhd \tilde G$, $G = \CB_G(H)H$ and $\CB_{\tilde G}(H) = \ZB(\tilde G)=C \times \ZB(H)$, where
$\ZB(H) = \langle \bj \rangle \cong C_2$, and either 
\begin{enumerate}[\rm(i)]
\item $C= 1$ or $C_2$,  or 
\item $e=1$, $p=3$ divides $|\det(\Phi^\eps(\tilde G))|$ for each $\eps = \pm$, $2 \nmid f$, and $C \in \{C_3,C_6\}$.
\end{enumerate}
Moreover, $\tilde G$ induces a field automorphism of order $e$ of $H$, and $\tilde G/\CB_{\tilde G}(H) \cong \PSp_{2n}(q) \rtimes C_e$.
\end{thm}

\begin{proof}
Let $L := \tilde G^{(\infty)} = G^{(\infty)}$. 
Since $H$ is perfect, $L \geq H$, and so $\Phi^\eps(L)$ contains $\Phi^\eps(H)$, which in turn contains 
$\Phi^\eps(H_1)$ with $H_1 \cong \SL_2(q^n)$ a subgroup of $H$. Since $L$ is perfect, $\Phi^\eps(L)$ lands in 
$\SL_{(q^n+\eps)/2}(\C)$. Now, applying Theorems 4.1 and 4.2 of \cite{KT2} to $\Phi^\eps:L \to \SL_{(q^n+\eps)/2}(\C)$, for each
$\eps = \pm$ we obtain a pair $(n_\eps,a_\eps)$ with $nf = n_\eps a_\eps$ and $\Phi^\eps(L)$ is isomorphic to a quotient of 
$\Sp_{2n_\eps}(p^{a_\eps})$. By Theorem \ref{weil-sp}, $L \cong \Sp_{2n/d}(q^d)$ for some divisor $d|n$. But $L \geq H=\Sp_{2n}(q)$,
hence $H=L$ by order comparison. Thus $H \lhd \tilde G$, and the statements about $G$ and $\CB_{\tilde G}(H)$ follow from Theorem \ref{ext-sp} and Lemma \ref{cent-sp}. We are also done if $e=1$ (by taking $G = \tilde G$).

\smallskip
Consider the case $e > 1$. Since no outer-diagonal automorphism of $H$ can preserve the equivalence class of $\Phi^\eps|_H$, 
$\tilde G$ can only induce inner and field automorphisms of $H$, whence there exists some $e'|f$ such that 
$\tilde G/\CB_{\tilde G}(H) \cong \PSp_{2n}(q) \rtimes C_{e'}$, where the subgroup $C_e$ is generated by the field automorphism 
$\sigma$ induced by the Galois automorphism $x \mapsto x^{q^{1/e'}}$. 
By Lemma \ref{emb}, we may embed $H$ in $\Gamma:=\Sp_{2nf}(p)$ and extend $\Phi$ to a total Weil representation 
$\Gamma \to \GL_{q^n}(\C)$. As noted in \cite[\S3]{KT3}, there is a standard subgroup
$$\tilde H = \Sp_{2n}(q) \rtimes C_{e'}$$
of $\Gamma$ with $[\tilde H,\tilde H] = H$ that induces the same automorphism subgroup of $H$ as the one induced by $\tilde G$.  
In particular, for any element $x \in \tilde G$, there is an element $x^* \in \tilde H$ such that conjugations by $x$ and $x^*$ induce
the same automorphism of $H$. By Schur's lemma, for each $\eps = \pm$ there is $\alpha_\eps(x) \in \C^\times$ such that
\begin{equation}\label{eq:up11}
  \Phi^\eps(x) = \alpha_\eps(x)\Phi^\eps(x^*).
\end{equation}
Note that $\CB_{\tilde H}(H) = \ZB(H) = \langle \bj \rangle C_2$, so $x^*$ is unique up to a power of $\bj$. We will show that
\begin{equation}\label{eq:up12}
  \alpha_-(x),\alpha_+(x) \in \{1,-1\}.
\end{equation}  
Indeed, as $\tilde H$ is contained in the perfect group $\Gamma$, $\det(\Phi^\eps(x^*))=1$. Now, by assumption (e),
$$\alpha^\eps(x)^{(q^n -\eps)/2} = \det\bigl( \alpha_\eps(x)\Phi^\eps(x^*)\bigr) = \det \bigl( \Phi^\eps(x) \bigr)$$
has order coprime to $p$ when $p=3$. Thus 
\begin{equation}\label{eq:up13}
  p\nmid |\alpha_\eps(x)| \mbox{ when }p=3. 
\end{equation}  
Next, by \cite[Lemma (8.14)(c)]{Is}, since $\Phi^\eps|_H$ is irreducible, for each $\eps=\pm$
there is some $y_\eps =h_\eps x \in Hx$ such that $\Tr(\Phi^\eps(y_\eps)) \neq 0$. Then, by the above remark,  if $y^*_\eps$ is 
chosen to fulfill \eqref{eq:up11} for $y_\eps$, then it is equal to $h_\eps x^*$ up to a power of $\bj$. As $\Phi^\eps(\bj) = \pm \mathrm{Id}$, 
we have
$$\Phi^\eps(h_\eps)\Phi^\eps(x) = \Phi^\eps(y_\eps) = \alpha_\eps(y_\eps)\Phi^\eps(y^*_\eps) =
    \pm \alpha_\eps(y_\eps)\Phi^\eps(h_\eps x^*) = \pm \alpha_\eps(y_\eps)\Phi^\eps(h_\eps)\Phi^\eps(x^*).$$
Comparing this to \eqref{eq:up11}, we obtain $\alpha_\eps(x) = \pm \alpha_\eps(y_\eps)$. On the other hand,
taking traces, we see that $0 \neq \Tr(\Phi^\eps(y_\eps) = \alpha_\eps(y_\eps)\Tr(\Phi^\eps(y^*_\eps))$, and furthermore,
$\Tr(\Phi^\eps(y_\eps) \in \K_1:=\Q(\sqrt{(-1)^{(p-1)/2}p})$ by (b) and $\Tr(\Phi^\eps(y^*_\eps)) \in \K_1$ by \cite[Lemma 13.5]{Gross1}
applied to Weil characters of $\Gamma=\Sp_{2nf}(p)$. It follows that
$\alpha_\eps(x) = \pm \alpha_\eps(y_\eps)$ is  a root of unity in $\K_1$, whose order is coprime to $p$ when $p=3$ by \eqref{eq:up13}.
Hence \eqref{eq:up12} follows.

Now, if $\alpha_-(x)=\alpha_+(x)$, then \eqref{eq:up11} implies that $\Phi(x)=\alpha_+(x)\Phi(x^*)$; set $x^\sharp :=x^*$
in this case. If $\alpha_-(x)=-\alpha_+(x)$, then taking $x^\sharp := \bj x^*$, we again have that $\Phi(x)=\pm\alpha_+(x)\Phi(x^\sharp)$.
Thus in all cases, given any $x \in \tilde G$, there is (a unique) $x^\sharp \in \tilde H$ such that 
\begin{equation}\label{eq:up14}
  \Phi(x) = \pm \Phi(x^\sharp), \mbox{ and conjugations by }x \mbox{ and }x^\sharp\mbox{ induce the same automorphism of }H.
\end{equation}  
Applying this result to the element $g \in \tilde G$ in (e), we see that $|\Tr(\Phi(g^\sharp))|^2 = q^{1/e}$. By \cite[Theorem 3.5]{KT3},
$q^{1/e}$ is a power of $q^{1/e'}$.

Again by \cite[Theorem 3.5]{KT3}, $\tilde H$ contains an element $h$ such that  $|\Tr(\Phi(h)|^2 = q^{1/e'}$. 
Since $\tilde H$ and $\tilde G$ induce the same automorphism subgroup of $H$ and since 
$\CB_{\tilde H}(H) = \langle \bj \rangle$, by \eqref{eq:up14} there exists some $k \in \{0,1\}$ and some $s \in \tilde G$ such 
that $\bj^k h = s^\sharp$. Also by \eqref{eq:up14}, 
$$\Phi(\bj^ks) = \Phi(\bj^k)\Phi(s) = \pm \Phi(\bj^k)\Phi(s^\sharp) = \pm \Phi(\bj^k)\Phi(\bj^k h) = \pm \Phi(\bj^{2k})\Phi(h) = \pm \Phi(h),$$
in particular, $|\Tr(\Phi(\bj^ks)|^2 = q^{1/e'}$. But $\bj^k s \in \tilde G$, hence $q^{1/e'}$ is a power of $q^{1/e}$. We conclude 
that $q^{1/e} = q^{1/e'}$, i.e. $e=e'$, as stated.
\end{proof}

Next we prove a going-down result:

\begin{thm}\label{sp-down}
Let $p > 2$ be a prime, $N \in \Z_{\geq 1}$, $p \geq 13$ if $N = 1$, and $(p,N) \neq (3,2)$, $(3,3)$, $(5,3)$. Consider 
a total Weil representation $\Phi: \Gamma \to \GL_{p^N}(\C)$ of $\Gamma := \Sp_{2N}(p)$ and extend it to 
$$\Phi: \tilde\Gamma := C \times \Gamma \to \GL_{p^N}(\C),$$
where $C$ is a finite cyclic group and acts faithfully via scalars in $\Phi$. Suppose $G$ is a finite subgroup of $\tilde\Gamma$ 
with the following properties:
\begin{enumerate}[\rm(a)]
\item Each of the two irreducible components $\Phi^\eps$, of degree $(p^N-\eps)/2$ for $\eps = \pm$, of $\Phi$ is irreducible over 
$L:=G^{(\infty)}$, with $\Phi^\eps(L)$ being quasisimple; and
\item For all $x \in G$, $|\Tr(\Phi(x))|^2$ is always a power of $q=p^f$.
\end{enumerate}
Then $f|N$, and there exist a divisor $d$ of $N/f$ and a divisor $e$ of $d$ such that 
$$\Sp_{2N/df}(q^d) \cong L \lhd G \leq C \times \bigl(\Sp_{2N/df}(q^d) \rtimes C_e\bigr) = CG,$$ 
with the subgroup $\Sp_{2N/df}(q^d) \rtimes C_e$ identified in $\Gamma$ as in {\rm \cite[\S4]{KT2}}, and 
$G$ inducing a subgroup of order $e$ of outer field automorphism of $L$. Moreover, if there exists 
$g \in G$ with $|\Tr(\Phi(g))|^2=q$, then $e=d$.
\end{thm}

\begin{proof}
Note that $L \leq \Gamma$ and $\Phi^\eps(\Gamma) < \SL_{(p^N-\eps)/2}(\C)$ since $\Gamma$ is perfect; 
furthermore, for all $x \in \Gamma$ we have
$\Tr(\Phi^\eps(x)) \in  \Q(\sqrt{(-1)^{(p-1)/2}p})$, cf. \cite[Lemma 13.5]{Gross1}.
By \cite[Theorem 4.7]{KT2} applied to the irreducible subgroup $\Phi^-(L)$ of  $\SL_{(p^N+1)/2}(\C)$, one of the following must occur:

(i) There is a factorization $N = AB$, 
a divisor $b|B$, and a standard subgroup $H := \Sp_{2A}(p^B) \rtimes C_b$ of $\Gamma$ such that $\Phi^-(L) = \Phi^-(H)$. 

(ii) $p=3$, $2 \nmid N$, and $\Phi^-(L)$ contains $\SU_N(3)$ as a proper normal subgroup of $2$-power index.

As $\Phi^-(L)$ is perfect, (i) must hold and moreover $b=1$. In particular, $L$ projects onto the simple group $\PSp_{2A}(p^B)$.
Applying Theorem \ref{weil-sp}, we conclude that there exists some $d \in \Z_{\geq 1}$ such that 
$$B=df,~~N = AB= Adf,$$ 
and $L \cong \Sp_{2N/df}(q^d)$; in particular, $f|N$ and $d|(N/f)$.  Using the equality
$\Phi^-(L) = \Phi^-(H)$ and the inclusion $\Ker(\Phi^-) \cap \Gamma \leq \ZB(\Gamma) = \ZB(H)$, we see that $L = H$. 
Thus $L$ is the standard subgroup $\Sp_{2A}(p^B)$, with normalizer $C \times \bigl(\Sp_{2A}(p^B) \rtimes C_B\bigr)$ in $\tilde\Gamma$,
that induces the full group (of order $B$) of outer field automorphisms of $L$.

Since $L \lhd G \leq \NB_{\tilde\Gamma}(L)$, there is some $e|B$ that $G$ induces a subgroup $C_e$ of outer field automorphisms of
$L$. As $\CB_{\tilde \Gamma}(L) = \ZB(L)C$, in this case we have 
$$G \leq C \times \bigl(\Sp_{2N/df}(q^d) \rtimes C_e\bigr) = CG.$$
By \cite[Theorem 3.5]{KT3}, we can find $h \in \Sp_{2N/df}(q^d) \rtimes C_e$ such that $|\Tr(\Phi(h))|^2 = p^{B/e}$. As $h \in CG$,
we can find $z \in C$ such that $zh \in G$. But $\Phi(C)$ is scalar, so $|\Tr(\Phi(g))|^2 = |\Tr(\Phi(h))|^2 = p^{B/e}$, and 
condition (b) implies that $e|d$.

Assume now that $|\Tr(\Phi(g))|^2 = q$ for some $g \in G$. Then $g=z_1h_1$ for some $z_1 \in C$ and 
$h_1 \in \Sp_{2N/df}(q^d) \rtimes C_e$. This again implies that $|\Tr(\Phi(h_1))|^2 = |\Tr(\Phi(g))|^2 = p^f$.
By \cite[Theorem 3.5]{KT3}, $p^f$ is a power of $p^{B/e}$, i.e. $B/e=df/e$ divides $f$, and we conclude that $e=d$ in this case. 
\end{proof}

\section{Local systems and total Weil representations: Symplectic groups over $\F_p$}\label{p12}
Fix a prime $p > 2$ and $N \geq 3$. In this section, we will work with the local system $\sG^{u,r,s,t}$ on 
$\G_m \times \A^3/\F_p$ whose trace function is given as follows: for $k/\F_p$ a finite extension, and $(u,r,s,t) \in k^\times \times k^3$,
$$(u,r,s,t) \mapsto \frac{1}{\Gauss(\overline\psi_{k},\chi_2)}\sum_{x \in k}\psi_{-1/2,k}\bigl(-ux^{p^N+1} + rx^{p^2+1}+sx^{p+1}+tx^2\bigr),$$
and its various specializations, say $\sG^{1,r,0,0}$ obtained when we take $u=1$, $s=0$, and $t=0$. Then $G^{u,r,s,t}_{\ari}$ and 
$G^{u,r,s,t}_{\geo}$ denote the arithmetic and the geometric monodromy groups of $\sG^{u,r,s,t}$, and similarly,
$G^{1,r,0,0}_{\ari}$ and $G^{1,r,0,0}_{\geo}$ denote the arithmetic and the geometric monodromy groups of $\sG^{1,r,0,0}$.

\begin{thm}\label{main-sp1}
Over any finite extension $k$ of $\F_p$, the following statements hold.
\begin{enumerate}[\rm(i)]
\item 
$G^{-1,0,s,t}_{\geo}$ equals $L=\Sp_{2N}(p)$ in 
one of its total Weil representations. If $2|N$, then we also have $G^{-1,0,s,0}_{\geo}= L$. Furthermore, 
$$C_0 \times L=G^{-1,r,s,t}_{\ari} \geq G^{-1,r,s,t}_{\geo} \rhd L,$$ 
where $C_0$ is a cyclic scalar subgroup, and either $|C_0| \leq 2$, or $p=3$ and $|C_0|$ divides $6$. Moreover, 
$$C \times L=G^{u,r,s,t}_{\ari} \geq G^{u,r,s,t}_{\geo} \rhd L,$$ 
where $C$ is a cyclic scalar subgroup, and either $|C|=1,2$, or $p=3$ and $|C|$ divides $6$.
\item Assume $2 \nmid N$. Then each of $G^{u,1,0,0}_{\geo}$ and $G^{u,1,0,0}_{\ari}$ 
contains the normal subgroup $L=\Sp_{2N}(p)$ acting in 
one of its total Weil representations, and furthermore, is equal to $C' \times L$ for a cyclic scalar subgroup $C'$ of order $\leq 2$.
\end{enumerate}
\end{thm}

\begin{proof}
(i) First we choose $k$ to contain $\F_{p^2}$, so that any element of $\F_p^\times$ is a square in $k$. 
In this case, $\Gauss(\psi_k,\chi_2)=\Gauss((\psi_a)_k,\chi_2)$ for any $\psi_a: t \mapsto \psi(at)$ with $a \in \F_p^\times$.
It follows that  
$\sG^{-1,0,s,t}$ is the local system $\sW_{2{\tiny \mbox{-param}}}(\psi_{-1/2},N,p)$
introduced in \cite[\S4]{KT3} when $2\nmid N$ and in \cite[\S9]{KT3} when $2|N$. Hence $G^{-1,0,s,t}_{\geo} = L$ by Theorem 4.3 and Theorem 10.3 of \cite{KT3}. Similarly, when $2|N$, we have $G^{-1,0,s,0}_{\geo} = L$ by \cite[Theorem 10.6]{KT3}.

Now we return to work with any extension $k$ of $\F_p$. Then $G^{-1,0,s,t}_{\ari}$ contains $G^{-1,0,s,t}_{\geo} = L$ as a 
normal subgroup.
Since $G^{-1,r,s,t}_*$ contains $G^{-1,0,s,t}_*$ and is finite (with $* = ${\small {\it arith}} or 
{\small {\it geom}}), 
applying Theorem \ref{thm:vdG-vdV} and the second moment two result (Proposition \ref{ABmoment}), we see that
the second statements in (i) follows from Theorem \ref{sp-up} (with $\tilde G = G = G^{-1,r,s,t}_*$ and $H = L=\Sp_{2N}(p)$). 
The same argument also applies to $G^{u,r,s,t}_*$.

\smallskip
(ii) Let $\Phi:G^{u,r,s,t}=CL \to \GL_{p^N}(\C)$ denote the corresponding representation of $G^{u,r,s,t}$ acting on $\sG^{u,r,s,t}$. By 
Corollary \ref{multiABmomentrho}, $\Phi$ is 
a sum of two irreducible representations $\Phi^\eps$ of degree $(p^N -\eps)/2$, $\eps = \pm$. 

Now we aim to determine $G:=G^{u,1,0,0}_{\geo} \leq CL$, which is irreducible in both $\Phi^\eps$ by Proposition \ref{ABmoment}. 
Recall that each of the two irreducible summands of $\sG^{u,1,0,0}$ is the Kummer pullback by $u \mapsto u^A$, $A = (p^2+1)/2$, of 
one of the two irreducible hypergeometric summands $\sH^{u,1,0,0,\eps}$ of rank $(p^N - \eps)/2$, $\eps = \pm$, which both satisfy 
$(\mathbf{S}+)$ by Proposition \ref{cond-S}.  [These two hypergeometric sheaves were denoted by 
$\sH_{small,A,B,descent}$ and $\sH_{big,A,B,descent}$ in Corollary \ref{pullback}, with $(A,B) := ((p^2+1)/2,(p^N+1)/2)$.]
Hence we can apply \cite[Proposition 2.8]{G-T} to its geometric monodromy group $H^{u,1,0,0,\eps}_{\geo}$ which is finite.

Consider some $\eps=\pm$ and assume we are in the extraspecial case of \cite[Proposition 2.8(iii)]{G-T}.
Then $(p^N-\eps)/2 = (p_1)^m$ for some prime $p_1$, and $H^{u,1,0,0,\eps}_{\geo}$ contains a normal $p_1$-subgroup $P_1$ 
that acts irreducibly on the sheaf $\sH^{u,1,0,0,\eps}$. As $N \geq 3$ is odd, $p_1 \nmid A$. On the other hand, $G^{u,1,0,0,\eps}_{\geo}$ is a normal subgroup of  $H^{u,1,0,0,\eps}_{\geo}$ of index dividing $A$. It follows that $P_1 \lhd G^{u,1,0,0,\eps}_{\geo}=\Phi^\eps(G)$. Recall that 
$G \leq CL = C \times L$ with $L = \Sp_{2N}(p)$. Since $(p^N-\eps)/2 = (p_1)^m$, we see that Sylow $p_1$-subgroups of $L$ are abelian 
(in fact cyclic). Hence Sylow $p_1$-subgroups of $G$ are abelian, and so $P_1$ is abelian. But this contradicts the irreducibility of $P_1$ on 
$\sH^{u,1,0,0,\eps}$.

Thus we have shown that $H^{u,1,0,0,\eps}_{\geo}$ is almost quasisimple for each $\eps = \pm$. Using property $(\mathbf{S}+)$ and 
\cite[Lemma 2.5]{G-T}, we then have that $\Phi^\eps(G^{(\infty)})$ is a quasisimple irreducible subgroup of $\SL_{(p^N-\eps)/2}(\C)$. 
Furthermore, $G^{(\infty)} \leq (CL)^{(\infty)} = L=\Sp_{2N}(p)$.
%
%
By Theorem \ref{sp-down}, there are some divisors $d|N$ and $e|d$ such that $G^{(\infty)} = \Sp_{2N/d}(p^d) \rtimes C_e$, whence $e=1$ by perfectness. (Note that Theorem \ref{sp-down} assumes $p > 5$ when $N=3$. However, when $N=3$, since
$|H^{u,1,0,0,\eps}_{\geo}|$ is divisible by $(p^3-\eps)/2$ for each $\eps = \pm$, it is easy to see that $|G^{(\infty)}|$ is divisible by 
$13 \cdot 7$ when $p=3$ and by $31 \cdot 7$ when $p=5$. Using the list of maximal subgroups of $\Sp_6(p)$ 
\cite[Tables 8.28, 28.29]{BHR}, we see that the same conclusion holds for $p=3,5$.) In particular,
$H^{u,1,0,0,\eps}_{\ari} \rhd H^{u,1,0,0,\eps}_{\geo}$ 
contains the normal quasisimple subgroup $\Phi^\eps(\Sp_{2N/d}(p^d))$. It also contains (the image) of the inertia subgroup $I(0)$, which has a cyclic $p'$-subgroup $\langle h \rangle$ of order divisible by $(p^{N-2}-1)/2$ 
that cyclically permutes the $(p^{N-2}-1)/2$ 
irreducible $P(0)$-submodules of dimension $p^2$ by Proposition \ref{cond-S}.

Next we choose $\eps_0 = \pm$
such that $D_0:=(p^N-\eps_0)/2$ is even. Since no outer-diagonal automorphisms of $\Sp_{2N/d}(p^d)$ can preserve the Weil 
representation $\Phi^{\eps_0}(\Sp_{2N/d}(p^d))$ up to equivalence, by Schur's lemma we have
\begin{equation}\label{eq:d1}
  \Sp_{2N/d}(p^d) \lhd H^{u,1,0,0,\eps_0}_{\ari} \leq \NB_{\GL_{D_0}(\C)}(\Sp_{2N/d}(p^d)) \leq (\Sp_{2N/d}(p^d) \cdot C_{d})Z,
\end{equation}  
where $Z = \ZB(\GL_{D_0}(\C))$.

Consider the case $N \geq 5$. Then $p^{N-2}-1$ admits a primitive prime divisor $\ell$ by \cite{Zs}, which is either equal to
$N-1$ or at least $2(N-2)+1 > N$. In either case, $\ell$ is coprime to $N$ but divides $|h|$. 
Let $h_0$ denote the $\ell$-part of $h$. Now using $\ell \nmid d$ and \eqref{eq:d1}, we see
that $h_0 \in \Sp_{2N/d}(q^d)Z$. Since $h_0$ acts nontrivially on the set of $(p^{N-2}-1)/2$ irreducible $P_0$-submodules
in $\sH^{u,1,0,0,\eps_0}$, we conclude 
that $h_0 \notin Z$ and $\ell$ divides $|\Sp_{2N/d}(p^d)|$. 
Thus there exists $1 \leq i \leq N/d$ such that $\ell|(p^{2di}-1)$, whence $N-2$ divides $2id$ by the 
choice of $\ell$. As $2 \nmid N \geq 5$, it follows that $N-2=id$. Hence $d$ divides both $N-2$ and $N$, and we conclude that $d=1$.

Next we consider the case $N=3$ but $d > 1$. Then $d=3$. Let $Q$ denote the image of $P(0)$ in $H^{u,1,0,0,\eps_0}_{\ari}$. Then \eqref{eq:d1} shows that $Q$ has a normal subgroup $Q_1$ of index dividing $d=3$, where $Q_1 \in \Syl_p(\Sp_2(p^3)Z)$ is abelian.
It follows from Ito's theorem \cite[(6.15)]{Is} that any irreducible $\C Q_1$-module has dimension dividing $3$. But this contradicts 
the fact that $P(0)$ possesses an irreducible submodule of dimension $p^2$ on $\sH^{u,1,0,0,\eps_0}$.

Thus we have shown that $d=1$ and so $G^{(\infty)} = \Sp_{2N}(p) = L$. Clearly, $G^{(\infty)}$ is a normal subgroup of each of 
$G^{u,1,0,0}_{\geo}$, $G^{u,1,0,0}_{\ari}$. Furthermore, by Theorem \ref{sp-det}, $\det(H^{u,1,0,0,\eps}_{\ari})$ has order a $2$-power
(dividing $4$; note that we use the oddness of $p$ here to deduce the normality of arithemetic monodromy groups of $[2]^\star$ pullbacks). 
It follows that $\det(\Phi^\eps(G^{u,1,0,0}_{\ari}))$ is also a $2$-group.
The statement now follows from Theorem \ref{ext-sp}.
\end{proof}

For later use, in the case $N > N' > 2$, we also need to consider the local system $\tilde\sG^{v,r,s,t}$ on 
$\A^4/\F_p$ whose trace function, for $k/\F_p$ a finite extension and $(v,r,s,t) \in k^4$, is given by
$$(v,r,s,t) \mapsto \frac{1}{\Gauss(\overline\psi_k,\chi_2)}\sum_{x \in k}\psi_{-1/2,k}\bigl(x^{p^N+1} + vx^{p^{N'}+1}+rx^{p^2+1}+sx^{p+1}+tx^2\bigr).$$
 Let $\tG^{v,r,s,t}_{\ari}$ and $\tG^{v,r,s,t}_{\geo}$ denote its arithmetic and geometric monodromy groups,
respectively.

\begin{thm}\label{main-sp1b}
Over any finite extension $k$ of $\F_p$, we have 
$$\tilde C \times L=\tG^{v,r,s,t}_{\ari} \geq \tG^{v,r,s,t}_{\geo} \rhd L,$$ 
where $L=\Sp_{2N}(p)$ in one of its total Weil representations, $\tilde C$ is a cyclic scalar subgroup,
and either $|\tilde C|=1,2$, or $p=3$ and $|\tilde C|$ divides $6$.
\end{thm}

\begin{proof}
Note that $\tilde\sG^{0,0,s,t}$ is the local system $\sG^{-1,0,s,t}$ considered in Theorem \ref{main-sp1}. Hence, 
$$\tG^{v,r,s,t}_{\ari} \rhd \tG^{v,r,s,t}_{\geo} \geq G^{-1,0,s,t}_{\geo} = L = \Sp_{2N}(p).$$
Using the finiteness, Proposition \ref{ABmoment}, and Theorem \ref{thm:vdG-vdV}, we see that
the statements follow from Theorem \ref{sp-up}. 
\end{proof}

\section{Local systems and total Weil representations: Symplectic groups over $\F_q$}
We continue to work with the prime $p > 2$, and fix a power $q=p^f$ and positive integers $n,m$, where 
\begin{equation}\label{eq:mn1}
  n > m,~~\gcd(n,m) = 1,~~2|mn,~~q^n > 9,~~\mbox{and either }m< n/2,\mbox{ or }(n,m) = (3,2),\;(2,1).
\end{equation}  
This assumption implies that 
\begin{equation}\label{eq:mn12}
  \gcd(q^n+1,q^m+1) = 2.
\end{equation}  

For compatibility with the notations used in $\S4$, we recall that precisely one of $n,m$ is even, and we define the integers $A,B$ as follows:
\begin{equation}\label{eq:ab11}
  (A,B) = \left\{ \begin{array}{ll}\bigl( (q^n+1)/2, (q^m+1)/2\bigr), & \mbox{if }2|n,\\
  \bigl((q^m+1)/2, (q^n+1)/2\bigr), & \mbox{if }2|m. \end{array} \right.
\end{equation}  

In this section, our ultimate target is the local system $\sW(\psi,n,m,q)$ on $\A^1/\F_p$, whose trace function for $k/\F_p$ a finite extension and $r \in k$ is given as follows:
\begin{equation}\label{eq:sheaf1}
  r \mapsto\frac{1}{\Gauss(\overline\psi_k,\chi_2)} \sum_{x \in k}\psi_{-1/2,k}\bigl(x^{q^n+1} + rx^{q^m+1}\bigr).
\end{equation}  

To study $\sW(\psi,n,m,q)$, we first study the local system $\stW^{u,r}$ on $\G_m \times \A^1/\F_p$
whose trace function is given as follows. For $k/\F_p$ a finite extension, and $(u,r)\in k^\times \times k$, 
$$(u,r) \mapsto \frac{1}{\Gauss(\overline\psi_k,\chi_2)}\sum_{x \in k}\psi_{-1/2,k}\bigl(-ux^{q^n+1} - rx^{q^m+1}\bigr).$$
By Proposition \ref{ABmomentrho} and \eqref{eq:mn12}, $\stW^{u,r}$ is the sum of two irreducible subsystems of rank $(q^n \pm 1)/2$.
Let $\tG^{u,r}_{\ari}$ and 
$\tG^{u,r}_{\geo}$ denote the arithmetic and the geometric monodromy groups of $\stW^{u,r}$, and similarly,
$\tG^{u,r,\eps}_{\ari}$ and $\tG^{u,r,\eps}_{\geo}$ denote the arithmetic and the geometric monodromy groups of each of the 
two irreducible subsystems $\stW^{u,r,\eps}$, of rank $(q^n-\eps)/2$ for $\eps = \pm$.
  

Now, let $\sW(n,m)$ denote the local system defined as follows. When $2|n$, i.e. when $A >B$ in \eqref{eq:ab11}, it is the local system $\stW^{-1,r}$ on $\A^1/\F_p$, that is, the one with trace function 
given as follows: for $k/\F_p$ a finite extension, and $r \in k$,
$$r \mapsto \frac{1}{\Gauss(\overline\psi_k,\chi_2)} \sum_{x \in k}\psi_{-1/2,k}\bigl(x^{q^n+1} - rx^{q^m+1}\bigr).$$
When $2 \nmid n$, i.e. when $A < B$ in \eqref{eq:ab11}, it is 
the local system $\stW^{u,-1}$ on $\G_m/\F_p$, 
with trace function given as follows: for $k/\F_p$ a finite extension, and $u \in k^\times$,
$$u \mapsto  \frac{1}{\Gauss(\overline\psi_k,\chi_2)} \sum_{x \in k}\psi_{-1/2,k}\bigl(-ux^{q^n+1} + x^{q^m+1}\bigr).$$
For $\eps=\pm$, let $\sW(n,m,\eps)$ denote the irreducible subsystem of $\sW(n,m)$ of rank $(q^n-\eps)/2$.
Then $\sW(n,m,\eps)$ is the $[A]^\star$ Kummer pullback  of the hypergeometric sheaf $\sH(n,m,\eps)$ defined by
\begin{equation}\label{eq:ab12}
  \begin{array}{lr}\sH(n,m,+) := & \sH_{small,A,B,descent}\otimes \bigl(-\Gauss(\overline\psi_{\F_p},\chi_2)\bigr)^{-\deg},\\
    \sH(n,m,-) := & \sH_{big,A,B,descent}\otimes \bigl(-\Gauss(\overline\psi_{\F_p},\chi_2)\bigr)^{-\deg}.\end{array}
\end{equation}    

Let $G(n,m)_{\ari}$, $G(n,m)_{\geo}$, $G(n,m,\eps)_{\ari}$, $G(n,m,\eps)_{\geo}$, $H(n,m,\eps)_{\ari}$, and 
$H(n,m,\eps)_{\geo}$ denote the arithmetic and geometric monodromy groups of the local systems
$\sW(n,m)$, $\sW(n,m,\eps)$, and $\sH(n,m,\eps)$, respectively.
The pullback relation implies that $G_{\geo} \lhd H_{\geo}$ and the quotient is a cyclic group of order dividing $A$, 
for a pair of respective geometric monodromy groups $G_{\geo}$ and $H_{\geo}$.


\begin{thm}\label{main-sp2}
Given the assumption \eqref{eq:mn1}, and over any finite extension $k$ of $\F_q$, the following statements hold.
Each of $G(n,m)_{\geo}$ and $G(n,m)_{\ari}$ 
contains the normal subgroup $M=\Sp_{2n}(q)$ acting in 
one of its total Weil representations, and furthermore, is of the form $C' \times M$ 
for a suitable cyclic scalar subgroup $C'$ of order $\leq 2$.
\end{thm}

\begin{proof}
(i) Write $q=p^f$, define $N:=nf$, $N':=mf$, 
and choose $\kappa := 1$ if $2|N$ and $\kappa:= 2$ if $2 \nmid N$. First we 
consider the local system $\stW^{u,r,s}$ on $\G_m \times \A^2/\F_p$, with arithmetic monodromy group $\tG^{u,r,s}_{\ari}$, and  
with trace function given as follows: for $(u,r,s)\in k^\times \times k^2$,
$$(u,r,s) \mapsto  \frac{1}{\Gauss(\overline\psi_k,\chi_2)}  \sum_{x \in k}\psi_{-1/2,k}\bigl(-ux^{q^n+1} - rx^{q^m+1}+sx^{p^\kappa+1}\bigr)=$$
$$= \frac{1}{\Gauss(\overline\psi_k,\chi_2)}  \sum_{x \in k}\psi_{-1/2,k}\bigl(-ux^{p^N+1} - rx^{p^{N'}+1}+sx^{p^\kappa+1}\bigr).$$
When $2|N$, the system $\stW^{-1,0,s}$ at $u=1$ and $r=0$ is exactly 
the system $\sG^{-1,0,s,0}$ considered in \S\ref{p12}. Likewise, when $2 \nmid N$, the system $\stW^{u,0,1}$ at $r=0$ is exactly 
the system $\sG^{u,1,0,0}$ considered in \S\ref{p12}. 
It follows from Theorem \ref{main-sp1} (applied to $\stW^{1,0,s}$, respectively
$\stW^{u,0,1}$) that $(\tG^{u,r,s}_{\ari})^{(\infty)}$ contains $L = \Sp_{2N}(p)$ acting in one of its total Weil representations. 

By Theorem \ref{thm:vdG-vdV}, the sheaf $\stW^{u,r,s}$ and its various specializations satisfy the $p$-power property 
for the entire sheaf and the property of having all traces belonging to $\K$ for the two irreducible subsheaves. Moreover,
their monodromy groups satisfy the second moment $2$ property, as follows from Corollary \ref{multiABmomentrho}. In the subsequent
arguments, we will repeatedly use these properties without recalling them explicitly again.
Now, Theorem \ref{sp-up} applied to $\tG^{u,r,s}_{\ari} \geq L$, with
$(\tG,G,H,e) = (\tG^{u,r,s}_{\ari},\tG^{u,r,s}_{\ari},L,1)$,  yields that
\begin{equation}\label{eq:mn13}
  \tG^{u,r,s}_{\ari} = C \times L,
\end{equation}  
where $C$ a cyclic scalar subgroup, and either $|C| = 1,2$, or $p=3$ and $|C| =3,6$.

Let $\Phi:\tG^{u,r,s}_{\ari}=CL \to \GL_{p^N}(\C)$ denote the corresponding representation of $\tG^{u,r,s}_{\ari}$ 
acting on $\stW^{u,r,s}$, 
which is a sum of two irreducible representations $\Phi^\eps$ of degree $(p^N -\eps)/2$, $\eps = \pm$.

\smallskip
(ii) Given the information about respective cyclic quotients, we see that the groups $G(n,m)_{\geo}$ and $G(n,m)_{\ari}$ 
have a common last term $M$ of their derived series:
$$M = (G(n,m)_{\geo})^{(\infty)} =  (G(n,m)_{\ari})^{(\infty)}.$$
As $G(n,m)_{\ari} \leq \tG^{u,r,s}_{\ari}$, it follows from \eqref{eq:mn13} that 
\begin{equation}\label{eq:mn14}
  M \leq L=\Sp_{2N}(p).
\end{equation}  

Recall from \eqref{eq:ab12} 
that each of the two irreducible summands $\sW(n,m,\eps)$, $\eps=\pm$, is the $[A]^\star$ Kummer pullback, with
\begin{equation}\label{eq:mn14a}
  A = (q^n+1)/2  \mbox{ when }2|n \mbox{ and }A = (q^m+1)/2 \mbox{ when }2 \nmid n, 
\end{equation}  
of the irreducible hypergeometric sheaf $\sH(n,m,\eps)$ of rank $(p^N - \eps)/2$, which satisfies 
$(\mathbf{S}+)$ by Proposition \ref{cond-S}.
Hence we can apply \cite[Proposition 2.8]{G-T} to its geometric monodromy group $H(n,m,\eps)_{\geo}$ which is finite.

Assume we are in the extraspecial case of \cite[Proposition 2.8(iii)]{G-T} for some $\eps = \pm$.
Then $(q^n-\eps)/2 = (p_2)^a$ for some prime $p_2$ and some $a \in \Z_{\geq 1}$, and $H(n,m,\eps)_{\geo}$ contains a normal 
$p_2$-subgroup $P_2$ that acts irreducibly on the sheaf $\sH(n,m,\eps)$. 
Assume in addition that $\eps = +$ when $2|n$. Then recalling \eqref{eq:mn1} and \eqref{eq:mn14a},
we easily check that $p_2 \nmid A$. On the other hand, $G(n,m,\eps)_{\geo}$ is a normal subgroup of  $H(n,m,\eps)_{\geo}$ of index dividing $A$. It follows that 
$$P_2 \lhd G(n,m,\eps)_{\geo}=\Phi^\eps(G(n,m)_{\geo}) \leq \Phi^\eps(\tG^{u,r,s}_{\ari}).$$ 
Now using \eqref{eq:mn13} and the equality  $(p^N-\eps)/2 = (p_2)^a$, we see that Sylow $p_2$-subgroups of $\tG^{u,r,s}_{\ari}$ are abelian. Hence Sylow $p_2$-subgroups of $G(n,m,\eps)_{\geo}$ are abelian, and so $P_2$ is abelian. But this contradicts the 
irreducibility of $P_2$ on $\sH(n,m,\eps)$.

We still assume the extraspecial case, but now with $\eps = -$ and $2|n$. Then $A = (q^n+1)/2 = (p_2)^a$. Again, $G(n,m,-)_{\geo}$ is a normal subgroup of $H(n,m,-)_{\geo}$ of index dividing $A$, and  
$$G(n,m,-)_{\geo}=\Phi^-(G(n,m)_{\geo}) \leq \Phi^-(\tG^{u,r,s}_{\ari}).$$ 
Now using \eqref{eq:mn13} and the equality  $(p^N+1)/2 = (p_2)^a$, we see that Sylow $p_2$-subgroups of $\tG^{u,r,s}_{\ari}$ are  cyclic of order $(p_2)^a$. Hence $Q_2 := P_2 \cap G(n,m,-)_{\geo} \lhd P_2$ is cyclic of order say $(p_2)^b$ with $0 \leq b \leq a$,
and $P_2/Q_2$ is a cyclic group of order dividing $A$. 
Note that $\Aut(Q_2)$ is trivial if $b=0$ and is cyclic of order $(p_2)^{b-1}(p_2-1)$
if $b \geq 1$. As $b \leq a$ and $P_2/R_2 \inj \Aut(Q_2)$ for $R_2 := \CB_{P_2}(Q_2) \lhd P_2$, we have 
\begin{equation}\label{eq:mn15}
  |P_2/R_2| \leq (p_2)^{a-1}.
\end{equation}  
Next, $R_2/Q_2 \leq P_2/Q_2$ is cyclic, and $Q_2 \leq \ZB(R_2)$. Hence $R_2$ is abelian. This, together with 
\eqref{eq:mn15}, implies by Ito's theorem \cite[(6.15)]{Is} that any irreducible $\C P_2$-module has dimension at most $(p_2)^{a-1}$.
But this again contradicts the irreducibility of $P_2$ on $\sH(n,m,-)$.

\smallskip
(iii) Thus we have shown that $H(n,m,\eps)_{\geo}$ is almost quasisimple for all $\eps = \pm$.
Using property $(\mathbf{S}+)$ and 
\cite[Lemma 2.5]{G-T}, we then have that $\Phi^\eps(M)$ is a quasisimple irreducible subgroup of $\SL_{(p^N-\eps)/2}(\C)$, 
and, furthermore, $M \leq  L=\Sp_{2N}(p)$ by \eqref{eq:mn14}.
By Theorem \ref{sp-down}, there are some divisors $d$ of $n=N/f$ and $e$ of $d$ such that $M = \Sp_{2n/d}(q^d) \rtimes C_e$, whence $e=1$ by perfectness. [Note that Theorem \ref{sp-down} assumes $p > 5$ when $N=3$ and $p >3$ when $N=2$. However, when $N=3$, the statement follows from
Theorem \ref{main-sp1}(i) and (iii); and the case $(p,N)=(3,2)$ is excluded by the assumption $q^n=p^N > 9$.]
In particular, $H(n,m,\eps)_{\ari} \rhd H(n,m,\eps)_{\geo}$ contains the normal quasisimple subgroup $\Phi^\eps(\Sp_{2n/d}(q^d))$. 
By Proposition \ref{cond-S}, it also contains (the image) of the inertia subgroup $I(\delta)$, which has a cyclic $p'$-subgroup $\langle h \rangle$ of order divisible by $(q^{n-m}-1)/2$ that cyclically permutes the $(q^{n-m}-1)/2$ 
irreducible $P(\delta)$-submodules of dimension $q^m$, where $\delta := \infty$ if $2|n$ and $\delta:= 0$ if $2 \nmid n$.

Next we choose $\eps_0 = \pm$ such that $D_0:=(q^n-\eps_0)/2$ is even. Since no outer-diagonal automorphism of 
$\Sp_{2n/d}(q^d)$ can preserve the Weil representation $\Phi^{\eps_0}(\Sp_{2n/d}(q^d))$ up to equivalence, by Schur's lemma we have
\begin{equation}\label{eq:d2}
  \Sp_{2n/d}(q^d) \lhd H(n,m,\eps_0)_{\ari} \leq \NB_{\GL_{D_0}(\C)}(\Sp_{2n/d}(q^d)) \leq (\Sp_{2n/d}(q^d) \cdot C_{df})Z,
\end{equation}  
where $Z = \ZB(\GL_{D_0}(\C))$.

Consider the case $n-m \geq 3$; in particular, $m < n/2$.
Then $q^{n-m}-1=p^{(n-m)f}-1$ admits a primitive prime divisor $\ell$ by \cite{Zs}, which is either equal to
$(n-m)f+1$ or at least $2f(n-m)+1 > N=nf$. Clearly, $\ell \nmid df$ in the latter case. In the former case, if $f \geq 2$ we have
$\ell > 2(n-m) > n \geq d$ and $\ell \nmid f$, whence $\ell \nmid df$. On the other hand, if $f=1$ in the former case, then
$n/2 < \ell = n-m+1$, so $\ell|df$ would imply $\ell=n=d$, $m=1$, $2|n$, and so $n-m=1$, a contradiction. Thus 
$\ell \nmid df$ in all cases, but $\ell$ divides $|h|$. Let $h_0$ denote the $\ell$-part of $h$.
Now using $\ell \nmid df$ and \eqref{eq:d2}, we see
that $h_0 \in \Sp_{2n/d}(q^d)Z$. Since $h_0$ acts nontrivially on the set of $(q^{n-m}-1)/2$ of irreducible $P(\delta)$-submodules
in $\sH(n,m,\eps_0)$, 
we conclude that $h_0 \notin Z$ and so $\ell$ divides $|\Sp_{2n/d}(p^{df})|$. Thus there exists $1 \leq i \leq n/d$ such that $\ell|(p^{2idf}-1)$, whence $n-m$ divides $2id$ by the 
choice of $\ell$. As $n-m > n/2$ and $n-m$ is odd, it follows that $n-m=id$. Hence $d$ divides both $n-m$ and $n$. Since 
$\gcd(n,m) = 1$ by \eqref{eq:mn1}, we conclude that $d=1$.

Next we consider the case $n-m<3$ but $d > 1$. Then $(n,m)=(3,2)$ or $(2,1)$, and $d=n$. Let $Q$ denote the image of $P(\delta)$ in 
$H(n,m,\eps_0)_{\ari}$. Then \eqref{eq:d2} shows that $Q$ has a normal subgroup $Q_1$ of index dividing $nf$, where $Q_1 \in \Syl_p(\Sp_2(q^n)Z)$ is abelian.
It follows from Ito's theorem \cite[(6.15)]{Is} that any irreducible $\C Q_1$-module has dimension dividing $nf$. But this contradicts 
the fact that $P(\delta)$ possesses an irreducible submodule of dimension $q^m = p^{mf}$ on $\sH(n,m,\eps_0)$.

Thus we have shown that $d=1$ and so $M = \Sp_{2n}(q) = L$. Clearly, $M$ is a normal subgroup of each of 
$G(n,m)_{\geo}$, $G(n,m)_{\ari}$.
Furthermore, by Theorem \ref{sp-det}, $\det(H(n,m,\eps)_{\ari})$ has order a $2$-power
(dividing $4$; again, we are using the oddness of $p$ here). It follows that $\det(\Phi^\eps(G(n,m)_{\ari}))$ is also a $2$-group.
The statement now follows from Theorem \ref{ext-sp}.
\end{proof}

The first main result of this section is the following theorem describing the monodromy groups of the local system 
$\sW(\psi,n,m,q)$ defined in \eqref{eq:sheaf1}.

\begin{thm}\label{main-sp3}
Given the assumption \eqref{eq:mn1}, the following statements hold.
\begin{enumerate}[\rm(i)]
\item Let $k$ be any finite extension of $\F_q$. Then the geometric monodromy group $G_{\geo}(\psi,n,m,q)$ 
and the arithmetic monodromy group $G_{\ari}(\psi,n,m,q,k)$
of $\sW(\psi,n,m,q)$ on $\A^1/k$ are 
$$G_{\geo}(\psi,n,m,q) = M,~~G_{\ari}(\psi,n,m,q,k) = C_{\ari,k} \times M,$$
where $M=\Sp_{2n}(q)$ acts via
one of its total Weil representations, 
and either $C_{\ari,k} \leq C_2$,
or $2 \nmid nf$, $p=3$, and $C_{\ari,k} \leq C_6$.
\item Let $e|f$ and let $k = \F_{q^{1/e}}$ be a subfield of $\F_q$. Then on $\A^1/k$ the arithmetic monodromy group $G_{\ari}(\psi,n,m,q,k)$ 
of $\sW(\psi,n,m,q)$ contains $G_{\ari}(\psi,n,m,q,\F_q)$ as a normal subgroup of index $e$:
$$G_{\ari}(\psi,n,m,q,k) = \bigl(C_{\ari,\F_q} \times M\bigr) \cdot C_e,$$
and induces a subgroup of order $e$ of outer field automorphisms of $M=\Sp_{2n}(q)$.
\end{enumerate}
\end{thm}

\begin{proof}
(i) 
In the case $2|n$, $\sW(\psi,n,m,q)$ is the pullback by $[r \mapsto -r]$ of $\sW(n,m)$, and the statements are already proved in Theorem \ref{main-sp2}, using 
the extra information that $G_{\geo}(\psi,n,m,q)$ has no nontrivial $p'$-quotient.

Consider the case $2 \nmid n$. Then, the Kummer pullback 
$$\mathcal{K} =[q^m+1]^\star\sW(\psi,n,m,q)$$ 
of $\sW(\psi,n,m,q)$ has trace function at $r \in k^\times$
$$r \mapsto  \frac{1}{\Gauss(\overline\psi_k,\chi_2)}  \sum_{x \in k}\psi_{-1/2,k}\bigl(x^{q^n+1}+(rx)^{q^m+1}\bigr) = \frac{1}{\Gauss(\overline\psi_k,\chi_2)}\sum_{x \in k}\psi_{-1/2,k}\bigl((r^{-1}x)^{q^n+1}+x^{q^m+1}\bigr)$$
on $\G_m/k$. 

On the other hand, if we define 
$$\sW'(n,m):=[u \mapsto -u]^\star \sW(n,m),$$
and define
$$\sK':=[u \mapsto 1/u]^\star[q^n+1]^\star\sW'(n,m),$$
then $\sK'$ has trace function at $u \in k^\times$ 
$$u \mapsto \frac{1}{\Gauss(\overline\psi_k,\chi_2)}\sum_{x \in k}\psi_{-1/2,k}\bigl((u^{-1}x)^{q^n+1}+x^{q^m+1}\bigr).$$
Thus $\mathcal{K}'$ is arithmetically isomorphic to $\mathcal{K}$, because they have equal trace functions and are each arithmetically semisimple. So they have the same geometric and arithmetic monodromy groups as each other:
$$K_{\geo} =K'_{\geo},\ \ K_{\ari}=K'_{\ari}.$$
From the definition of $\sK'$ as a pullback, we see that  $K'_{\geo}=K_{\geo}$ is a normal subgroup of $G(n,m)_{\geo}$, with cyclic quotient. 
It follows from Theorem \ref{main-sp2} that 
$$M = (G(n,m)_{\geo})^{(\infty)} \leq K_{\geo} \leq G(n,m)_{\geo} = C_0 \times M,$$
for some cyclic scalar subgroup $C_0$. Hence $(K_{\geo})^{(\infty)} = M$. 
From the definition of $\sK$ as a pullback, we see that
 $K_{\geo}$ is a normal subgroup of $G_{\geo}(\psi,n,m,q)$, with cyclic quotient, and that
$G_{\ari}(\psi,n,m,q,k)/G_{\geo}(\psi,n,m,q)$ is cyclic. This in turn implies that 
\begin{equation}\label{eq:sp31}
  (G_{\ari}(\psi,n,m,q,k))^{(\infty)} = (G_{\geo}(\psi,n,m,q))^{(\infty)} = (K_{\geo})^{(\infty)} = M,
\end{equation}  
where $M = \Sp_{2n}(q)$ acts on $\sW(\psi,n,m,q)$ via one of its total Weil representations. 

Note that the arithmetic monodromy group $K_{\ari}$ of $\sK$ is a subgroup of $G(n,m)_{\ari}$ containing $M$. Hence, by Theorem
\ref{main-sp2}, $K_{\ari}/M$ is a $2$-group.
Next, $K_{\geo}$ is a normal subgroup of $G_{\geo}(\psi,n,m,q)$, with cyclic quotient of
order dividing $q^m+1$ which is coprime to $p$. It follows that $p \nmid |G_{\geo}(\psi,n,m,q)/M|$.
As $G_{\geo}(\psi,n,m,q)$ has no nontrivial $p'$-quotient, it follows from \eqref{eq:sp31} that 
$G_{\geo}(\psi,n,m,q) = M$. 
The statement for $G_{\ari}(\psi,n,m,q,k)$ now follows by applying 
Theorem \ref{thm:vdG-vdV} (guaranteeing the necessary properties on traces) and Theorem \ref{ext-sp}.

\smallskip
(ii) It suffices to consider the case $q=p^f > p$.
By assumption, $\F_q$ is an extension of degree $e$ of $k$. Hence, $G_{\ari}(\psi,n,m,q,\F_q)$ is a normal subgroup of 
$\tilde G:=G_{\ari}(\psi,n,m,q,k)$, with cyclic quotient of order dividing $e$. In particular, it follows from (i) that
\begin{equation}\label{eq:sp32}
  \tilde G^{(\infty)} = M = \Sp_{2n}(q),~~|\tilde G| \leq e|G_{\ari}(\psi,n,m,q,\F_q)|.
\end{equation}
On the other hand, if $mf > 2$, note that $\sW(\psi,n,m,q)$ is precisely the sheaf $\tilde\sG^{r,0,0,0}$ (over $k$) considered in 
Theorem \ref{main-sp1b}, with $(N,N') = (nf,mf)$, whence $\tilde G$ is a subgroup of 
$\tilde \Gamma:=\tilde G^{r,v,s,t}_{\ari} = C \times L$, with $L = \Sp_{2N}(p)$ acting via one of its total Weil representations, and 
$C$ a finite cyclic subgroup. If $mf=2$ (and so $(m,f) = (1,2)$), then $\sW(\psi,n,m,q)$ is the sheaf $\sG^{-1,r,0,0}$ considered in 
Theorem \ref{main-sp1}, with $N = nf$, whence $\tilde G$ is a subgroup of 
$\tilde \Gamma:=G^{u,r,s,t}_{\ari} = C \times L$ with $C$ a finite cyclic subgroup. Now we can apply Theorem \ref{vdG-vdV2}(i) to 
$\sW(\psi,n,m,q)$ to see that $|\Tr(\Phi(x))|^2$ is a power of $q^{1/e}$ for all $x \in \tilde G$, and $q^{1/e}$ can be attained,
if $\Phi:\tilde G \to \GL_{q^n}(\C)$ is the representation of $\tilde G$ on the sheaf $\sW(\psi,n,m,q)$. By Theorem \ref{sp-down},
there exists some divisor $d$ of $N/(f/e) = ne$ such that
$$\Sp_{2ne/d}(q^{d/e}) \lhd \tilde G \leq C \times \bigl(\Sp_{2ne/d}(q^{d/e}) \rtimes C_d \bigr) = C\tilde G.$$
Recalling \eqref{eq:sp32}, we now see that $d=e$, and that $\tilde G$ induces a subgroup of order $e$ of outer field automorphisms 
of $M$. As $G_{\ari}(\psi,n,m,q,\F_q) = C_{\ari,\F_q} \times M$ induces only inner automorphisms of $M$, \eqref{eq:sp32} implies 
that $\tilde G = G_{\ari}(\psi,n,m,q,\F_q) \cdot C_e$, as stated.   
\end{proof}

In fact, the central factor $C_{\ari,k}$ in Theorem \ref{main-sp3}(i) will be explicitly determined in Theorem \ref{sp-center}.

To formulate the second main result of the section, recall the assumptions \eqref{eq:mn1} and \eqref{eq:ab11},
and consider the hypergeometric sheaves 
$\sH(n,m,+)$ of rank $(q^n-1)/2$ and $\sH(n,m,-)$ of rank $(q^n+1)/2$ introduced in \eqref{eq:ab12}.
Among these two sheaves, we denote the one of even rank by $\sH^{even}(n,m)$ and the one of odd rank
by $\sH^{odd}(n,m)$. Also, let 
$$\sH(n,m) := \sH(n,m,+) \oplus \sH(n,m,-) = \sH^{odd}(n,m) \oplus \sH^{even}(n,m).$$


\begin{thm}\label{main-sp4}
Given the assumption \eqref{eq:mn1} and the above notation, the following statements hold.
\begin{enumerate}[\rm(i)]
\item Let $k$ be any finite extension of $\F_q$. Then the arithmetic monodromy group $H_{\ari}^{even}(n,m,k)$ and the geometric monodromy group $H_{\geo}^{even}(n,m)$ of $\sH^{even}(n,m)$ on $\G_m/k$ are 
$$H_{\ari}^{even}(n,m,k) = H_{\geo}^{even}(n,m) = M \cong \Sp_{2n}(q)$$
and $M=\Sp_{2n}(q)$ acts in 
one of its even-degree irreducible Weil representations. Furthermore, the arithmetic monodromy group $H_{\ari}^{odd}(n,m,k)$ and the geometric monodromy group $H_{\geo}^{odd}(n,m)$ 
of $\sH^{odd}(n,m)$ on $\G_m/k$ are 
$$H_{\ari}^{odd}(n,m,k) = C'_{\ari,k} \times M/\ZB(M),~H_{\geo}^{odd}(n,m) = C'_{\geo} \times M/\ZB(M),$$
where $M/\ZB(M)\cong \PSp_{2n}(q)$ acts in 
one of its odd-degree irreducible Weil representations, and 
$$1 \leq C'_{\geo} \leq C'_{\ari,k}$$
with $C'_{\ari,k}$ a central subgroup of order $\leq 2$.
 
\item Let $e|f$ and let $k = \F_{q^{1/e}}$ be a subfield of $\F_q$. Then on $\G_m/k$ the arithmetic monodromy group $H_{\ari}^{even}(n,m,k)$ 
of $\sH^{even}(n,m)$ contains $H_{\ari}^{even}(n,m,\F_q)$ as a normal subgroups of index $e$, and likewise for 
the monodromy group of $\sH^{odd}(n,m)$:
$$H_{\ari}^{even}(n,m,k) = M \cdot C_e,~
    H_{\ari}^{odd}(n,m,k) = \bigl( C'_{\ari,\F_q} \times M/\ZB(M) \bigr) \cdot C_e,$$
and each of them induces a subgroup of order $e$ of outer field automorphisms of $M=\Sp_{2n}(q)$.
\end{enumerate}
\end{thm}

\begin{proof}
(i) Note that 
$\sH(n,m) = \sH^{even}(n,m) \oplus \sH^{odd}(n,m)$ 
is exactly the system $\sW_{\sH,\Sp}(1/2) $ defined after Corollary \ref{pullback}, because we already built the 
Tate $(1/2)$-twist  into the definition of  $\sH^{even}(n,m)$ and $ \sH^{odd}(n,m)$; in particular,
Theorem \ref{vdG-vdV2} applies to $\sH(n,m)$. Now,
the geometric monodromy group $G_{\geo}$ of $\sH(n,m)$ contains $G(n,m)_{\geo}$ as a normal subgroup, with cyclic quotient of
order dividing $A$ which is coprime to $p$, and $G(n,m)_{\geo}$ is described in Theorem \ref{main-sp2}. 
As $G_{\geo}$ is a normal subgroup with cyclic quotient in the arithmetic monodromy group $G_{\ari,k}$ of $\sH(n,m)$, it follows that
$G_{\ari,k}^{(\infty)} = G_{\geo}^{(\infty)} = M \cong \Sp_{2n}(q)$, acting via a total Weil representation. Next, as $p$ is odd,
Theorem \ref{sp-det} shows that the determinant of $G_{\ari,k}$ on each of the two irreducible subsheaves of $\sH(n,m)$ is 
a $2$-group. Now, Lemma \ref{ourtracesinK1} and Corollary \ref{pullback} ensure that we can first apply Theorem \ref{ext-sp}
to obtain 
\begin{equation}\label{eq:z2}
  G_{\ari,k} = C'_{\ari,k} \times M,~G_{\geo} = C'_{\geo} \times M,
\end{equation}  
where $C'_{\geo} \leq C'_{\ari,k}$ are both central of order $\leq 2$, acting 
on $\sH(n,m)$ via scalars. The statements in (i) then follow, by recalling that $H_{\ari}^{even}(n,m,k)$ and $H_{\ari}^{odd}(n,m,k)$
are the images of $G_{\ari,k}$ acting on $\sH^{even}(n,m)$ and $\sH^{odd}(n,m)$, with cyclic centers, and noting that
$\ZB(M)$ acts trivially on $\sH^{odd}(n,m)$ and as $\{\pm 1\}$ on $\sH^{even}(n,m)$.

\smallskip
(ii) We again work with $\sH(n,m)$ and its arithmetic monodromy group $G_{\ari,k}$. 
By Theorem \ref{sp-det}, the determinantal image
of $G_{\ari,k}$ on each of $\sH^{even}(n,m)$ and $\sH^{odd}(n,m)$ is a $p'$-group. Now, 
Theorem \ref{vdG-vdV2} 
ensures that we can apply Theorem \ref{sp-up} to $(\tilde G,G,H) = (G_{\ari,k},G_{\ari,\F_q},M)$. As $G_{\ari,k}/G_{\ari,\F_q}$ is cyclic of order dividing $e$, the statements follow.
\end{proof}

Our final result in this section determines all the central subgroups involved in Theorems \ref{main-sp2}--\ref{main-sp4}.

\begin{thm}\label{sp-center}
Keep the assumption \eqref{eq:mn1}. Then the following statements hold.
\begin{enumerate}[\rm(i)]
\item Assume that $2|n$. Then the central subgroups
$C'$ in Theorem \ref{main-sp2}, $C_{\ari,k}$ in Theorem \ref{main-sp3}(i), $C'_{\geo}$ and $C'_{\ari,k}$ in Theorem \ref{main-sp4}(i) are all trivial. Furthermore, for any extension $k$ of $\F_q$, $\sH(n,m)$ has its geometric and arithmetic monodromy groups
$G_{\geo} = G_{\ari,k} = \Sp_{2n}(q)$.
\item Assume that $2 \nmid n$. Then the central subgroups
$C'_{\geo}$ and $C'_{\ari,k}$ in Theorem \ref{main-sp4}(i) are both cyclic of order $2$. Furthermore, 
$G(n,m)_{\ari}=G(n,m)_{\geo} = C_2 \times \Sp_{2n}(q)$ in Theorem \ref{main-sp2}.
Moreover, the local system $\stH(n,m):=\sH(n,m) \otimes \sL_{\chi_2}$ has its geometric monodromy group
$\tG_{\geo} = \Sp_{2n}(q)$.
\item Assume again that $2 \nmid n$. Then the central subgroup $C_{\ari,k}$ in Theorem 
\ref{main-sp3}(i) has order $1$ when $q \equiv 1 (\bmod\ 4)$ or if $k \supseteq \F_{p^2}$, and 
has order $2$ if $q \equiv 3 (\bmod\ 4)$ and $k \not\supseteq \F_{p^2}$.
\end{enumerate}
\end{thm}

\begin{proof}
(i) Our assumptions on $(n,q)$ imply that the sheaf $\sH^{odd}(n,m)$ is the sheaf $\sH(n,m,-)$ of rank $A=(q^n+1)/2$ defined using 
$\sH_{big,A,B,descent}$ in \eqref{eq:ab12}. 
By Theorem \ref{det}(ii), $\sH^{odd}(n,m)$ has trivial arithmetic determinant. Since the 
rank $(q^n+1)/2$ is odd, any central element of order $2$ in $H^{odd}_{\ari}(n,m,k)$ would have determinant $-1$, a contradiction.
It follows that $\ZB(H^{odd}_{\ari}(n,m,k))$ has odd order, and so $C'_{\ari,k} = 1$ in Theorem \ref{main-sp4}(i). 
In particular, we have $G_{\ari,k}=M \cong \Sp_{2n}(q)$ in \eqref{eq:z2}. Now, using the $[A]^\star$
Kummer pullback to get back to $\sW(n,m)$, and the further pullback by $[r \mapsto -r]$ to get to $\sW(\psi,n,m,q)$, we conclude that 
$C'=1$ in Theorem \ref{main-sp2} and 
$C_{\ari,k}=1$ in Theorem \ref{main-sp3}(i).
%

\smallskip
(ii) Our assumptions on $(n,q)$ imply that the sheaf $\sH^{odd}(n,m)$ is the sheaf $\sH(n,m,-)$ of rank $B=(q^n+1)/2$ defined using 
$\sH_{big,A,B,descent}$ in \eqref{eq:ab12} when $2 \nmid B$, and it is the sheaf $\sH(n,m,+)$ of rank $B-1=(q^n-1)/2$ defined using 
$\sH_{small,A,B,descent}$ in \eqref{eq:ab12} when $2|B$. 
Now, by Theorem \ref{detbis}, $\sH^{odd}(n,m)$ has geometric determinant $\sL_{\chi_2}$, and so some element of 
$H^{odd}_{\geo}(n,m)$ has determinant $-1$ on $\sH^{odd}(n,m)$. Hence, $H^{odd}_{\geo}(n,m)$ cannot be perfect, and therefore 
$H^{odd}_{\geo}(n,m) \cong C_2 \times \PSp_{2n}(q)$ and $C'_{\ari,k}=C'_{\geo} \cong C_2$ in Theorem \ref{main-sp4}(i);
in particular, \eqref{eq:z2} implies that $\sH$ has geometric monodromy group 
\begin{equation}\label{eq:z3}
  G_{\geo} = \langle c \rangle \times M \cong C_2 \times \Sp_{2n}(q).
\end{equation} 
Next, since $\sW(n,m,-)$ is the $[A]^\star$ Kummer pullback of $\sH^{odd}(n,m)$, $G_{\geo}(n,m,-)$ is a normal subgroup
of $H^{odd}_{\geo}(n,m,k)$ of index dividing $A$, which is odd, and this implies that 
$G(n,m,-)_{\geo}$ cannot be perfect. As $G(n,m,-)_{\geo}$ is the image of $G(n,m)_{\geo}$ acting on 
$\sW(n,m,-)$, it follows that $G(n,m)_{\ari}=G(n,m)_{\geo} = C_2 \times \Sp_{2n}(q)$ in Theorem \ref{main-sp2}.

Let $\ZB(M) = \langle t \rangle$. Then $t$ acts as $-1$ on the even-rank subsheaf $\sH^{even}(n,m)$ of $\sH(n,m)$ and trivially on
the odd-rank subsheaf $\sH_{odd}(n,m)$. Replacing $c$ from \eqref{eq:z3} by $ct$ if necessary, we may assume that $c$ acts trivially on 
$\sH^{even}(n,m)$, whence $c$ acts as $-1$ on $\sH^{odd}(n,m)$ (otherwise $c$ would be trivial). 
Now \eqref{eq:z3} implies that $\sH^{odd}(n,m)$ has geometric determinant $\sL_{\chi_2}$ and $\sH^{even}(n,m)$ has trivial geometric determinant. Hence, both $\sH^{odd}(n,m) \otimes \sL_{\chi_2}$ and $\sH^{even}(n,m) \otimes \sL_{\chi_2}$ have trivial geometric determinants.

Next, tensoring with $\sL_{\chi_2}$ changes the trace at $v \in E^\times$ by a factor of $\chi_2(v) = \pm 1$. In particular, it does not
change the absolute value of the trace at any $v \in E^\times$. Furthermore, the $[2]^\star$ Kummer pullbacks of $\sH(n,m)$
and $\stH(n,m)$ are isomorphic, and so $\tG_{\geo}$ has a normal subgroup $X$ of index at most $2$, which is also a normal subgroup of $G_{\geo}$ of index at most $2$. It follows that $(\tG_{\geo})^{(\infty)} = X^{(\infty)} = M \cong \Sp_{2n}(q)$. Applying Theorem \ref{ext-sp}
to $\tG_{\geo}$ and arguing as in p. (i) of the proof of Theorem \ref{ext-sp}, we conclude that $M \lhd \tG_{\geo} \leq C_2 \times L$. 
Now, if $\tG_{\geo} > M$, then we have 
$\tG_{\geo} = \langle \tilde c \rangle \times M$ with $\langle \tilde c \rangle \cong C_2$. Since $\tilde c$ has trivial determinant on 
$\sH^{odd}(n,m) \otimes \sL_{\chi_2}$, it acts trivially on it, and $\tilde c$ acts as $1$ or $-1$ on $\sH^{even}(n,m) \otimes \sL_{\chi_2}$. But this means that the action of $\tilde c$ on $\stH(n,m)$ agrees with some element in $\ZB(M)$ and so $\tilde c \in \ZB(M)$ by faithfulness, a contradiction. Thus $\tG_{\geo}=M$, as stated.

\smallskip
(iii) In this case we have $2 \nmid n$. Recalling \eqref{eq:ab11}, we note from Corollary \ref{Apullbackbis1} that $\sW(\psi,n,m,q)$ 
is the pullback by 
$[r \mapsto -r]$ of $[B]^\star \stW$, where 
$$\stW:= \sH_{small,B,A,descent} \otimes \bigl(-\Gauss(\overline\psi_{\F_p},\chi_2)\bigr)^{-\deg} \oplus 
                            \sH_{big,B,A,\rho,descent} \otimes \bigl(-\Gauss(\overline\psi_{\F_p},\chi_2)\bigr)^{-\deg},$$
and $\rho$ is chosen so that $\rho^B=\chi_2$ (in particular, we will take $\rho=\chi_2$ if $B=(q^n+1)/2$ is odd).  

Consider the case $2 \nmid B$, equivalently, $4|(q-1)$. Then, for any $k \supseteq \F_q$, $-1$ is a square in $k$, whence
$\Gauss(\psi_k,\chi_2)=\Gauss(\overline\psi_k,\chi_2)$. By Proposition \ref{arithdet}, both 
$\sH_{small,B,A,descent} \otimes \bigl(-\Gauss(\overline\psi_{\F_p},\chi_2)\bigr)^{-\deg}$ and 
$\sH_{big,B,A,\rho,descent} \otimes \bigl(-\Gauss(\overline\psi_{\F_p},\chi_2)\bigr)^{-\deg}$ have trivial arithmetic determinants.
Now consider any central element $c$ in the arithmetic monodromy group of $[B]^\star\stW$. By Theorem \ref{main-sp3}(i), 
$\ord(c)$ divides $2p$. Now, on the subsystem of $[B]^\star\stW$ of odd rank $(q^n-\eps)/2$ (for a suitable $\eps \in \{-1,1\}$), 
$c$ acts as a scalar $\alpha$ with $\alpha^{2p}=1$ and $1=\det(c)=\alpha^{(q^n-\eps)/2}$, whence $\alpha=1$. On the subsystem of 
$[B]^\star\stW$ of even rank $(q^n+\eps)/2$, $c$ acts as a scalar $\beta$ with 
$\beta^{2p}=1$ and $1=\det(c)=\beta^{(q^n+\eps)/2}$, whence $\beta= \pm 1$. We see that the action of $c$ agrees with the action 
of a central element of $M = \Sp_{2n}(q)$, and therefore $c \in \ZB(M)$. Pulling back by $[r \mapsto -r]$ to $\sW(\psi,n,m,q)$, we 
obtain $C_{\ari,k}=1$ in Theorem \ref{main-sp3}(i).

Now assume $2|B$, equivalently, $4|(q+1)$. We apply Proposition \ref{arithdetbis} with $C=-1$ to see that the subsheaf
$[B]^\star\sH_{small,B,A,descent} \otimes \bigl(-\Gauss(\overline\psi_{\F_p},\chi_2)\bigr)^{-\deg}$, which has odd rank $B-1$,
has arithmetic determinant $(-1)^{\deg}$, whereas  
$[B]^\star\sH_{big,B,A,\rho,descent} \otimes \bigl(-\Gauss(\overline\psi_{\F_p},\chi_2)\bigr)^{-\deg}$, which has even rank $B$, 
has trivial arithmetic determinant.
Again consider any central element $c$ in the arithmetic monodromy group of $[B]^\star\stW$. By Theorem \ref{main-sp3}(i), 
$\ord(c)$ divides $2p$. Now, if $k \supseteq \F_{p^2}$, equivalently, $[k:\F_p]$ is even, then $c$ has trivial determinant on 
both subsystems of $[B]^\star\stW$, and the previous arguments show that $c \in \ZB(M)$, and 
pulling back by $[r \mapsto -r]$, we see that $C_{\ari,k}=1$ in Theorem \ref{main-sp3}(i). The same arguments also show
that we always have $c^2 \in \ZB(M)$, whence $\ord(c)$ divides $\gcd(4,2p) = 2$ and thus $\ord(c)$ divides $2$.
Now assume that $[k:\F_p]$ is odd. We have just proved that $\ZB(G_{\ari}(\psi,n,m,q,k))$ is a $2$-group and contains
$\ZB(M)\cong C_2$, whence $C_{\ari,k}$ has order $1$ or $2$. Suppose that $C_{\ari,k}=1$. Then 
$G_{\ari}(\psi,n,m,q,k)=M$ is perfect, and so it has trivial arithmetic determinant on both subsystems of $\sW(\psi,n,m,q)$,
a contradiction. Thus $C_{\ari,k}=C_2$ in this case.
\end{proof}


\section{Local systems and total Weil representations: Unitary groups over $\F_q$}
We continue to work with the prime $p > 2$, and fix a power $q=p^f$ and positive integers $n,m$, where 
\begin{equation}\label{eq:mn2}
  n > m,~~\gcd(n,m) = 1,~~2 \nmid mn,~~n \geq 3,~~\mbox{and either }m< n/2,\mbox{ or }(n,m) = (5,3).
\end{equation}  
This assumption implies that 
\begin{equation}\label{eq:mn22}
  \gcd(q^n+1,q^m+1) = q+1.
\end{equation}  
For compatibility with the notations used in section $\S5$, we denote
$$A:=(q^n+1)/(q+1), \ \ B:=(q^m+1)/(q+1).$$
In this section, we study the local system $\sW^{\, n,m}$ on $\A^1/\F_{q^2}$ with trace function given as follows: for $k/\F_{q^2}$ a finite extension, and $r \in k$,
$$r \mapsto \frac{1}{\Gauss(\psi_k,\chi_2)} \sum_{x \in k}\psi_k\bigl(x^{q^n+1} - rx^{q^m+1}\bigr).$$
Next, we fix a character $\chi_{q+1}$ of order $q+1$, and then, for $0 \leq j \leq q$, define $\sW^{\, n,m,j}$ to be the local system on $\G_m/\F_{q^2}$ whose trace function is given by as follows: for
$k/\F_{q^2}$ a finite extension, and $r \in k$,
$$r \mapsto \frac{1}{\Gauss(\psi_k,\chi_2)} \sum_{x \in k}\psi_k\bigl(x^A - rx^B\bigr)\chi_{q+1}^j(x).$$
By \eqref{eq:mn22} and Proposition \ref{ABmoment}, $\sW^{\, n,m} = \oplus^q_{j=1}\sW^{\, n,m,j}$ is the sum of
$(q+1)$ irreducible subsystems $\sW^{\, n,m,j}$, of rank $(q^n-q)/(q+1)$ for $j = 0$ and $(q^n+1)/(q+1)$ when $1 \leq j \leq q$.
Let $G^{n,m}_{\ari}$ and $G^{n,m}_{\geo}$, respectively $G^{n,m,j}_{\ari}$ and $G^{n,m,j}_{\geo}$, denote the arithmetic and the geometric monodromy groups of $\sW^{\, n,m}$, respectively of $\sW^{\, n,m,j}$.
 
Next, recall that in Lemma \ref{choice} we defined 
$$n_0:=\gcd(n,(q+1)^n),$$ 
and showed that we can fix a character $\nu$ of order $n_0(q+1)$ such that $\nu^A = \chi_{q+1}$.
We then define the hypergeometric sheaves $\sH^{n,m,j}$ over $\G_m/\F_{q^2}(\nu)$ for $j \in \Z$ to be
\begin{equation}\label{eq:ab21}
  \sH^{n,m,j} = \left\{ \begin{array}{ll}\sH_{small,A,B,descent}\otimes \bigl(-\Gauss(\psi_{\F_{q^2}(\nu)},\chi_2)\bigr)^{-\deg}, & \mbox{if }j \equiv
  0 (\bmod\ q+1),\\
    \sH_{big,A,B,\nu^j,descent}\otimes \bigl(-\Gauss(\psi_{\F_{q^2}(\nu)},\chi_2)\bigr)^{-\deg}, & \mbox{if }j \not\equiv 0 (\bmod\ q+1),\end{array} \right.
\end{equation}    
with $\sH_{small,A,B,descent}$ and $\sH_{big,A,B,\chi,descent}$ as defined in \S5. We also let
\begin{equation}\label{eq:ab21a}
  \sH^{n,m}=\bigoplus^q_{j=0}\sH^{n,m,j}.
\end{equation}  
By Propositions \ref{totaltraces} and \ref{totaltraces-1}, $\sW^{\, n,m,j}$ is the $[A]^\star$ Kummer pullback of $\sH^{n,m,-j}$.  
Denote the arithmetic and geometric monodromy groups of $\sH^{n,m,j}$ by $H^{n,m,j}_{\ari}$ and 
$H^{n,m,j}_{\geo}$. Again, this pullback relationship implies that $G^{n,m,j}_{\geo} \lhd H^{n,m,-j}_{\geo}$ 
and the quotient is a cyclic group of order dividing $A$.

\smallskip
We will need the following statement, which is an odd-$n$ analogue of \cite[Lemma 17.3]{KT5}. For the reader's convenience,
we give the proof.

\begin{lem}\label{center-su}
Let $Z$ be a finite abelian group, $q$ a prime power, and let $\lambda_0, \lambda_1, \ldots ,\lambda_q \in \Irr(Z)$.
\begin{enumerate}[\rm(i)]
\item Suppose $\Lambda:=\sum^q_{i=0}\lambda_i$ vanishes on $Z \smallsetminus \{1\}$. Then $|Z|$ divides $q+1$.
\item Suppose there is some $z \in Z$ such that $\Lambda=\sum^q_{i=0}\lambda_i$ vanishes on $Z \smallsetminus \{1,z\}$ and 
$\Lambda(z)=-(q+1)$. Then $|Z|$ divides $2(q+1)$.
\item Suppose $2 \nmid n \geq 3$, $(n,q) \neq (3,2)$, $\lambda_0^2=1_Z$, and that 
$$\Sigma:= -\lambda_0 + A\sum^q_{i=0}\lambda_i,$$
with $A:=(q^n+1)/(q+1)$, 
takes values only in $\{-q^n,0, \pm q^i \mid 0 \leq i \leq n-1\}$ on $Z \smallsetminus \{1\}$. Then either 
$|Z|$ divides $q+1$, or $Z$ contains an element $z$ with $\lambda_i(z)=-1$ for all $0 \leq i \leq q$. In the latter case,
if in addition $\Sigma$ is faithful, then $|Z|$ divides $2(q+1)$.
\end{enumerate}
\end{lem} 
 
\begin{proof}
(i) Note that 
$$[\Lambda,1_Z]_Z = \frac{1}{|Z|}\sum_{x \in Z}\Lambda(x) = \frac{q+1}{|Z|}$$
is an integer, whence the statement follows.

(ii) Let $\alpha$ be the linear character of $\langle z\rangle$ sending $z$ to $-1$. Since $Z$ is abelian, we can find a linear
extension $\beta$ of $\alpha$ to $Z$. Now 
$$[\Lambda,\beta]_Z = \frac{1}{|Z|}\sum_{x \in Z}\Lambda(x)\beta(x) = \frac{(q+1)\beta(1)-(q+1)\beta(z)}{|Z|}=\frac{2(q+1)}{|Z|}$$
is an integer, whence the statement follows.

\smallskip
(iii) Consider any $1 \neq x \in Z$. By the assumption, $\lambda_0(x) = \pm 1$, and $\Sigma(x)=0$, $-q^n$, or $\pm q^j$ for 
some $0 \leq j \leq n-1$. Now
$$\Z \ni \Sigma(x)+\lambda_0(x) = A\cdot \Lambda(x),$$
and so $\Lambda(x) = (\Sigma(x)+\lambda_0(x))/A$ is both rational and an algebraic integer, whence 
\begin{equation}\label{eq:Adiv}A \mbox{ divides } \Sigma(x)+\lambda_0(x). 
\end{equation}
We will now show that either $\Sigma(x) = -\lambda_0(x)$ or $\Sigma(x) = -q^n$.
If $\Sigma(x)=0$, or $\Sigma(x)=\pm q^j$ with $1 \leq j \leq n-2$, or if $\Sigma(x)=\lambda_0(x)$, then 
$\Sigma(x)+\lambda_0(x) \neq 0$ and $|\Sigma(x)+\lambda_0(x)| \leq q^{n-2}+1 < A$ (as $n \geq 3$ and $(n,q) \neq (3,2)$), contradicting (\ref{eq:Adiv}). If 
$\Sigma(x)=\lambda_0(x) q^{n-1}$, then we have $A|(q^{n-1}+1)$ by \eqref{eq:Adiv}, whence $\frac{q^n+1}{\gcd(2,q-1)}$ divides 
$(q+1) \cdot \frac{q^{n-1}+1}{\gcd(2,q-1)}$, 
which is impossible since $\gcd(q^n+1,q^{n-1}+1) = \gcd(2,q-1)$. If 
$\Sigma(x)=-\lambda_0(x) q^{n-1}$, then we have $A|(q^{n-1}-1)$ by \eqref{eq:Adiv},
which is also impossible since $\gcd(q^n+1,q^{n-1}-1)=q+1 < (q^n+1)/(q+1)$ for $2 \nmid n \geq 3$ and $(n,q) \neq (3,2)$. 

\smallskip
(iv) Now, if $\Sigma(x) \neq -q^n$ for all $1 \neq x \in Z$, 
then $\Sigma(x)=-\lambda_0(x)$ and $\Lambda(x)=0$
for all $1 \neq x \in Z$, whence the statement follows from (i).

Consider the case $\Sigma(x)=-q^n$ for some $1 \neq x \in Z$. Then we must have 
$\lambda_0(x)=-1$, and 
$$\sum^q_{i=0}(-\lambda_i(x)) = -\Lambda(x) = (-\lambda_0(x)-\Sigma(x))/A = q+1,$$ 
implying that all roots of unity $-\lambda_i(x)$ must be $1$. Now, assume $\Sigma$ is faithful, and 
fix an element $z \in Z$ with $\lambda_i(z)=-1$ for all $i$. In this case, $\lambda_i(xz^{-1})=1$
for all $i$, and so $\Sigma(xz^{-1})=q+1$ and $x=z$ by faithfulness of $\Sigma$. We have shown that $\Lambda(x) = -(q+1)$ for 
$x=z$, and $\Lambda(x) = 0$ for all $x \in Z \smallsetminus \{1,z\}$, and so the statement follows from (ii).
\end{proof}

\smallskip
In this section, we will work with a subgroup $\GU_n(q) = \GU(W)$ of $\Sp_{2n}(q)$ as specified in 
\cite[Theorem 3.4]{KT3}, $W = \F_{q^2}^n$, and with $\bj = -1_W$, the central involution of both $\GU(W)$ and $\Sp_{2n}(q)$. 

The main result of this section is the following theorem:

\begin{thm}\label{main-su1}
Given the assumption \eqref{eq:mn2}.
Then the geometric monodromy group $G^{n,m}_{\geo}$ of $\sW^{\, n,m}$ is isomorphic to $\SU_n(q)$ acting in its total Weil representation of degree $q^n$. Furthermore, for any finite extension $k$ of $\F_{q^2}$, the arithmetic monodromy group $G^{n,m}_{\ari,k}$ of $\sW^{\, n,m}$ on 
$\G_m/k$ is 
$$G^{n,m}_{\ari,k} = C_{\ari,k} \times \SU_n(q),$$ 
where $C_{\ari,k} = C_{\ari,\F_{q^2}} = \langle \bj \rangle \cong C_2$ if 
$2 \nmid \deg(k/\F_{q^2})$, and $C_{\ari,k}=1$ if $2 \mid \deg(k/\F_{q^2})$. 
\end{thm}

\begin{proof}
(i) Note that, since $k \supseteq \F_{q^2}$, $\Gauss(\psi_k,\chi_2) = \Gauss((\psi_a)_k,\chi_2)$ for any
$\psi_a:t \mapsto \psi(at)$ with $a \in \F_p^\times$. In particular, 
$\sW^{n,1}$ is precisely the pullback by $[r \mapsto -r]$ of the local system $\sW(\psi,n,q)$ considered in \cite[\S4]{KT3}, where we have shown in Theorem
4.2 that it has geometric monodromy group $\SU_n(q)$ (in its total Weil representation of degree $q^n$). Thus 
$G^{n,1}_{\geo} = K:= \SU_n(q)$. 

\smallskip
(ii) In this and the next part of the proof we will assume that $m > 1$.
Consider the local system $\sW^{\, n,m,1,0}$ on $\A^3/\F_{q^2}$ with trace function given as follows. 
For $k/\F_{q^2}$ a finite extension,  and $r,s,t \in k$,
$$(r,s,t) \mapsto  \frac{1}{\Gauss(\psi_k,\chi_2)}\sum_{x \in k}\psi_k\bigl(x^{q^n+1} - rx^{q^m+1}+sx^{q+1}+tx^2\bigr),$$
with arithmetic monodromy group $G^{n,m,1,0}_{\ari}$. As mentioned above, $k \supseteq \F_{q^2}$ implies
that the system $\sW^{\, n,m,1,0}|_{r=0}$ at $r=0$ is exactly 
the local system $\sW_{{\tiny \mbox{2-param}}}(\psi,n,q)$ considered in \cite[\S4]{KT3}. By \cite[Theorem 4.3]{KT3}, 
the arithmetic monodromy group of $\sW^{\, n,m,1,0}|_{r=0}$ equals $L:= \Sp_{2n}(q)$ in one of its total Weil representations 
of degree $q^n$. Thus $G^{n,m,1,0}_{\ari}$ contains $L$. {By Proposition \ref{ABmoment}, $\sW^{\, n,m,1,0}$ is still a sum of two 
subsystems of rank $(q^n \pm 1)/2$}.  Furthermore, it satisfies the conclusions of Theorem \ref{thm:vdG-vdV}.
Now, applying Theorem \ref{sp-up} to $G^{n,m,1,0}_{\ari} \geq L$ (with $e=1$), 
we obtain that
\begin{equation}\label{eq:mn23}
  G^{n,m,1,0}_{\ari} = C \times L,
\end{equation}  
where $C$ a cyclic scalar subgroup, and either $|C| = 1,2$, or $p=3$, $2 \nmid f$, and $|C| =3,6$.

\smallskip
Let $\Phi:G^{n,m,1,0}_{\ari}=CL \to \GL_{q^n}(\C)$ denote the corresponding representation of $G^{n,m,1,0}_{\ari} = CL$ 
acting on $\sW^{\, n,m,1,0}$. Note that $\sW^{\, n,m}$ is precisely $\sW^{\, n,m,1,0}|_{s=0,\ t=0}$, hence $\tG:=G^{n,m}_{\ari,k}$ is a subgroup of 
$G^{n,m,1,0}_{\ari}$. As $\sW^{\, n,m}$ is a local system on $\A^1$, its geometric monodromy group $G:=G^{n,m}_{\geo}$ satisfies 
\begin{equation}\label{eq:mn23b}
  G = \OB^{p'}(G). 
\end{equation}
Given the information about respective cyclic quotients, we see that the two groups 
$G^{n,m}_{\geo}$ and $G^{n,m}_{\ari}$
have a common last term $K$ of their derived series:
\begin{equation}\label{eq:mn24}
  K = (G^{n,m}_{\geo})^{(\infty)} =  (G^{n,m}_{\ari,k})^{(\infty)} \leq (G^{n,m,1,0}_{\ari})^{(\infty)} = L.
\end{equation}  
By \eqref{eq:mn22} and Proposition \ref{ABmoment}, $\Phi|_G$ is a sum of $q+1$ irreducible summands $\Phi_j$
acting on $\sW^{\, n,m,j}$, $0 \leq j \leq q$. Since $\Phi(C)$ consists of scalar matrices, the same is true for 
$CG = C \times (CG \cap L)$, whence also for $CG \cap L$. Applying \cite[Theorem 3.4]{KT3}
to $CG \cap L$, we see that 
$$\SU_n(q) \lhd CG \cap L \leq \GU_n(q),$$ 
where $\GU_n(q)$ is realized inside $L$ via a standard Hermitian 
structure on $\F_q^{2n}$. As $K \leq CG \cap L$, we now have
$$\SU_n(q) = (CG \cap L)^{(\infty)} \leq (CG)^{(\infty)} = G^{(\infty)} = K=K^{(\infty)} \leq (CG \cap L)^{(\infty)} = \SU_n(q),$$
i.e. $K = \SU_n(q)$, acting in its total Weil representation.

\smallskip
(iii) Since $\tG \rhd K$ by \eqref{eq:mn24}, from \eqref{eq:mn23} we now get
$$G \lhd \tG \leq \NB_{C \times L}(K) = C \times \NB_L(K) = C \times \bigl( \GU_n(q) \rtimes C_2\bigr).$$
Note that $C \times \GU_n(q)$ preserves the equivalence of each of $q+1$ irreducible summands $(\Phi_i)|_K$, but the 
subgroup $C_2$ (generated by a field automorphism) does not. It follows that 
\begin{equation}\label{eq:mn25}
  \SU_n(q) = K \lhd G \leq \tG \leq C \times \GU_n(q).
\end{equation}   

Recall from Lemma \ref{detdown2} that the sheaves $\sH_{small,A,B}$ and $\sH_{big,A,B,\chi}$ all have geometric determinants 
being trivial or $\sL_{\chi_2}$. As $\sW^{n,m,j}$ is the $[A]^\star$ Kummer pullback of $\sH^{n,m,-j}$, the same is true for 
$G=G^{n,m}_{\geo}$ acting on each $\sW^{n,m,j}$. Hence, if $\Phi^\eps$ denotes the two summands of degree $(q^n-\eps)/2$, $\eps = \pm$,
of the $CL$-representation $\Phi$, then $\det(\Phi^\eps(g))^2=1$ for all $g \in G$. However, $\det(\Phi^\eps(x)) = 1$ for all
$x \in L$ as $L$ is perfect, and $\det(\Phi^\eps(c))$ has order $3$ if $1 \neq c \in \OB_3(C)$ when $p=3$, since
$C$ is scalar and $\deg(\Phi^\eps)$ is coprime to $p$. Recalling $C \leq C_{2 \cdot\gcd(p,3)}$, we now see from \eqref{eq:mn25}
that $\SU_n(q) \lhd G \leq \OB_{p'}(C) \times \GU_n(q)$. Together with \eqref{eq:mn23b}, this implies that $G = K=\SU_n(q)$.

\smallskip
(iv) Now we return to the general case $m \geq 1$ and let $\lambda_j$ be the central character of $\ZB(\tG)$ acting on 
$\sW^{\, n,m,j}$, $0 \leq j \leq q$. Recall that $\Phi|_{\tG}$ has integer traces, belonging to $\{\pm q^i \mid 0 \leq i \leq n\}$ by Theorem \ref{tracevalues1}, and so it is self-dual. But $(\Phi_0)|_{\tG}$ is the unique 
irreducible constituent of $\Phi|_{\tG}$ of degree $A-1$, hence $(\Phi_0)|_{\tG}$ is self-dual; in particular, 
$\lambda_0^2$ is trivial. It follows that $\Sigma:=-\lambda_0+A\sum^q_{i=0}\lambda_i$ satisfies 
all the hypotheses of Lemma \ref{center-su}, whence 
\begin{equation}\label{eq:mn26}
  |\ZB(\tG)| \mbox{ divides }2(q+1).
\end{equation}  
In particular, we are done if $C \leq C_2$. Consider the case $C \geq C_3$, whence $p=3$. By \eqref{eq:mn25}, 
$\CB_{\tG}(K) = \ZB(\tG)$, and $\tG/\CB_{\tG}(K) \leq \PGU_n(q)$. It then follows from \eqref{eq:mn26} that 
$|\tG/G|$ divides $2(q+1)^2$. On the other hand, 
$\tG/G \leq C \times C_{q+1}$, with $C \times C_{q+1}$ being an abelian group. Hence, 
$$\tG/G \leq \OB_{3'}(C \times C_{q+1}) = \OB_{2}(C) \times C_{q+1} = \bigl(\OB_{2}(C) \times \GU_{n}(q) \bigr)/G,$$
and so 
\begin{equation}\label{eq:mn27}
  G^{n,m}_{\ari,k} \leq \OB_{2}(C) \times \GU_n(q).
\end{equation}  

\smallskip
(v) To completely determine $G^{n,m}_{\ari,k}$, 
first we show that in \eqref{eq:mn27} in fact we have 
\begin{equation}\label{eq:mn28}
  \SU_n(q)=H \lhd G^{n,m}_{\ari,k} \leq \GU_n(q). 
\end{equation}  
This is obvious if $\OB_2(C)=1$, so we will assume that $\OB_2(C) = \langle \bt \rangle$
with $\Phi(\bt) = -\Id$ and that $\tG \ni \bt h$ for some $h \in \GU_n(q)$. We will decompose the total Weil representation $\Phi$ of 
$\GU_n(q)$ as $\oplus^q_{i=0}\Psi_i$ as in \cite[\S3]{KT3}; in particular, $\deg(\Psi_i) = (q^n+1)/(q+1)-\delta_{i,0}$. 
The same decomposition applies to $\OB_2(C) \times \GU_n(q)$, as $\Phi(\bt) = -\Id$.
Restricted to $\tG$, each $\Psi_i$ with $1 \leq i \leq q$ corresponds to the action of $\tG$ on some $\sW^{n,m,j}$ which in turn is
the $[A]^\star$ Kummer pullback of $\sH_{big,A,B,\chi,descent}$ for some $\chi$. Restricted further down to $H =\SU_n(q)$, 
$\Psi_{(q+1)/2}$ is the only self-dual one among the $q$ irreducible Weil representations of degree $A$ of $H$. 
Taking $\rho=\chi=\chi_2$ in Proposition \ref{totaltraces}, we see that $\sH_{big,A,B,\chi_2,descent}$ is geometrically self-dual of rank $A$;
hence this sheaf corresponds to $\Psi_{(q+1)/2}$. Furthermore it has trivial arithmetic determinant, by Proposition 
\ref{arithdet}(iii), see also Corollary \ref{det-stick}.
On the other hand, by \cite[Lemma 3.2(iii)]{KT3},
$\det(\Psi_{(q+1)/2}(h))=1$, and so $\det(\Psi_{(q+1)/2}(\bt h))=-1$, a contradiction.  

Having established \eqref{eq:mn28}, we can write $G^{n,m}_{\ari,k} = \langle H,g\rangle$, where $g := \diag(\rho^j,1, \ldots ,1)$, 
$\rho \in \F_{q^2}^\times$ has order $q+1$, and $0 \leq j \leq q$. As shown in the proof of \cite[Lemma 3.2]{KT3}, 
\begin{equation}\label{eq:mn29}
  \det(\Psi_i(g)) = \zeta^{j(i+(q+1)/2)},
\end{equation}  
if $1 \leq i \leq q$ and $\zeta = \zeta_{q+1} \in \C^\times$ has order $q+1$. According to Corollary \ref{det-stick},
all the $q$ components of degree $A$ of $\sW^{n,m}$ have arithmetic determinant $\pm 1$, hence 
$\det(\Psi_i(g)) = \pm 1$ for all $1 \leq i \leq q$. Applying this and \eqref{eq:mn29} to $i=(q+3)/2$, we get $\zeta^{2j}=1$,
i.e. $(q+1)/2$ divides $j$. Since $2 \nmid n$ it is easy to see that $\langle H,g^{(q+1)/2} \rangle = H \times \langle \bj \rangle$,
and so we have shown that 
\begin{equation}\label{eq:mn29a}
  G^{n,m}_{\ari,k}=C_{\ari,k} \times \SU_n(q),
\end{equation}  
with $C_{\ari,k} \leq C_{\ari,\F_{q^2}} \leq \langle \bj \rangle$.

Assume now that for $C_{\ari,\F_{q^2}}=1$. Then, 
$G^{n,m}_{\ari,\F_{q^2}}=H=\SU_n(q)$ is perfect. It follows that all $q+1$ subsheaves of $\sW^{n,m}$ have trivial arithmetic determinants
over $\F_{q^2}$. If $q \equiv 3 (\bmod\ 4)$, then we choose $\theta$ of order $r:=q+1$, so 
that $(-1)^{(q+1)/r} = -1 \neq (-1)^{(q+1)/2}$. If $q \equiv 1 (\bmod\ 4)$, then we choose $\theta$ of order $r:=(q+1)/2$, so 
that $(-1)^{(q+1)/r} = 1 \neq (-1)^{(q+1)/2}$. In both cases, by Corollary \ref{det-stick}, this choice of $\theta$ implies that the 
subsheaf of rank $A$ of $\sW^{n,m}$ labeled by $\theta$ has nontrivial arithmetic determinant $(-1)^{\deg}$ over extensions of $\F_{q^2}$,  a contradiction. Hence
$C_{\ari,\F_{q^2}}=\langle \bj \rangle$.

Finally, since $G^{n,m}_{\ari,\F_{q^2}}/G^{n,m}_{\geo} = C_2$, the $C_2$ quotient is geometrically trivial and so must be $(-1)^{\deg}$ 
arithmetically. Together with \eqref{eq:mn29a}, this implies that $C_{\ari,k} = C_{\ari,\F_{q^2}}$ when 
$2 \nmid \deg(k/\F_{q^2})$ and $C_{\ari,k}$ is trivial when $2|\deg(k/\F_{q^2})$.
\end{proof}

\begin{thm}\label{main-su2}
Given the assumption \eqref{eq:mn2}. Then the following statements hold.

\begin{enumerate}[\rm(a)]
\item The geometric monodromy group $H=H^{n,m}_{\geo}$ of $\sH^{n,m}$ contains $G^{n,m}_{\geo} = \SU_n(q)$ as 
a normal subgroup, with $H^{n,m}_{\geo}/G^{n,m}_{\geo}$ being cyclic of order $n_0$. Furthermore, $H/\ZB(H) \cong \PGU_n(q)$.
\item Let $0 \leq j \leq q$ and let $H_j=H^{n,m,j}_{\geo}$ be the geometric monodromy group of the hypergeometric 
sheaf $\sH^{n,m,j}$, defined in \eqref{eq:ab21}. Then $H_j^{(\infty)}$ is the image of 
$\SU_n(q)$ in an irreducible Weil representation, of 
degree $A=(q^n+1)/(q+1)$ if $1 \leq j \leq q$ and $A-1$ if $j = 0$, 
$H_j/H_j^{(\infty)}$ is cyclic of order dividing $n_0$, and $H_j/\ZB(H_j) \cong \PGU_n(q)$.
\item Let $k$ be any finite extension of $k_0:=\F_{q^2}(\nu)=\F_{q^{2n_0}}$. Then the arithmetic monodromy group of $\sH^{n,m}$ over 
$\G_m/k$ is $C_{\ari,k} \times H^{n,m}_{\geo}$, where $C_{\ari,k} = \langle \bj \rangle$ if $2 \nmid \deg(k/k_0)$ and 
$C_{\ari,k}=1$ if $2|\deg(k/k_0)$.
\end{enumerate}
\end{thm}

\begin{proof}
(i) The definition given in \eqref{eq:ab21} tells us that $H_j$ has its $I(0)$ being cyclic of order $A=(q^n+1)/(q+1)$.
Moreover, $H_j$ has property $({\mathbf S}+)$ by Proposition \ref{prop-Sbis}, and by Theorem \ref{main-su1}, $H_j^{(\infty)}=G^{n,m,-j}_{\geo}$ is 
the image of $\SU_n(q)$ in the relevant irreducible Weil representation, with $H_j/H_j^{(\infty)}$ being cyclic of order dividing $A$. 
Hence, $\PSU_n(q)$ is the unique non-abelian 
composition factor of $H_j$, and by Theorem 8.3 and Corollary 8.4 of \cite{KT4}, 
\begin{equation}\label{eq:hs10}
  H_j/\ZB(H_j) \cong \PGU_n(q).
\end{equation}  

\smallskip
(ii) Next, since the $[A]^\star$ Kummer pullback of $\sH^{n,m}$ is $\sW^{n,m}$, $G:=G^{n,m}_{\geo} \cong \SU_n(q)$ is a normal subgroup
of $H:=H^{n,m}_{\geo}$, with cyclic quotient of order dividing $A$; in particular, we can write
\begin{equation}\label{eq:hs11}
  H = \langle G,g \rangle \rhd G
\end{equation}  
for some element $g \in H$. 

Let $\Psi_j$ denote the representation of $H$ on $\sH^{n,m,j}$, so that $H_j=\Psi_j(H)$ and $(\Psi_j)|_G$ is 
an irreducible Weil representation of $G = \SU_n(q)$. Note that the only automorphisms of $G$ that preserve the equivalence class of
each $(\Psi_j)|_G$ are the inner-diagonal automorphisms, i.e. the ones induced by elements in $\GU_n(q)$ (via conjugation). It follows
that we can find an element $h \in \GU_n(q) \leq L$ 
(with $L=\Sp_{2n}(q)$ as in the proof of Theorem \ref{main-su1}) such that 
$g$ and $h$ induce the same automorphism of $G$. Changing $g$ to another representative in its coset $gG$, we 
can make sure that 
\begin{equation}\label{eq:hs11a}
  h = \diag(\rho,1,1, \ldots,1) 
\end{equation}  
for some $\rho \in \mu_{q+1} \leq \F_{q^2}^\times$. In particular, 
\begin{equation}\label{eq:hs12}
  h^{q+1}=1,
\end{equation}  
and $\Psi_j(g)\Psi_j(h)^{-1}$ centralizes $\Psi_j(G)$, whence
\begin{equation}\label{eq:hs13}
  \Psi_j(g)=\alpha_j \Psi_j(h)
\end{equation}
for some $\alpha_j \in \C^\times$. In fact, $\alpha_j$ is a root of unity because both $g$ and $h$ have finite order.

Recall by \cite[(3.1.2)]{KT3} that $\Tr(\Psi_j(h)) \in \Q(\zeta_{q+1})$. On the other hand, since $\nu$ is chosen to have order
$n_0(q+1)$, $\Tr(\Psi_j(g)) \in \Q(\zeta_{n_0(q+1)})$ by Lemma \ref{tracesinchi1}. Hence the root of unity $\alpha_j$ belongs 
to $\Q(\zeta_{n_0(q+1)})$, and so, as $2|(q+1)$, we have that 
\begin{equation}\label{eq:hs14}
  \alpha_j^{n_0(q+1)}=1 
\end{equation}
for all $j$. Together with 
\eqref{eq:hs12} and \eqref{eq:hs13}, this implies that $\Psi_j(g)^{n_0(q+1)}=\Id$ for all $j$, whence $\Phi(g)^{n_0(q+1)}=\Id$
and $g^{n_0(q+1)}=1$ by faithfulness of $\Phi$. Coupled with \eqref{eq:hs11}, we deduce that $|H/G|$ divides $n_0(q+1)$.
But $|H/G|$ divides $A$ and $\gcd(A,n_0(q+1))=n_0$ by \eqref{eq:cp11}. Consequently, $|H/G|$ divides $n_0$. Applying 
$\Psi_j$, we also get that $|H_j/H_j^{(\infty)}|$ divides $n_0$.

Next we show that 
\begin{equation}\label{eq:hs15}
  \CB_H(G)=\ZB(H),~~H/\ZB(H) \cong \PGU_n(q).
\end{equation}  
Indeed, note that $\CB_H(G)$ acts via scalars in each $\Psi_j$ and so centralizes
$\Psi_j(H)$, whence $\CB_H(G) = \ZB(H)$. We already showed that $H/\CB_H(G)$ embeds in $\PGU_n(q)$ and contains 
$\PSU_n(q) = G/\ZB(G)$. If $H/\ZB(H) < \PGU_n(q)$, then applying $\Psi_j$ and using $\Psi_j(\ZB(H)) \leq \ZB(H_j)$, we 
would have that $H_j/\ZB(H_j)$ is properly contained in $\PGU_n(q)$, contradicting \eqref{eq:hs10}. 

\smallskip
(iii) The relation \eqref{eq:hs15} shows that $H$ induces the full subgroup $\PGU_n(q)$ of inner-diagonal automorphisms of $G$. As 
$H = \langle G,g \rangle$, see \eqref{eq:hs11}, we may therefore assume that for the element $h = \diag(\rho,1, \ldots,1)$ in 
\eqref{eq:hs11a} we have $\rho \in \F_{q^2}^\times$ is of order $q+1$. Write 
\begin{equation}\label{eq:hs16}
  d:=\gcd(n,q+1) = an-b(q+1),~~h^d = (\rho^a \cdot 1_W)h' \mbox{ with }h' := \diag(\rho^{d-a},\rho^{-a}, \ldots,\rho^{-a}), 
\end{equation}  
for some $a,b \in \Z$. Then $\rho^d = \rho^{an-b(q+1)} = \rho^{an}$, hence $\det(h') = \rho^{d-an}=1$, i.e. $h' \in \SU_n(q)$. 

We will now fix $j:=(q+3)/2$ in \eqref{eq:hs13} and let $\al:=\al_j$. By the proof of \cite[Lemma 3.2]{KT3}, this choice of $j$ (and 
the fact that $\rho$ has order $q+1$) ensures 
that $\det(\Psi_j(h))$ is a primitive $(q+1)^{\mathrm {th}}$ root $\zeta_{q+1}$ of unity. On the other hand, 
by Lemma \ref{detdown1} and \eqref{eq:hs13}, 
$$1=\det(\Psi_j(g)) = \al^A\det\Psi_j(h) = \al^A\zeta_{q+1}.$$
Recalling by \eqref{eq:hs14} that $\al^{n_0(q+1)}=1$, we can write 
$\al=\zeta_{n_0(q+1)}^c$ for a primitive $(n_0(q+1))^{\mathrm {th}}$ root $\zeta_{n_0(q+1)}$ of unity with 
$\zeta_{n_0(q+1)}^{n_0}=\zeta_{q+1}$ and $c\in \Z$. Now 
$\zeta_{q+1}=\al^{-A} = \zeta_{q+1}^{-(A/n_0)c}$ has order $q+1$, and so $\gcd(c,q+1)=1$. As $n_0|(q+1)^n$, this implies
that 
\begin{equation}\label{eq:hs17}
  \gcd(c,n_0(q+1))=1, \mbox{ i.e. }\al =\zeta_{n_0(q+1)}^c\mbox{ has order exactly }n_0(q+1).
\end{equation}   
Also write 
$$n_0=de,~~q+1=dr$$
with $e,r \in \Z_{\geq 1}$. 

Recall we have shown that $|H/G|$ divides $n_0=de$, and $H$ induces the subgroup $\PGU_n(q)$ of $\Aut(G)$, whereas 
$G$ induces the subgroup $\PSU_n(q)$ of order $|\PGU_n(q)|/d$ of $\Aut(G)$. It follows that $|H/G| = ds$ for some divisor $s$ of $e$. 
In particular, $g^{ds} \in G$, whence using \eqref{eq:hs16} we obtain that
$$\Psi_j(g^{ds}) = \al^{ds}\Psi_j(h^{ds}) = \al^{ds}\Psi_j\bigl( (\rho^a \cdot 1_W)^s \bigr) \Psi_j\bigl( (h')^s \bigr)$$
belongs to $\Psi_j(G)$. As $h' \in G$ and $\Psi_j(\rho \cdot 1_W) = \zeta_{q+1} \cdot \Id$ by \cite[(3.2.1)]{KT3}, this implies that 
the scalar transformation $\al^{ds}\zeta_{q+1}^{as} \cdot \Id$ belongs to $\Psi_j(G)$. As the quasisimple group $G=\SU_n(q)$ acts 
irreducibly in $\Psi_j$ and has center of order $d$, this scalar transformation has order dividing $d$, that is,
\begin{equation}\label{eq:hs18}
  \bigl( \al^{d^2}\zeta_{q+1}^{ad})^s=1.
\end{equation}   
Now, $\al^{d^2} = \zeta_{n_0(q+1)}^{cd^2} = \zeta_{er}^c$ by \eqref{eq:hs17}, if we take $\zeta_{er}:=\zeta_{n_0(q+1)}^{d^2}$. Next,
$\zeta_{q+1}^{ad} = \zeta_{n_0(q+1)}^{n_0ad}=\zeta_{er}^{ae}$. It follows that $\al^{d^2}\zeta_{q+1}^{ad} = \zeta_{er}^{c-ae}$, and 
so \eqref{eq:hs18} implies that $er$ divides $s(c-ae)$; in particular, $e$ divides $sc$. But $c$ is coprime to $n_0=de$ by 
\eqref{eq:hs17}, hence $e|s$. Consequently, $s=e$, i.e. $|H/G|=n_0$, as stated in (a).

\smallskip
(iv) Now we note that, since all prime divisors of the odd integer $n_0$ divide $q+1$ and $\nu$ has order $n_0(q+1)$, 
$k_0 = \F_{q^2}(\nu)$ equals $\F_{q^{2n_0}}$. [Indeed, if $\ell$ is any (odd) prime divisor of $n_0$ and $\ord_\ell(n_0)=c >0$, then,
as in the proof of Lemma \ref{choice}, we have that $\ell^c|(q^{2a}-1)/(q+1)$ if and only if $\ell^c|a$. Proceeding $\ell$ by $\ell$, we get that $n_0|(q^{2a}-1)/(q+1)$, i.e. $n_0(q+1) \mid (q^{2a}-1)$, if and only if $n_0|a$, and thus $ \F_{q^2}(\nu)= \F_{q^{2n_0}}$.]
To determine $\tH:=H^{n,m}_{\ari,k}$, we recall that $\sW^{n,m}$ is the $[A]^\star$ Kummer pullback of $\sH^{n,m}$, hence 
$\tG:=G^{n,m}_{\ari,k} = C_{\ari,k} \times G^{n,m}_{\geo}$ is a subgroup in $H^{n,m}_{\ari,k}$, with cyclic quotient of order dividing $A$. 
At the same time,
$H^{n,m}_{\ari,k}$ contains $H=H^{n,m}_{\geo}$ as a normal subgroup, also with cyclic quotient, and with 
$H^{(\infty)} = \tG^{(\infty)}=G^{n,m}_{\geo} \cong \SU_n(q)$. 
It follows that $G^{n,m}_{\geo} \lhd H^{n,m}_{\ari,k}$, 
whence 
\begin{equation}\label{eq:hs19}
  G^{n,m}_{\geo} \lhd H^{n,m}_{\ari,k} \geq H^{n,m}_{\geo}, 
  \mbox{ and }[H^{n,m}_{\ari,k}:G^{n,m}_{\ari,k}] \mbox{ divides } A.
\end{equation}
Recall by \eqref{eq:hs15} that $H^{n,m}_{\geo}$ induces the subgroup $\PGU_n(q)$ of all inner-diagonal 
automorphisms of $G^{n,m}_{\geo}=\SU_n(q)$.
Again, since only inner-diagonal automorphisms of $\SU_n(q)$ can fix the equivalence of each irreducible Weil representations $(\Psi_j)|_G$ of
$\SU_n(q)$, $H^{n,m}_{\ari,k}$ must induce the same subgroup $\PGU_n(q)$ while acting on $G^{n,m}_{\geo}$. In particular,
\begin{equation}\label{eq:hs19a}
  |H^{n,m}_{\ari,k}| = |\PGU_n(q)| \cdot |\tC| = |\SU_n(q)| \cdot |\tC| = |G^{n,m}_{\ari,k}/C_{\ari,k}| \cdot |\tC|,
\end{equation}
where $\tC:=\CB_{H^{n,m}_{\ari,k}}(G^{n,m}_{\geo})$. Together with \eqref{eq:hs19}, this implies that
\begin{equation}\label{eq:hs20}
  |\tC| \mbox{ divides }|C_{\ari,k}| \cdot A,
\end{equation}

Consider any $c \in \tC$. For any $0 \leq j \leq q$, 
$\Psi_j(c)$ centralizes the irreducible subgroup $\Psi_j(G)$, hence $\Psi_j(c) = \gamma_j \cdot \Id$
for some root of unity $\gamma \in \C^\times$. Just as above, we see that the field $E_\nu = \F_p(\mu_{n_0(q+1)(p-1)})$  of Lemma \ref{tracesinchi1} is equal to $k_0=\F_{q^{2n_0}}$. Hence, by Lemma \ref{tracesinchi1}, $\gamma_j \deg(\Psi_j) = \Tr(\Psi_j(c))$ belongs to 
$\Q(\nu)=\Q(\zeta_{n_0(q+1)})$. As $2|(q+1)$ and $\gamma_j$ is a root of unity, we conclude that $\gamma_j^{n_0(q+1)}=1$, and 
so $c^{n_0(q+1)}=1$ for all $c \in \tC$, i.e. the exponent of $\tC$ divides $n_0(q+1)$. 
On the other hand, as $\tC$ acts via scalars in all $\Psi_j$, it is a (finite) abelian group. Thus $|\tC|$ divides $(n_0(q+1))^{q+1}$. Applying
\eqref{eq:hs20} and \eqref{eq:cp11}, we now obtain that $|\tC|$ divides 
$$\gcd\bigl( |C_{\ari,k}| \cdot A, (n_0(q+1))^{q+1} \bigr) = 
    |C_{\ari,k}| \cdot n_0 \cdot \gcd\biggl(\frac{A}{n_0},\frac{n_0^q(q+1)^{q+1}}{|C_{\ari,k}|} \biggr) = |C_{\ari,k}| \cdot n_0.$$
Together with \eqref{eq:hs19a}, this implies that $|H^{n,m}_{\ari,k}/G^{n,m}_{\ari,k}| = n_0/e$ for some odd integer $e|n_0$; in particular,
$|H^{n,m}_{\ari,k}|=|C_{\ari,k}| \cdot |\SU_n(q)| \cdot (n_0/e) = \bigl(|C_{\ari,k}|/e\bigr) \cdot |H^{n,m}_{\geo}|$. But $H^{n,m}_{\geo}$ is 
a subgroup of $H^{n,m}_{\ari,k}$ and $|C_{\ari,k}| \leq 2$, so we conclude that $e=1$.

Now, if $C_{\ari,k}=1$, then $H^{n,m}_{\ari,k}=H^{n,m}_{\geo}$. Assume $C_{\ari,k}= \langle \bj \rangle$. Then 
$2=|H^{n,m}_{\ari,k}/H^{n,m}_{\geo}|$. Recall by (a) that $H^{n,m}_{\geo}$ is an extension of 
the quasisimple subgroup $G=\SU_n(q)$ of odd index $n_0$. On the other hand, by Theorem \ref{main-su1}, 
$G^{n,m}_{\ari,k} = C_{\ari,k} \times \SU_n(q)$, and the order $2$ subgroup $C_{\ari,k} = \langle \bj \rangle \leq \ZB(\GU_n(q))$
acts via scalars in each of $\Psi_j$, hence it centralizes $H^{n,m}_{\ari.k}$. It follows that $C_{\ari,k} \cap H^{n,m}_{\geo}=1$ and 
$H^{n,m}_{\ari,k} = C_{\ari,k} \times H^{n,m}_{\geo}$. 

Finally, by Theorem \ref{main-su1}, $|C_{\ari,k}|=2$ if and only if $2 \nmid \deg(k/\F_{q^2})$ if and only if 
$2 \nmid \deg(k/k_0)$, since $\deg(k_0/\F_{q^2})=n_0$ is odd.    
\end{proof}

\begin{rmk}In the special case where $n=3$, $m=1$, and $3|(q+1)$, Theorem \ref{main-su2} complements \cite[Theorem 19.2]{KT1}.
\end{rmk}
Now we specialize to the case where $\gcd(n,q+1)=1$, and follow Remark \ref{coprime} to choose $a \in \Z$ so that
$aA \equiv 1$ (mod $(q+1)$) and take $\nu =\chi_{q+1}^a$. Then the hypergeometric sheaves $\sH^{n,m,j}$ of 
\eqref{eq:ab21} are defined over $\G_m/\F_{q^2}$, and their sum $\sH^{n,m}=\oplus^q_{j=0}\sH^{n,m,j}$ has trace function
$$u \in k^\times  \mapsto \frac{1}{\Gauss(\psi_k,\chi_2)}\sum_{x \in k}\psi_k\bigl(u^\alpha x^{q^m+1}-u^\beta x^{q^n+1}\bigr),$$
with $\alpha A-\beta B=1$ and $\beta \in (q+1)\Z$, see \eqref{eq:510a}.

\begin{thm}\label{main-su3}
Given the assumption \eqref{eq:mn2}, assume in addition that $\gcd(n,q+1)=1$. Then we have the following results.
\begin{enumerate}[\rm(a)]
\item The geometric monodromy group $H^{n,m}_{\geo}$ of $\sH^{n,m}$ is $\SU_n(q)$ 
acting in its total Weil representation of degree $q^n$.
\item The geometric monodromy group $H^{n,m,j}_{\geo}$ of the hypergeometric sheaf $\sH^{n,m,j}$, $0 \leq j \leq q$, 
is the image of $\SU_n(q)$ in an irreducible Weil representation, of degree $A=(q^n+1)/(q+1)$ if $1 \leq j \leq q$ and $A-1$ if $j=0$.
\item Over any finite extension $k$ of $\F_{q^2}$, the arithmetic monodromy group $H^{n,m}_{\ari,k}$ of $\sH^{n,m}$ is 
equal to the arithmetic monodromy group $G^{n,m}_{\ari,k}$ of $\sW^{n,m}$ in Theorem \ref{main-su1}. 
\end{enumerate}
\end{thm}

\begin{proof}
Note Theorem \ref{main-su3} is the $n_0=1$ case of Theorem \ref{main-su2}. But we will give an alternative
proof, which will later apply to the proof of Theorem \ref{main-su4}.

\smallskip
(i) Consider the local system $\sW^{\, n,m,2}$ on $\G_m \times \A^2/\F_{q^2}$ with trace function given as follows. 
For $k/\F_{q^2}$ a finite extension,  and $v \in k^\times$, $r,s \in k$,
$$(v,r,s) \mapsto  \frac{1}{\Gauss(\psi_k,\chi_2)}\sum_{x \in k}\psi_k\bigl(vx^{q^n+1} + rx^{q^m+1}+sx^{q^2+1}\bigr),$$
with geometric monodromy group $G^{n,m,2}_{\geo}$ and arithmetic monodromy group $G^{n,m,2}_{\ari,k}$ over 
any finite extension $k$ of $\F_{q^2}$. Then the system $\sW^{\, n,m,2}|_{v=1,\ r=0}$ at $(v,r)=(1,0)$ is exactly 
the local system $\sW(\psi_{-2},n,2,q)$ considered in \eqref{eq:sheaf1}. By Theorem \ref{main-sp3},
the geometric monodromy group of $\sW^{\, n,m,2}|_{v=1,\ r=0}$ equals $L:= \Sp_{2n}(q)$ in one of its total Weil representations 
of degree $q^n$. Thus $G^{n,m,2}_{\geo}$ contains $L$. By Proposition \ref{ABmoment}, $\sW^{\, n,m,2}$ is a sum of two 
subsystems of rank $(q^n \pm 1)/2$.  Furthermore, it satisfies the conclusions of Theorem \ref{thm:vdG-vdV}.
Now, applying Theorem \ref{sp-up} to $G^{n,m,2}_{\ari,k} \geq L$ (with $e=1$), 
we obtain that
\begin{equation}\label{eq:mn31}
  L \lhd G^{n,m,2}_{\geo}  \leq G^{n,m,2}_{\ari,k} \leq C \times L,
\end{equation}  
where $C$ a cyclic scalar subgroup, and either $|C| = 1,2$, or $p=3$, $2 \nmid f$, and $|C| =3,6$.

Now, by specializing $\sW^{\, n,m,2}$ to the curve $v = -u^\beta$, $r=u^\alpha$, $s=0$, we obtain from \eqref{eq:mn31} that
the geometric monodromy group $H:=H^{n,m}_{\geo}$ of $\sH^{n,m}$ is contained in $C \times L$.

\smallskip
(ii) By \eqref{eq:510b} (and the fact that the traces are all real-valued), 
the $[A]^\star$ Kummer pullback of $\sH^{n,m}$ is the local system $\sW^{\, n,m}$ which has 
geometric monodromy group $G^{n,m}_{\geo}=\SU_n(q)$ by Theorem \ref{main-su1}. Hence, $G^{n,m}_{\geo}$ is a normal 
subgroup of $H$ with cyclic quotient of order dividing $A=(q^n+1)/(q+1)$, which is coprime to $2p$. It follows that 
$L \geq H^{(\infty)}=G^{n,m}_{\geo}$ and $|H/H^{(\infty)}|$ is coprime to $2p$. But $H/(H \cap L)$ embeds in $CL/L \cong C$, and 
$|C|$ divides $2p$. Hence $\SU_n(q) =G^{n,m}_{\geo} \lhd H =H \cap L \leq L \cong \Sp_{2n}(q)$. 
Furthermore, the action of $H$ on $\sH^{n,m}$ is the sum of 
$q+1$ irreducible representations, one of degree $A-1$ and $q$ of degree $A$. Hence, by \cite[Theorem 3.4]{KT3}, we know
that $\SU_n(q) \lhd H \leq \GU_n(q)$. Recall again that $G^{n,m}_{\geo} = \SU_n(q)$ has 
index dividing $A=(q^n+1)/(q+1)$ which is coprime to $q+1 = |\GU_n(q)/\SU_n(q)|$ since $\gcd(n,q+1)=1$. Consequently, $H = \SU_n(q)$ as stated in (a).
Now (b) follows from (a), since $H^{n,m,j}_{\geo}$ is the image of $H=H^{n,m}_{\geo}$ acting on an individual sheaf $\sH^{n,m,j}$.

\smallskip
(iii) For (c), we note that $H^{n,m}_{\ari,k}$ contains $H$ as a normal subgroup of cyclic index, hence 
$H =H^{(\infty)} = (H^{n,m}_{\ari,k})^{(\infty)}$. The specialization $v = -u^\beta$, $r=u^\alpha$, $s=0$ also shows that 
$H^{n,m}_{\ari,k}$ is contained in $G^{n,m,2}_{\ari,k}$, hence 
$$H \lhd H^{n,m}_{\ari,k} \leq \NB_{C \times L}(H) = C \times (\GU_n(q) \rtimes C_2).$$
Note that $\bigl[\bigl(C \times (\GU_n(q) \rtimes C_2)\bigr):H\bigr]$ divides $4(q+1) \cdot \gcd(p,3)$.
On the other hand, since $\sW^{n,m}$ is the $[A]^\star$ Kummer pullback of $\sH^{n,m}$, 
$G^{n,m}_{\ari,k} \geq G^{n,m}_{\geo}=H$ is subgroup of $H^{n,m}_{\ari,k}$ of index dividing $A=(q^n+1)/(q+1)$ which is coprime to $2p(q+1)$. Thus $[H^{n,m}_{\ari,k}:G^{n,m}_{\ari,k}]$ divides 
$4p(q+1)$ and at the same time is coprime to $2p(q+1)$. Hence $H^{n,m}_{\ari,k}=G^{n,m}_{\ari,k}$.
\end{proof}

Next, we will work with any odd $n \geq 3$ and  any odd $m < n$ that is coprime to $q+1$, e.g. $m=1$. Then we follow Remark \ref{coprimebis} to study the sheaf $\sH^{n,m}_{bis}$ defined in \eqref{eq:511a} over $\G_m/\F_{q^2}$, which   
has trace function
$$u \in k^\times  \mapsto \frac{1}{\Gauss(\psi_k,\chi_2)}\sum_{x \in k}\psi_k\bigl(u^\alpha x^{q^m+1}-u^\beta x^{q^n+1}\bigr),$$
with $\alpha A-\beta B=1$ and $\alpha \in (q+1)\Z$.

\begin{thm}\label{main-su4}
Given the assumption \eqref{eq:mn2}, assume in addition that $\gcd(m,q+1)=1$. Then we have the following results.
\begin{enumerate}[\rm(a)]
\item The geometric monodromy group $H^{n,m}_{bis,\geo}$ of $\sH^{n,m}_{bis}$ is $\GU_n(q)$ 
acting in its total Weil representation of degree $q^n$, with character 
$$\zeta_{n,q}: g \mapsto (-1)^n(-q)^{\dim_{\F_{q^2}}\Ker(g-1)}.$$ 
\item The geometric monodromy group $H^{n,m,j}_{bis,\geo}$ of the $q+1$ summands $\sH^{n,m,j}_{bis}$ of $\sH^{n,m}_{bis}$, 
is the image of $\GU_n(q)$ in $q+1$ irreducible Weil representations, $q$ of degree $A=(q^n+1)/(q+1)$ and $1$ of degree 
$A-1$.
\item For any finite extension $k$ of $\F_{q^2}$, the arithmetic monodromy group of $\sH^{n,m}_{bis}$ over $\G_m/k$ is
$$H^{n,m}_{bis,\ari,k} = C_{\ari,k} \times \GU_n(q),$$ 
with $C_{\ari,k}$ being a cyclic scalar subgroup of order $\leq 2$. In fact, 
if $k/\F_{q^2}$ has even degree, or if  $q \equiv 3 (\bmod\ 4)$, then $C_{\ari,k}=1$. If $q \equiv 1 (\bmod\ 4)$ and
$k/\F_{q^2}$ has odd degree, then $C_{\ari,k} = \langle \bt \rangle \cong C_2$.

\item For the sheaf  $\sW_{bis}^{\, n,m}=[B]^\star\sH^{n,m}_{bis}$, whose trace function, cf. \eqref{eq:511bb}, is
$$u \in E^\times \mapsto  \frac{1}{\Gauss(\psi_E,\chi_2)}\sum_{z \in E}\psi_E( z^{q^m+1}-u^{-1}z^{q^n+1}),$$
its geometric monodromy group $G^{n,m}_{bis,\geo}$ is $\GU_n(q)$ 
acting in its total Weil representation of degree $q^n$. Furthermore, for any finite extension $k$ of $\F_{q^2}$, the arithmetic monodromy group $G^{n,m}_{bis,\ari,k}$ of $\sH^{n,m}_{bis}$ is equal to $H^{n,m}_{bis,\ari,k}$.
\end{enumerate}
\end{thm}

\begin{proof}
(i) Using the same notation and the arguments in the proof of Corollary \ref{main-su3}, by specializing $\sW^{\, n,m,2}$ to the curve 
$v = -u^\beta$, $r=u^\alpha$, $s=0$, we again have that
the geometric monodromy group $H:=H^{n,m}_{bis,\geo}$ of $\sH^{n,m}_{bis}$ is contained in $C \times L$, with $C$ a cyclic scalar subgroup
of order dividing $2 \cdot \gcd(p,3)$. Defining 
$L_2:= \OB_2(C) \times L$, note that $L_2$ is a normal subgroup of $CL$ of index $1$ or $p$.

Next, note that the Kummer pullback 
$$\mathcal{K} =[q^m+1]^\star\sW^{\, n,m}$$ 
of $\sW^{\, n,m}$, has trace function at $r \in k^\times$
$$r \mapsto  \frac{1}{\Gauss(\overline\psi_k,\chi_2)}  \sum_{x \in k}\overline\psi_{k}\bigl(x^{q^n+1}-(rx)^{q^m+1}\bigr) = 
    \frac{1}{\Gauss(\psi_k,\chi_2)}\sum_{x \in k}\psi_{k}\bigl(x^{q^m+1}-(r^{-1}x)^{q^n+1}\bigr)$$
on $\G_m/k$ (since $-1$ is a square in $k \supseteq \F_{q^2}$ and all traces are integers).
On the other hand, by \eqref{eq:511bb}, if we define 
$$\sK':=[q^n+1]^\star\sW^{\, n,m}_{bis}=[q^n+1]^\star[B]^\star\sH^{n,m}_{bis},$$
then $\sK'$ has trace function at $u \in k^\times$ 
$$u \mapsto \frac{1}{\Gauss(\psi_k,\chi_2)}\sum_{x \in k}\psi_{k}\bigl(x^{q^m+1}-(u^{-1}x)^{q^n+1}\bigr).$$
Thus $\mathcal{K}'$ is arithmetically isomorphic to $\mathcal{K}$, because they have equal trace functions and are each arithmetically semisimple. So their geometric monodromy groups are the same: $K_{\geo} =K'_{\geo}$.

\smallskip
(ii) The aforementioned pullback relationships imply that $K_{\geo}$ is a normal 
subgroup of $G^{n,m}_{\geo}$ with cyclic quotient of order dividing $q^m+1$. It follows from Theorem 
\ref{main-su1} that $K_{\geo} \cong \SU_n(q)$, whence $K'_{\geo} \cong \SU_n(q)$. 
Next, $K'_{\geo}$ is a normal subgroup of $H=H^{n,m}_{bis,\geo}$ with cyclic quotient of order 
dividing $AB(q+1)$ which is coprime to $p$. But $H/(H \cap L_2)$ embeds in $CL/L_2$, and 
$|CL/L_2|$ divides $p$. Hence 
$$\SU_n(q) = K'_{\geo} = H^{(\infty)} \lhd H \cap L_2 = H   \leq L_2=\OB_2(C) \times L \leq C_2 \times \Sp_{2n}(q).$$ 
Furthermore, the action of $H$ on $\sH^{n,m}_{bis}$ is the sum of 
$q+1$ irreducible representations, one of degree $A-1$ and $q$ of degree $A$, and these representations remain irreducible upon 
restriction to $K'_{\geo} \leq H \cap L$. Hence, by \cite[Theorem 3.4]{KT3} applied to $H \cap L$, we obtain that 
$$\SU_n(q) \lhd H \cap L \leq \GU_n(q);$$
in particular, $[H \cap L:K'_{\geo}]$ divides $q+1$. 
As $H/(H \cap L)$ embeds in $L_2/L$ which has order $1$ or $2$,
we see that $|H/K'_{\geo}|$ divides $2(q+1)$. At the same time, $|H/K'_{\geo}|$ divides $AB(q+1)$, an odd multiple of $q+1$. 
It follows that $|H/K'_{\geo}|$ divides $q+1$. 

Choosing $\chi$ of order $q+1$ and using $B=(q^m+1)/(q+1)$ is coprime to $q+1$, by Lemma \ref{detdown2} we see that 
$\sH^\sharp_{big,A,\chi,B,descent}$ has geometric determinant $\sL_\chi$. Hence,
$H/H^{(\infty)}$ has order divisible by $q+1$. Since $K'_{\geo}=H^{(\infty)}$, we
have shown that 
\begin{equation}\label{eq:511c}
  H/K'_{\geo} \cong C_{q+1}; 
\end{equation}  
in particular, $|H| = |\GU_n(q)|$.

\smallskip
(iii) Next, we claim that in fact $\OB_2(C) = C_2 = \langle \bt \rangle$ and $H \neq H \cap L$. Assume the contrary:
$\OB_2(C) = 1$ or $H = H \cap L$. Then $H \leq L$ acts 
on $\sH^{n,m}_{bis}$ via restricting a total Weil representation $\Phi$ of $L \cong \Sp_{2n}(q)$ to $H$. By \cite[Lemma 3.2(iii)]{KT3},
the image of $H$ on one of the irreducible summands of rank $A$ of $\sH^{n,m}_{bis}$ has trivial determinant, which 
is impossible (since the only summand of $\sH^{n,m}_{bis}$ that has trivial geometric determinant has rank $A-1$).
 
As shown on \cite[p. 9]{KT3}, ${\mathbf{N}}_L(K'_{\geo}) \cong M \rtimes \langle \sigma \rangle$, with $M \cong \GU_n(q)$ and 
$\sigma \in L$ an involution that acts as inversion on $\ZB(M) \cong C_{q+1}$. It follows that 
$$H \leq {\mathbf{N}}_{L_2}(K'_{\geo}) = (M \rtimes \langle \sigma \rangle) \times \langle \bt \rangle.$$
Now using \eqref{eq:511c}, we can write $H = \langle K'_{\geo},h \rangle$ where $h= \bt^i\sigma^jg$ for some $g \in M$ and 
$i,j \in \{0,1\}$. Note that $i=1$ since $H \neq H \cap L$. On the other hand, if $j=1$, then $h$ does not fix invariant
some of the irreducible Weil representations of $K'_{\geo}$ occurring in $\sH^{n,m}_{bis}$, a contradiction. 
Thus $h=\bt g$ with $g \in M$.

Let $e$ denote the order of the coset $gK'_{\geo}$ as an element in $M/K'_{\geo} \cong C_{q+1}$, 
in particular, $e|(q+1)$. By the choice of $h$, $hK'_{\geo}$ has order 
$q+1$ in $H/K'_{\geo}$. But $h^{2e} = g^{2e} \in K'_{\geo}$, so $(q+1)/2$ divides $e$. We claim that 
\begin{equation}\label{eq:511d}
  e=q+1.
\end{equation}   
Assume the contrary: $e=(q+1)/2$. If $q \equiv 3 (\bmod\ 4)$, then 
$h^{(q+1)/2} = g^{(q+1)/2} \in K'_{\geo}$, a contradiction. Consider the case $q \equiv 1 (\mod 4)$, in particular, $2 \nmid e$,
and an odd prime divisor $r$ of $q+1$. Then some subsheaf $\sH'$ of $\sH^{n,m}_{bis}$ of odd rank $A$ has geometric determinant
$\sL_\chi$, with $\chi$ of order $r$. On the other hand, $\bt$ acts on $\sH^{n,m}_{bis}$ as scalar $-1$, and $g^e \in K'_{\geo} \cong \SU_n(q)$
has trivial determinant on $\sH'$. It follows that $h=\bt g$ has determinant of even order on $\sH'$, a contradiction.

Now, \eqref{eq:511d} implies that $\langle K'_{\geo},g \rangle = M \cong \GU_n(q)$. We also note that the action of 
$h=\bt g$ on $\sH^{n,m}_{bis}$ is $-\Phi(g)$. Since $H = \langle K'_{\geo},h \rangle$, using \cite[Theorem 3.1(i)]{KT3}, it follows that the action of $H$ on
$\sH^{n,m}_{bis}$ affords the total Weil character $\zeta_{n,q}$ and that $H \cong \GU_n(q)$, and the statements (a)  and (b) follow.

\smallskip
(iv) The specialization $v = -u^\beta$, $r=u^\alpha$, $s=0$ of $\sW^{n,m,2}$ at the beginning of (i) also shows that 
$\tH:=H^{n,m}_{bis,\ari,k}$ embeds in $C' \times L$ for some cyclic scalar subgroup $C' \leq C_{2\cdot\gcd(p,3)}$. Recalling
that $H^{n,m}_{bis,\ari,k}$ normalizes the standard subgroup $H^{(\infty)}=(H^{n,m}_{bis,\geo})^{(\infty)} \cong \SU_n(q)$ of 
$L = \Sp_{2n}(q)$ but preserves the equivalence class of each of the $q+1$ irreducible Weil representations of 
$\SU_n(q)$, we obtain
\begin{equation}\label{eq:511e}
  \GU_n(q) \cong H \lhd \tH \leq \NB_{C' \times L}(\SU_n(q)) = C' \times \GU_n(q).
\end{equation}  
In particular, the first statement in (c) follows if $C' \leq C_2$. Consider the case $C' \geq C_3$, whence $p=3$. In this case, \eqref{eq:511e} 
shows that 
\begin{equation}\label{eq:511f}
  \tH/\CB_{\tH}(H^{(\infty)}) \hookrightarrow \PGU_n(q),
\end{equation}  
and that $\CB_{\tH}(H^{(\infty)})$ is contained in  $C' \times \ZB(\GU_n(q))$ which centralizes $\tH$, whence
$\CB_{\tH}(H^{(\infty)})  = \ZB(\tH)$. Arguing as in the proof of \eqref{eq:mn26} and 
using Lemma \ref{center-su}, we also have that $|\ZB(\tH)|$ divides $2(q+1)$. Together with \eqref{eq:511f}, this implies 
that $|\tH/H^{(\infty)}|$ divides $2(q+1)^2$. On the other hand, by \eqref{eq:511e}, 
$\tH/H^{(\infty)}$ embeds in $C' \times C_{q+1}$, an abelian group.   It follows that 
$$\tH/H^{(\infty)} \leq \OB_{3'}(C' \times C_{q+1}) = \OB_{2}(C') \times C_{q+1} = \bigl(\OB_{2}(C') \times \GU_{n}(q) \bigr)/H^{(\infty)},$$
and so $\tH \leq \OB_{2}(C') \times \GU_n(q)$, and the first statement in (c) is proved in full generality.

To determine 
$C_{\ari,\F_{q^2}}$, we note that, since both $n$ and $m$ are odd, \eqref{eq:511b} shows that the trace at $u=1$
of $\sH^{n,m}_{bis}$ over $\F_{q^2}$ is 
\begin{equation}\label{eq:511g}
  q^2/\Gauss(\psi_{\F_{q^2}},\chi_2) = (-1)^{(q+1)/2}q, 
\end{equation}
since the Gauss sum
$\Gauss(\psi_{\F_{q^2}},\chi_2)$ is $(-1)^{(q+1)/2}q$ by Stickelberger's formula \eqref{stick}. By (a), the only trace of elements 
in $H=H^{n,m}_{bis,\geo}$ on $\sH^{n,m}_{bis}$ with absolute value $q$ is $(-1)^n(-q)=q$.

Assume that $q \equiv 1 (\bmod\ 4)$. As \eqref{eq:511g} gives trace $-q$, we must have that $C_{\ari,\F_{q^2}}=C_2$. Thus 
$[H^{n,m}_{bis,\ari,\F_{q^2}}:H^{n,m}_{bis,\geo}]=2$, and so $|C_{\ari,k}|=[H^{n,m}_{bis,\ari,k}:H^{n,m}_{bis,\geo}]$ is $2$ if 
$2 \nmid \deg(k/\F_{q^2})$ and $1$ if $2|\deg(k/\F_{q^2})$.

Next assume that $q \equiv 3 (\bmod\ 4)$ but $C_{\ari,\F_{q^2}}=C_2$, in particular, $[H^{n,m}_{bis,\ari,\F_{q^2}}:H^{n,m}_{bis,\geo}]=2$.
It follows that the traces of any elements $v$ with $\F_{q^2}(v)$ of odd degree over $\F_{q^2}$ should be $(-1)$ times the traces of elements 
in $H=H^{n,m}_{\geo}$. On the other hand, \eqref{eq:511g} gives trace at $u=1$ to be $q$, a contradiction. 
Thus $C_{\ari,\F_{q^2}}=1$ when $q \equiv 3 (\bmod\ 4)$. Furthermore, $C_{\ari,k}=1$ for any extension $k/\F_{q^2}$, simply because
$H^{n,m}_{bis,\geo} \leq H^{n,m}_{bis,\ari,k} \leq H^{n,m}_{bis,\ari,\F_{q^2}}=H^{n,m}_{bis,\geo}$.

\smallskip
For (d), recall that $\gcd(B,q+1)=1$, and thus $G^{n,m}_{bis,\geo} \rhd \SU_n(q)$ is a 
normal subgroup of $H^{n,m}_{bis,\geo}=\GU_n(q)$ of index dividing $B$, which is prime to $q+1$, so must itself be $\GU_n(q)$. Now,
$G^{n,m}_{bis,\ari,k}$ contains $G^{n,m}_{bis,\geo} = \GU_n(q)$ and has index dividing $B$, which is odd, in 
$H^{n,m}_{bis,\ari,k} = C_{\ari,k} \times \GU_n(q)$ with $C_{\ari,k} \leq C_2$. Hence $G^{n,m}_{bis,\ari,k} = H^{n,m}_{bis,\ari,k}$.
\end{proof}

\begin{rmk}\label{striking}It is striking that when $\gcd(m,q+1)=1$, the local systems $\sW^{\, n,m}$ and $\sW^{\, n,m}_{bis}$
have trace functions that differ ``only" in which power of $z$ has the parameter, yet the first has geometric monodromy group $\SU_n(q)$ while the second has geometric monodromy group $\GU_n(q)$. 

As a word of caution, we also mention that the subgroup 
$\langle \bj \rangle \times \SU_n(q) \leq \GU_n(q)$ in Theorem \ref{main-su1} is contained in 
a subgroup $\GU_n(q)$ of $\Sp_{2n}(q)$, which acts on a total Weil representation of $\Sp_{2n}(q)$ via the character $\tilde\chi_2\zeta_{n,q}$, where $\tilde\chi_2$ is the unique quadratic character of $\GU_n(q)$, cf. \cite[Theorem 3.1]{KT3}. 
In contrast, the subgroup $\GU_n(q)$ in Theorem 
\ref{main-su4} is {\bf not} contained in $\Sp_{2n}(q)$, and acts on a total Weil representation via the character $\zeta_{n,q}$.
\end{rmk}

%


%
%
%
%
%

\end{document}